\documentclass[psamsfonts]{amsart}

\usepackage{amssymb,amsfonts}
\usepackage[all,arc]{xy}
\usepackage{enumerate}
\usepackage{mathrsfs}
\usepackage{tikz-cd}
\usepackage{enumitem}
\usepackage{empheq}
\usepackage{graphicx}

\newtheorem{theorem}{Theorem}[section]
\newtheorem{corollary}[theorem]{Corollary}
\newtheorem{proposition}[theorem]{Proposition}
\newtheorem{lemma}[theorem]{Lemma}
\newtheorem{condition}[theorem]{Condition}

\newtheorem{assumption}[theorem]{Assumption}

\newtheorem{claim}[theorem]{Claim}
\newtheorem*{theorem*}{Theorem}
\newtheorem*{corollary*}{Corollary}
\newtheorem*{notation*}{Notation}
\newtheorem*{definition*}{Definition}
\newtheorem*{proposition*}{Proposition}
\newtheorem*{definition-proposition*}{Definition-Proposition}

\theoremstyle{definition}
\newtheorem{definition-proposition}[theorem]{Definition-Propostion}
\newtheorem{definition-corollary}[theorem]{Definition-Corollary}
\newtheorem{definition}[theorem]{Definition}

\newtheorem{example}[theorem]{Example}

\newtheorem{notation}[theorem]{Notation}

\newtheorem{remark}[theorem]{Remark}

\theoremstyle{remark}

\newenvironment{proof-sketch}{\noindent{\it{Sketch of proof.}}\hspace*{1em}}{\qed\bigskip}

\title{Morse theory with homotopy coherent diagrams}

\author{Taesu Kim}

\address{Center for Geometry and Physics, Institute for Basic Science (IBS), Pohang, 37673 Korea}

\email{tkim@ibs.re.kr}

\begin{document}

\maketitle

\begin{abstract}
We study Morse theory on noncompact manifolds equipped with exhaustions by compact pieces, defining the Morse homology of a pair which consists of the manifold and related geometric/homotopy data. We construct a collection of Morse data parametrized by cubes of arbitrary dimensions. From this collection, we obtain a family of linear maps subject to some coherency conditions, which can be packaged into a homotopy coherent diagram. We introduce a chain complex which is a colimit for the diagram and show that it computes the Morse homology.
\end{abstract}

\tableofcontents

\section{Introduction}

Morse theory studies a Riemannian manifold in terms of gradient flow lines of a smooth function, associating to it an invariant called Morse homology. While the theory was first formulated for compact manifolds, attempts to study the noncompact setting also have been made by several people (c.f. [CF], [Kan]). A major difference is that the noncompact theory lacks an invariance property for the homology: critical points may escape to infinity when we homotope the geometric data, e.g., Morse functions and metrics. In other words, Morse homology for noncompact manifolds is not a stable notion under homotopical changes. 

This paper introduces a different approach for a noncompact version of Morse theory, in which we exhaust a manifold with compact pieces and put additional data on them. More precisely, we understand the given manifold as a union of compact submanifolds equipped with Morse functions and Riemannian metrics.

\begin{equation}\nonumber
W = \bigcup\limits_{a \in \mathbb{Z}_{\geq 0}} M_a,
\end{equation}
where $M_a \subset W$ is a compact submanifold for each $a$ and the pair $(h_a, g_a)$ (called \textit{Morse datum}) consists of a Morse function and a Riemannian metric on $M_a$. We call the data $\mathcal{E} = \bigl\{(M_a, h_a, g_a)\bigr\}_{a \in \mathbb{Z}_{\geq 0}}$ an {\textit{exhaustion}} of $W$ provided that they satisfy several axioms. Most importantly, we put the following restriction on the data:

\begin{equation}\nonumber
- \nabla_{g_a} h_{a} |_{\partial M_a} \text{ is in the inward direction for each } a.
\end{equation}
This is to ensure that gradient trajectories do not touch the boundary of each compact piece, so that no critical points of the Morse function appear outside of a compact region. By virtue of this condition, we can apply the usual Morse theory for compact manifolds to have a family of chain complexes $\bigl\{MC_*(M_a, h_a, g_a)\bigr\}_{a \in \mathbb{Z}_{\geq 0}}.$

The next step is to relate the Morse theories of individual submanifolds in order to establish a global invariant. For this, we further equip the system with time-dependent Morse data. Including the exhaustion and the gluing parameters, we call the totality of such information \textit{1-homotopy data} and denote it by $\mathfrak{H}^1.$ The standard theory tells us that these data give rise to a direct system of Morse homologies $\bigl\{MH_*(M_a, h_a, g_a)\bigr\}_{a \in \mathbb{Z}_{\geq 0}},$ which in fact is a restatement of the fact that, on homologies, continuation maps between two compact piece of data are independent of the choice of homotopies.

\begin{definition-proposition*}
We define the \textit{Morse homology} of a pair $(W, \mathfrak{H}^1)$ by\footnote{This approach was motivated by symplectic homology theory for linear-at-infinity Hamiltonians, where the homology is defined by the colimit of a direct system (See [Abo]). In particular, our inward direction condition corresponds to the convexity condition put on Floer data.}
\begin{equation}\nonumber
MH_*(W, \mathfrak{H}^1) := \lim\limits_{\longrightarrow} MH_*(M_a, h_a, g_a),
\end{equation}
When we fix the exhaustion $\mathcal{E},$ it is independent of the choices of $\mathfrak{H}^1$ with $\mathcal{E} \subset \mathfrak{H}^1$ up to isomorphism.
\end{definition-proposition*}

Our approach is different from those of [CF] and [Kan] in that we focus considerably on what an appropriate chain-level theory should look like. For instance, we do not discard dependencies of Morse functions (and their homotopies). We will try to reach the most general situation, keeping track of all those various choices in a coherent way by studying a chain-level construction of the above type of Morse homology. For this purpose, we need more information, which we call \textit{higher homotopies} denoted by $\mathfrak{H}.$ Namely, they consist of an exhaustion and Morse data with parametrizations from cubes of arbitrary dimensions. We remark that the data $\mathfrak{H}$ are to extend the given 1-homotopy $\mathfrak{H}^1.$ To organize them systematically, we index them with the nerves $N(\mathcal{I}),$  where  $\mathcal{I}$ is the poset category. In other words, we consider a collection of pairs

\begin{equation}\nonumber
\begin{cases}
H_{\sigma_k} : [0,1]^{k-1} \times \mathbb{R} \times M_{t \sigma_k} \rightarrow \mathbb{R},\\
G_{\sigma_k} :  [0,1]^{k-1} \times {\mathbb{R}} \rightarrow Met(M_{t \sigma_k}), \ \sigma_k \in N(\mathcal{I})_k, \ k \geq 1
\end{cases}
\end{equation}
that satisfy several axioms. (Notation : $\sigma_k$ is given by $k$-many morphisms, and $s \sigma_k$ and $t\sigma_k$ mean the source of the first one and the target of the last one, respectively. We denote $|\sigma_k| = k.$ See Notation \ref{not12p}.) For instance, we require $ H_{\sigma_k}(\vec{t},s,\cdot) = h_{s \sigma_k}$ (and $h_{t \sigma_k}$) for all $s$ sufficiently small (and large, respectively). Using this collection, we can study the \textit{parametrized} negative gradient flow equation for a smooth map $u : \overline{\mathbb{R}} \rightarrow M_{t \sigma_k}$, where $\overline{\mathbb{R}} = \mathbb{R} \cup \{\pm \infty\}.$
\begin{equation}\label{parameqn2}
 \dot{u} + \frac{\nabla_{G_{\sigma_k}(\vec{t}, \cdot)}H_{\sigma_k}({\vec{t}}, \cdot)}{\sqrt{ 1+ |\dot{H}_{\sigma_k}(\vec{t}, \cdot)|^2 |\nabla_{G_{\sigma_k}(\vec{t}, \cdot)}H_{\sigma_k}({\vec{t}}, \cdot)|^2 }} \circ u = 0, \ \ \vec{t} \in [0,1]^{k-1}.
\end{equation}
Let $x$ and $y$ be critical points of the Morse functions $h_{s \sigma_k}$ and $h_{t \sigma_k},$ respectively. The set of all such pairs $(\vec{t}, u)$ that satisfy (\ref{parameqn2}), $u(-\infty) = x,$ and $u(\infty) = y$ is called \textit{parametrized moduli space of trajectories}. We denote it by $\mathcal{M}(H_{\sigma_k}; x,y).$

\begin{theorem*}
For a given homotopy of Morse functions $H_{\sigma_k},$ there is a generic choice of homotopy of Riemannian metrics $ G_{\sigma_k}$ such that the parametrized moduli space $\mathcal{M}(H_{\sigma_k}; x,y)$ is a smooth manifold (with corners) of dimension $|x| - |y|+ k-2.$
\end{theorem*}

To prove this theorem, we will show the Fredholm property of the linearized operator for the equation (\ref{parameqn2}) and use the Banach space implicit function theorem. In fact, we will see that most of the known results (such as transversality) in the standard low-degree theory can be extended to our general situation. In this regard, our results will be presented as a direct generalization of those of [AD] and [Sch]. For example, it is possible to compactify $\mathcal{M}(H_{\sigma_k}; x,y)$ by adding broken trajectories, and counting the orders of zero-dimensional compactified spaces. This leads to a family of linear maps:

\begin{equation}\nonumber
\varphi_{\sigma} : MC_*(M_{s \sigma}, h_{s \sigma}, g_{s \sigma}) \rightarrow MC_*(M_{t \sigma}, h_{t \sigma}, g_{t \sigma}),
\end{equation}
indexed by simplices $\sigma \in N(\mathcal{I}).$ These maps are supposed to satisfy the following relations:
\begin{equation}\label{introrel}
\partial \circ \varphi_{\sigma} + \varphi_{\sigma} \circ \partial + \sum^{|\sigma|}_{i=1} \varphi_{\partial_i \sigma} + \sum^{|\sigma|}_{i=1} \varphi_{(\sigma)^2_{|\sigma|-i}} \circ \varphi_{(\sigma)^1_i} = 0.
\end{equation}
Here $\partial_i \sigma$ means the $(|\sigma|-1)$-simplex obtained by composing the $i$-th and $(i+1)$-th morphisms that appear in $\sigma,$ while $(\sigma)^1_i$ stands for the $i$-simplex that consists of the first $i$-many morphisms. $(\sigma)^2_{|\sigma|-i}$ is given in a similar way. Achieving (\ref{introrel}) involves putting compatibilities among the parametrized Morse homotopy data; such data will be constructed by inductions. This is what is meant by {\textit{coherence}} of the homotopy data.  To realize our goal, we will need to take advantage of \textit{transversal homotopies,} which prevents unwanted solutions from appearing when we homotope the parameter spaces.

Among all of such analytic methods to study the moduli spaces with, a special emphasis will be put on the {\textit{parametrized gluing theorem}} of trajectories. This is because the compositions between the above linear maps play an important role when we try to obtain the relation (\ref{introrel}), and a proper description of the gluing is essential for this purpose. Let $\sigma_k$ and $\sigma_l$ be $k$- and $l$-simplices in $N(\mathcal{I}),$ respectively with $t \sigma_k =s =\sigma_l.$ When three critical points $x \in Crit(h_{s \sigma_k}), y \in Crit(h_{t \sigma_k}),$ and $z \in Crit(h_{t \sigma_l})$ satisfy the conditions $|y| - |x| = k-2$ and $|z| - |x| = l-2,$ the compactification results say that the moduli spaces $\mathcal{M}(H_{\sigma_k}; x,y)$ and $\mathcal{M}(H_{\sigma_l}; y,z)$ can be shown to be finite sets. We then have the gluing theorem for our situation as follows.

\begin{theorem*}
For all sufficiently large $\rho>0,$ we have a bijection of finite sets:
\begin{equation}\nonumber
\mathcal{M}(H_{\sigma_k} \#_{\rho} H_{\sigma_l}; x,z) \simeq  \mathcal{M}(H_{\sigma_k}; x,y) \times \mathcal{M}(H_{\sigma_l}; y,z),
\end{equation}
where $H_{\sigma_k} \#_{\rho} H_{\sigma_l}$ is the concatenation of homotopies for the gluing parameter $\rho$. 
\end{theorem*}

Our tool to investigate the related algebraic structure is $\infty$-category theory, whose objects of study can be thought of as a mixture of (ordinary) categories and topological spaces. It is useful for our purpose in two ways. First, it provides an efficient way to coherently organize higher homotopies. In particular, we can make precise and proper sense of the word \textit{coherence.} Secondly, it enables us to deal with a \textit{homotopy direct system.} In our case, it can be written in the form of a \textit{homotopy coherent diagram:}
\begin{equation}\nonumber
\mathscr{F} : N(\mathcal{I}) \rightarrow N_{dg}\big(Ch(\mathbb{Z}_2)\big).
\end{equation}
We will see the geometric data $\mathfrak{H}$ give rise to an example of such a diagram denoted by $\mathscr{F}_{\mathfrak{H}}.$ Our interest is in computing a model of a homotopy direct limit, as a desired chain-level universal object. We will see that a model for a colimit of $\mathscr{F}_{\mathfrak{H}}$ can be roughly said to be given by the following chain complex
\begin{equation}\nonumber
MC_*(W, \mathfrak{H}) := \bigoplus_{k \geq 0} \bigoplus_{\substack{\sigma_k \in N(\mathcal{I})_k,\\ \text{nondeg.}}} \mathbb{Z}_2 \langle \sigma_k \rangle \otimes MC_*(M_{s\sigma_k},h_{s\sigma_k}, g_{s\sigma_k})
\end{equation}
(See Definition \ref{mccdef} for the differential $\partial$ and the grading $|\cdot|.$)  
\begin{definition*}
\textup{We call the chain complex $\big(MC_*(W, \mathfrak{H}), \partial \big)$ the \textit{Morse chain complex} \textrm{of a pair} $(W, \mathfrak{H})$.}
\end{definition*}

In this paper, we will prove:
\begin{theorem*} For given homotopy data $\mathfrak{H}^{1} \subset \mathfrak{H},$ the following hold.
\begin{enumerate}[label =(\roman*)]
\item $\mathscr{F}_{\mathfrak{H}} $ extends to a colimit diagram $\overline{\mathscr{F}}_{\mathfrak{H}},$ so that its evaluation at the cone point $\overline{\mathscr{F}}_{\mathfrak{H}}(*)$ is given by $MC_*(W, \mathfrak{H}).$
\item $H_*\big(MC_*(W, \mathfrak{H}), \partial \big)$ is isomorphic to $MH_* (W, \mathfrak{H}^1).$
\end{enumerate}
\end{theorem*}

Our previously described situation is depicted by the following diagram.
\begin{equation}\nonumber
\begin{tikzcd}
{}&{}&\substack{\text{Higher homotopy data } \\ \mathfrak{H} \text{ on } W \arrow[dotted]{d}}\\
\substack{\text{1-homotopy data } \\ \mathfrak{H}^1 \text{ on } W } \arrow[hook]{rru} \arrow[dotted]{d} & {} & MC_*(W,\mathfrak{H}) \arrow{d}{\text{Homology}}\\
 MH_* \big(W,\mathfrak{H}^1 \big) & \xrightarrow{\qquad \simeq \qquad} & H_*\big(MC_*(W, \mathfrak{H})\big).
\end{tikzcd}
\end{equation}

Finally, we will try to find a relation between the constructions for two different geometric choices. Namely, for two data $\mathfrak{H}, \mathfrak{H}'$ that share an exhaustion, there is a way to relate the two chain complexes $MC_*(W, \mathfrak{H})$ and $MC_*(W, \mathfrak{H}'),$ provided that there is another diagram ${{\mathscr{F}}^{ext}}$ which extends ${{\mathscr{F}}}_{\mathfrak{H}}$ and ${{\mathscr{F}}}_{\mathfrak{H}'}$ in the sense of Condition \ref{extfuncchn} or Condition \ref{condlast}. 

\begin{proposition*}
For two higher data $\mathfrak{H}$ and $\mathfrak{H}'$ that admit an extension of diagrams ${{\mathscr{F}}^{ext}}$, the two chain complexes $MC_*(W,\mathfrak{H})$ and $MC_*(W,\mathfrak{H}')$  are related by the following diagram:
\begin{equation}\nonumber
\begin{tikzcd}
{} & colim {{\mathscr{F}}^{ext}}(*)  & {}\\
MC_*(W, \mathfrak{H}) \arrow{ru}{\simeq}[swap]{\text{q-isom.}} & {} & MC_*(W, \mathfrak{H}'), \arrow{lu}{\text{q-isom.}}[swap]{\simeq}
\end{tikzcd}
\end{equation}
where  colim ${{\mathscr{F}}^{ext}}(*)$ is the evaluation of its colimit at the cone point $*.$ 
\end{proposition*}

Notice that this verifies our earlier claim that the Morse homology is independent of the choices of 1-data $\mathfrak{H}^1$ up to isomorphism. 

\begin{notation*} 
\textup{(i) In this paper, we consider various types of homotopies, indexed by some simplicial sets. Whenever it is clear from the context, we put $\bullet$ for simplicity, not explicitly writing down the specific indices. (ii) The homotopies are given as a pair of parametrized Morse functions and Riemannian metrics. In most cases, we only consider homotopies of Morse functions, and those of Riemannian metrics are omitted from the notation. They should be regarded as being implicitly included in our discussions. (iii) The same letter $g$ is used for two different contexts: a Riemannian metric and a morphism of the indexing category $\mathcal{I}.$} 
\end{notation*}

\subsection*{Acknowledgements}
We thank Tomoyuki Abe, Anna Cepek, Gabriel C. Drummond-Cole, Kenji Fukaya, Damien Lejay, Andrew Macpherson, Yong-Geun Oh, and Otto van Koert for helpful discussions. Additionally, we appreciate Anna Cepek for her grammatical advice.

\section{Homotopy coherent diagrams in chain complexes}

\subsection{Homotopy direct systems}

Recall that the poset category $\mathcal{I}$ is given by:

\begin{itemize}
\item Ob($\mathcal{\mathcal{I}}) = \mathbb{Z}_{\geq 0},$
\item Mor($\mathcal{\mathcal{I}}$) is given by associating $ \text{exactly one morphism } f_{a,b} \text{ to each ordered}$ $\text{ pair } a \leq b$ of objects. In particular, when $a=b, \ f_{ab}$ is necessarily the identity morphism $id_a.$ 
\end{itemize}

Recall that a \textit{homotopy coherent diagram} is defined by a map of simplicial sets
\begin{equation}\nonumber
\mathscr{D} : K \rightarrow \mathcal{C},
\end{equation}
where $K$ is a simplicial set and $\mathcal{C}$ is an $\infty$-category. In this paper, we are interested in the following example.
\begin{equation}\nonumber
\mathscr{F} : N(\mathcal{I}) \rightarrow N_{dg}\big(Ch(\mathbb{Z}_2)\big).
\end{equation}
See appendices for its detail (e.g., on nerves and dg-nerves).

Effectively, the data of $\mathscr{F}$ can be thought of as follows.
\begin{enumerate}[label = \textbullet]
\item To a vertex $a \in N(\mathcal{I})_0,$ 
we assign a chain complex $C_a.$ If $a \in N(\mathcal{I})_0,$ then $C_a = \mathscr{F}(a).$
\item To a nondegenerate $k$-simplex $\sigma_k \in  N(\mathcal{I})_k$ with $k \geq 1,$ we assign a degree $k-1$ linear map, $\varphi_{\sigma_k} : C_{s \sigma_k} \rightarrow C_{t \sigma_k},$ satisfying   
\begin{equation}\label{hceq}
\partial \circ \varphi_{\sigma_k} + \varphi_{\sigma_k} \circ \partial + \sum^k_{i=1} \varphi_{\partial_i \sigma_k} + \sum^k_{i=1} \varphi_{(\sigma_k)^2_{k-i}} \circ \varphi_{(\sigma_k)^1_i} = 0.
\end{equation}
\item To a degenerate $k$-simplex, with $k \geq 1,$ we assign the zero map.
\end{enumerate} 
For more about $\mathscr{F}$, the reader can refer to appendices.

\begin{definition}
Let us call such a map $\mathscr{F}$ a \textit{homotopy direct system} of chain complexes (over $\mathbb{Z}_{2}$). We say that $\mathscr{F}$ is {\textit{strict}} if it satisfies 
\begin{equation}\nonumber
\mathscr{F}(f_1, \cdots, f_{k}) = 0
\end{equation}
for all $(f_1, \cdots, f_{k}) \in N(\mathcal{I})_k$ with $k \geq 2.$
\end{definition}

The following lemma is an easy exercise.
\begin{lemma}
When $\mathscr{F}$ is strict, we have for all $k \geq 1,$
\begin{equation}\nonumber
\mathscr{F}(f_k \circ f_{k-1} \circ \cdots \circ f_{1}) = \varphi_{f_k \circ f_{k-1} \circ \cdots \circ f_{1}} = \varphi_{f_{k}} \circ \cdots \circ \varphi_{f_1}= \mathscr{F}(f_{k}) \circ \cdots \circ \mathscr{F}(f_1).
\end{equation}
\end{lemma}

A \textit{colimit} of a homotopy coherent diagram is by definition an extension
\begin{equation}\nonumber
\overline{\mathscr{F}}: N(\mathcal{I})^{\triangleright} \rightarrow N_{dg}\big(Ch(\mathbb{Z}_2)\big)
\end{equation}
of $\mathscr{F}$ that is initial (unique up to a contractible choice). Here $ N(\mathcal{I})^{\triangleright}$ is a simplicial set obtained from $ N(\mathcal{I})$ by adding a cone point $*.$ (See Appendix A for the notation and the precise definition.)

\begin{proposition}\label{prcol}
There exists a colimit $\overline{\mathscr{F}}$ such that $\overline{\mathscr{F}}(*)$ is a chain complex over $\mathbb{Z}_2$ given by 

\begin{equation}\nonumber
\overline{\mathscr{F}}(*) := \bigoplus_{k \geq 0} \bigoplus_{\substack{\sigma_k \in N(\mathcal{I}),\\ \textup{nondeg.}}} \mathbb{Z}_2 \langle \sigma_k \rangle \otimes C_{s\sigma_k}, 
\end{equation}
with the differential $\partial$
\begin{equation}\label{diffctild}
\begin{split}
\partial : (\sigma_k ; x) \mapsto &\sum\limits_{i=0}^{k-1} (\partial_i \sigma_k ; x)+ \sum\limits_{i=0}^{k-1} \big((\sigma_k)_{i}^2 ; {\varphi}_{(\sigma)_{k-i}^1}(x)\big) + (\sigma_k ; \partial x),\\
\end{split}
\end{equation}
and the grading
\begin{equation}\nonumber
|(\sigma_k;x)| := k+ |x|.
\end{equation}
\end{proposition}

On the homologies, $\mathscr{F}$ induces a strict diagram $H\mathscr{F}$. We can consider its colimit, denoted by colim$H\mathscr{F},$ which exists in the category of chain complexes. In the case of $\mathscr{F}$, this amounts to considering a direct system of chain complexes and its direct limit, which we denote by $\lim\limits_{\longrightarrow} H\mathscr{F}.$ 
\begin{theorem}\label{ftiso}
There is an isomorphism
\begin{equation}\nonumber
H_*\big(\overline{\mathscr{F}}(*), \partial \big) \simeq \lim\limits_{\longrightarrow} H\mathscr{F}.
\end{equation}
\end{theorem}

For proofs of Proposition \ref{prcol} and Theorem \ref{ftiso}, we refer the reader to Appendix B.

\subsection{Morphisms} We propose models for the notion of a morphism between two homotopy direct systems of chain complexes.

\subsubsection*{Indexing categories} We introduce two indexing categories
\begin{enumerate}[label= (\roman*)]
\item We denote by $\mathcal{I}_0$ the category with two objects, say $0$ and $1,$ and one nontrivial morphism in each direction between them, we consider the product category $\widehat{\mathcal{I}} := \mathcal{I} \times \mathcal{I}_0.$
 
\item Let $\overline{\mathcal{I}}$ denote the category with 

\begin{enumerate}[label = \textbullet]
\item $Ob(\mathcal{\mathcal{I}}) = \mathbb{Z}_{\geq 0}$,\\
\item $Mor(\mathcal{\mathcal{I}})$ given by associating\\
\begin{equation}\nonumber
\begin{split}
- \text{exactly one morphism } f_{a,b}& \text{ to each ordered pair } a \leq b,\\
- \text{exactly one morphism } f_{a,b}& \text{ to each ordered pair } a \geq b, \text{ and }\\ 
&  a =2c+1, b=2c \text{ for some } c \in \mathbb{Z}_{\geq 0}.
\end{split}
\end{equation}
\end{enumerate}
In particular, we have \begin{equation}\nonumber
f_{2c+1, 2c} \circ f_{2c,2c+1} = id_{2c}, \ f_{2c, 2c+1} \circ f_{2c+1,2c} = id_{2c+1}.\\
\end{equation}
\end{enumerate}

\subsubsection*{A model with the category $\widehat{\mathcal{I}}$} Given homotopy coherent diagrams $\mathscr{F}, \mathscr{F}' : N(\mathcal{I}) \rightarrow N_{dg}\big(Ch(\mathbb{Z}_2)\big),$ we consider another
\begin{equation}\nonumber
\widehat{\mathscr{F}} : N(\widehat{\mathcal{I}}) \rightarrow N_{dg}\big(Ch(\mathbb{Z}_2)\big)
\end{equation} 
which extends $\mathscr{F}$ and $\mathscr{F}',$ i.e., $\widehat{\mathscr{F}}|_{N(\mathcal{I}) \times \{0\}}= \mathscr{F}$ and $\widehat{\mathscr{F}}|_{N(\mathcal{I}) \times \{1\}}= \mathscr{F}'$ 

\begin{example}\label{egstrdia}
For two strict diagrams $\mathscr{F}_s$ and $\mathscr{F}'_s,$ such that $\mathscr{F}_s (a) = \mathscr{F}'_s (a)$ for all $a \in Ob(\mathcal{I}),$ we can check that such an extension $\widehat{\mathscr{F}}$ exists. Consider a $k$-simplex $\big( (f_1, f_1'), \cdots, (f_k, f_k') \big) \in N(\widehat{\mathcal{I}})_k,$ where $(f_1, \cdots, f_k) \in N(\mathcal{I})_k$ and $(f_1', \cdots, f_k') \in N(\mathcal{I}_0)_k.$ Let $\#_{ni}$ denote the number of nonidentity morphisms in $ (f_1', \cdots, f_k').$ We define
\begin{equation}\nonumber
\widehat{\mathcal{F}} : \big( (f_1, f_1'), \cdots, (f_k, f_k') \big) \mapsto \widehat{\varphi}_{\big( (f_1, f_1'), \cdots, (f_k, f_k') \big)},
\end{equation}
as follows.
\begin{equation}\nonumber
\widehat{\varphi}_{\big( (f_1, f_1'), \cdots, (f_k, f_k') \big)}:= \begin{cases}
0 & \text{ if } \#_{ni} \geq 1 \text{ and } k \geq 2, \\
& \quad \quad \text{ or } \big( (f_1, f_1'), \cdots , (f_k, f_k') \big) \text{ is degenerate,}\\
id_{a} & \text{ if } \#_{ni}  = 1, k =1, \text{ and }  f'_{1} = id, \\
f_{1} & \text{ if } \#_{ni}  = 1, k =1, \text{ and }  f'_{1} \neq id, \\
& \quad \quad \text{ or } \#_{ni} = 0 \text{ and } k = 1,\\
0 & \text{ if } \#_{ni} = 0 \text{ and } k \geq 2.\\
\end{cases}
\end{equation}
It is straightforward to check that $\bigl\{ \widehat{\varphi}_{\big( (f_1, f_1'), \cdots, (f_k, f_k') \big)} \bigr\}$ satisfies the relation (\ref{hceq}), and hence defines a homotopy coherent diagram. 
\end{example}

\begin{remark}
We can think of a colimit of the $\infty$-functor $\widehat{\mathscr{F}}$, whose existence is guaranteed by the fact that in general the $\infty$-category $N_{dg}\big(Ch(\mathbb{Z}_2)\big)$ admits colimits for any type of diagrams.
\end{remark}

We consider the following pair of functors 
\begin{equation}\nonumber
\Phi : \mathcal{I} \rightleftarrows \widehat{\mathcal{I}} : \Psi,
\end{equation}
where $\Phi$ is the embedding to either obvious copy of $\mathcal{I}$ inside $\widehat{\mathcal{I}},$ and $\Psi$ is the map that forgets the $\mathcal{I}_0$ part. These functors give rise to an equivalence of categories, and the induced maps of simplicial sets 
\begin{equation}\nonumber
\Phi_* : N(\mathcal{I}) \rightleftarrows N(\widehat{\mathcal{I}}) : \Psi_*
\end{equation}
are categorical equivalences in the sense of [Lur1] and [Rie]. The equivalence

\begin{equation}\label{ceco12}
colim \mathscr{F} = colim(\widehat{\mathscr{F}} \circ \Phi_*) \simeq colim \widehat{\mathscr{F}},
\end{equation}
follows from the fact that a categorical equivalence gives rise to a cofinal map (c.f. Definition \ref{coflur}). Hence, we have a quasi-isomorphism of chain complexes $\overline{\mathscr{F}}(*) \xrightarrow{\simeq} colim \widehat{\mathscr{F}}(*).$ In following the same procedure for $\mathscr{F}',$ we obtain the following lemma.

\begin{lemma}
Suppose that the above type of extension $\widehat{\mathscr{F}}$ exists, then we have quasi-isomorphisms of chain complexes
\begin{equation}\nonumber\label{htpydirsysintro}
\begin{tikzcd}
{} & colim {\widehat{\mathscr{F}}}(*)  & {}\\
\overline{\mathscr{F}}(*) \arrow{ru}{\simeq} & {} &\overline{\mathscr{F}}'(*) \arrow[swap]{lu}{\simeq} .
\end{tikzcd}
\end{equation}
\end{lemma}

\subsubsection*{A model with the category $\overline{\mathcal{I}}$} We consider functors $\Psi', \Phi'_{0}, \Phi'_{1} : \mathcal{I} \hookrightarrow \overline{\mathcal{I}}$ given by
\begin{equation}\nonumber
\begin{cases}
\Psi'(a) = a, \ \Psi'( a \rightarrow b ) = (a \rightarrow b),\\
\Phi'_{i}(a) = 2 a +i, \  \Phi'_{i}(a \rightarrow b) = (2 a +i \rightarrow 2 b+i), \ i = 0,1,
\end{cases}
\end{equation}
and $\Phi''_{0}, \Phi''_{1} : \mathcal{I} \hookrightarrow {\mathcal{I}}$ by
\begin{equation}\nonumber
\Phi''_{i}(a) = 2 a +i, \  \Phi''_{i}(a \rightarrow b) = (2 a +i \rightarrow 2 b+i), \ i = 0,1.
\end{equation}
Observe that $\Phi'_{0}, \Phi'_{1}$ are equivalences of categories, while $\Phi''_{0*}, \Phi''_{1*}$ are not.
These functors induce maps of simplicial sets:
\begin{equation}\nonumber
\begin{split}
\Psi_{*}', \Phi'_{0*}, \Phi'_{1*} &: N(\mathcal{I}) \hookrightarrow N(\overline{\mathcal{I}}),\\
\Phi''_{0*}, \Phi''_{1*} &: N(\mathcal{I}) \hookrightarrow N({\mathcal{I}}).
\end{split}
\end{equation}
Then $\Phi'_{0*}, \Phi'_{1*}$ are categorical equivalences similarly as before.

For given homotopy coherent diagrams $\mathscr{F}, \mathscr{F}' : N(\mathcal{I}) \rightarrow N_{dg}\big(Ch(\mathbb{Z}_2)\big)$ with $\mathscr{F}(a) = \mathscr{F}'(a)$ for all $a \in \mathbb{Z}_{\geq 0},$ we consider the two diagrams
\begin{equation}\nonumber
\begin{split}
\mathscr{F}^{\overline{\mathcal{I}}} : N(\overline{\mathcal{I}}) \rightarrow  N_{dg}\big( Ch(\mathbb{Z}_{2}) \big),\\
\dot{\mathscr{F}} : N({\mathcal{I}}) \rightarrow  N_{dg}\big( Ch(\mathbb{Z}_{2}) \big).\\
\end{split}
\end{equation}
(Here $\mathscr{F}^{\overline{\mathcal{I}}}$ should not be mistaken for $\overline{\mathscr{F}}.$) We require them to be {\textit{extensions}} of $\mathscr{F}$ and $\mathscr{F}'$ in the the sense that the following relations hold:
\begin{equation}\nonumber
\begin{split}
\mathscr{F}^{\overline{\mathcal{I}}} \circ \Phi''_{0*} = \mathscr{F}&, \ \mathscr{F}^{\overline{\mathcal{I}}} \circ \Phi''_{1*} = \mathscr{F}',\\
\dot{\mathscr{F}} \circ \Phi'_{0*} = \mathscr{F}&, \ \dot{\mathscr{F}} \circ \Phi'_{1*} = \mathscr{F}'.\\
\end{split}
\end{equation}
And we require $\mathscr{F}^{\overline{\mathcal{I}}}$ to be an extension of $\dot{\mathscr{F}},$ i.e.,
\begin{equation}\nonumber
\mathscr{F}^{\overline{\mathcal{I}}} \circ \Psi'_{*} = \dot{\mathscr{F}}.
\end{equation}

If these extensions exist, then we have 
\begin{equation}\nonumber
colim \mathscr{F}  \xrightarrow{\simeq}  colim \mathscr{F}^{\overline{\mathcal{I}}}  \xleftarrow{\simeq} colim\mathscr{F}'
\end{equation}
for the same reason as (\ref{ceco12}). Now the following lemma follows.

\begin{lemma}
If an extension $\mathscr{F}^{\overline{\mathcal{I}}}$ exists, we have quasi-isomorphisms of chain complexes
\end{lemma}
\begin{equation}\nonumber
\begin{tikzcd}
{} & colim {{\mathscr{F}}^{\overline{\mathcal{I}}}}(*)  & {}\\
\overline{\mathscr{F}}(*) \arrow{ru}{\simeq} & {} & \overline{\mathscr{F}'}(*) \arrow[swap]{lu}{\simeq} .
\end{tikzcd}
\end{equation}

\section{Morse theory on compact manifolds}

In this section, we provide several standard results of Morse theory without proof. The reader can refer to [Sch] for full details. Those who familiar with basic Morse theory may feel free to skip this section and go to Section 4.

Let $(M,g)$ be a compact Riemannian manifold andwith a Morse function $h:M \rightarrow \mathbb{R}$ on it.

\subsection{The space of trajectories}

Consider the compactification $\overline{\mathbb{R}} = \mathbb{R} \cup \{\pm \infty\}.$ For a small open neighborhood $\mathcal{O}$ of the zero section in $TM$ and a smooth map $w : \overline{\mathbb{R}} \rightarrow M$ from $w(-\infty) = x$ to $w(\infty) = y,$ consider the exponential map $\exp_w : W^{1,2}(w^*\mathcal{O}) \rightarrow M.$ Then the map $\exp_w$ is a diffeomorphism onto the image when $\mathcal{O}$ is sufficiently small.

\subsubsection*{The Banach manifold $\mathcal{P}^{1,2}(x,y)$} For two points $x,y$ in $M,$ we denote 
\begin{equation}\nonumber
\begin{split}
C^{\infty}_{x,y}(\overline{\mathbb{R}}, M) := \bigl\{ w : \overline{\mathbb{R}} \rightarrow M \mid w(- \infty) & = x, \ w(\infty) = y, \ |w(\pm s)| < e^{\mp k s},\\ & \text{ for}  \text{ sufficiently large } |s| \text{ and some }k>0 \bigr\}.
\end{split}
\end{equation}
We also denote
\begin{equation}\nonumber
\begin{split}
\mathcal{P}^{1,2}(x,y) : = \bigl\{ \gamma : \overline{\mathbb{R}} \rightarrow M \mid & \gamma(-\infty) = x, \ \gamma(\infty) = y,\\ & \quad \exp_{w}(s) = \gamma \text{ for some } s \in W^{1,2}(w^*\mathcal{O})\\ & \quad \quad \text{ with } s(\pm \infty) =0, \text{ and } \ w \in C^{\infty}_{x,y}(\overline{\mathbb{R}}, M) \bigr\}.
\end{split}
\end{equation}
In our case, the inequality $k - \frac{m}{p} > 0$ holds for $(k,p,m) = (1,2,1),$ so we know $\gamma \in \mathcal{P}^{1,2}(x,y)$ is continuous by the Sobolev embedding theorem.

\begin{proposition}
We have
\begin{enumerate}[label=(\roman*)]
\item $\mathcal{P}^{1,2}(x,y)$ is a smooth Banach manifold with local charts 
\begin{equation}\nonumber
\bigl\{W^{1,2}(w^*\mathcal{O}), \exp_w( \cdot )\bigr\}_{w  \in C^{\infty}_{x,y}(\overline{\mathbb{R}}, M) }.
\end{equation}
\item $T\mathcal{P}^{1,2}(x,y)$ is given by $\bigcup\limits_{\gamma \in \mathcal{P}^{1,2}(x,y)} \gamma^* TM$, and we have 
\begin{equation}\nonumber
W^{1,2}\big(T\mathcal{P}^{1,2}(x,y)\big) = \bigcup\limits_{\gamma \in \mathcal{P}^{1,2}(x,y)}W^{1,2}(\gamma^* TM).
\end{equation} 
\end{enumerate}
\end{proposition}

\subsubsection*{The fiber derivatives} Consider the tangent bundle $\tau : TM \rightarrow M$ and a connection $K$ on it. For $\xi \in \mathcal{O},$ we have the following two isomorphisms:
\begin{equation}\nonumber
\begin{split}
\nabla_1 \exp (\xi) &:= d \exp(\xi) \circ \big( d \tau (\xi) |_{H} \big)^{-1} : T_{\tau(\xi)}M \xrightarrow{\simeq} T(T_{\tau(\xi)}M) \xrightarrow{\simeq} T_{\exp(\xi)}M\\
\nabla_2 \exp (\xi) &:= d \exp(\xi) \circ \big( K (\xi) |_{V} \big)^{-1} : T_{\tau(\xi)}M \xrightarrow{\simeq} T(T_{\tau(\xi)}M) \xrightarrow{\simeq} T_{\exp(\xi)}M,
\end{split}
\end{equation}
where $d \tau (\xi) |_H $ and $K(\xi)|_V$ are isomorphisms obtained from the restrictions to the horizontal and vertical subspaces, respectively. (Recall that $d \exp(\xi)$ is an isomorphism when $\mathcal{O}$ is sufficiently small.) $\nabla_2 \exp(\xi)$ is called the \textit{fiber derivative} of the exponential map. The fiber derivative of the transition map of the local charts $\varphi_{f_2 \circ f_1} : = \varphi_{f_2} \circ \varphi^{-1}_{f_1} : \mathcal{O} \rightarrow \mathcal{O}$ is given by 
\begin{equation}\nonumber
\nabla_2 \exp \big( \exp_{f_1(t)}^{-1} \circ \exp_{f_2(t)} (\xi) \big)^{-1} \circ \nabla_2 \exp(\xi).
\end{equation}
We have 
\begin{equation}\nonumber
\begin{split}
\frac{\partial}{\partial t} \exp \big(\xi(t) \big) & = \nabla_1 \exp(\xi)\big(d \tau (\dot{\xi})\big) + \nabla_2 \exp(\xi)\big(K(\dot{\xi})\big)\\
& = \nabla_1 \exp(\xi)(\dot{h}) + \nabla_2 \exp(\xi)(\nabla_t \xi),
\end{split}
\end{equation}
where $\nabla_t \xi = K(\dot{\xi})$ and $\dot{h} = d \tau(\dot{\xi}).$

\subsubsection*{The flow equation} We consider a map of Banach manifolds
\begin{equation}\nonumber
F : \mathcal{P}^{1,2}(x,y) \rightarrow \bigcup\limits_{\gamma \in \mathcal{P}^{1,2}(x,y)} L^2 (\gamma^* TM)
\end{equation}
that is given by $F(\gamma) = \dot{\gamma} + X \circ \gamma,$ where $X$ is either
\begin{equation}\nonumber
\begin{cases}
X = \nabla h \text{ for a time-independent Morse function $h,$ or } \\
X = \frac{\nabla h}{\sqrt{1 + |\nabla h|^2 |\dot{h}|^2}} \text{ for a time-dependent Morse function $h$}
\end{cases}
\end{equation}
The image of $F$ can be shown to lie in the class of $L^2$ maps. In the local chart, $\bigl\{W^{1,2}(w^* \mathcal{O}), \exp_w\bigr\},$ $F$ can be written for the time-independent case as follows:
\begin{equation}\nonumber
\begin{split}
F_{loc} : W^{1,2}(w^* \mathcal{O}) \rightarrow & L^2(w^* TM),\\
\xi \mapsto & \big( \nabla_2 \exp (\xi) \big)^{-1} \circ \big( \nabla_1 \exp(\xi)(\dot{h})\\  & \quad + \nabla_2 \exp (\xi)(\nabla_t \xi)  + (\nabla h \circ \exp_h )(\xi) \big)\\
& \quad \quad = \nabla_t \xi (t) + g(t, \xi(t)). 
\end{split}
\end{equation}
Here the map $g : {\mathbb{R}} \times w^* \mathcal{O} \rightarrow w^* TM$ is smooth and fiber respecting at each $t$ with the condition $\lim_{s \rightarrow \pm \infty} g(s, 0) = 0.$ Moreover, the asymptotical fiber derivatives $F_2(\pm \infty, 0)$ are conjugated to the Hessians of $h^{\alpha}$ and $h^{\beta}$ at $x^{\alpha}$ and $x^{\beta},$ respectively. The time-dependent case is similar.

It turns out that the zero set of $F$ coincides with the set of smooth curves $w : \overline{\mathbb{R}} \rightarrow M$ that are solutions of $\dot{w} + X \circ w = 0$ with the condition $w(-\infty) = x$ and $w(\infty) = y.$ We denote the zero set of $F$ by
\begin{equation}\nonumber
\begin{cases}
 \widehat{\mathcal{M}}(h;x,y) \text{ for the time-independent case,} \\
 \mathcal{M}(h;x,y) \text{ for the time-dependent case.}
  \end{cases}
 \end{equation}

The linearization of $F$ is computed as follows:

\begin{equation}\nonumber
\begin{split}
dF : W^{1,2}(\mathbb{R}, \mathbb{R}^n) \rightarrow L^2(\mathbb{R}, \mathbb{R}^n)\\
\xi \mapsto \sum\limits_{i} dF(u)(\xi),
\end{split}
\end{equation}
where 
\begin{equation}\label{aaaaa}
dF(u) (\xi) = \frac{\partial}{\partial s} \xi + A(\xi),
\end{equation}
where $A$ is a continuous family of the endomorphisms of $L^2(\mathbb{R}, \mathbb{R}^n)$ such that $A(\pm \infty)$ is nondegenerate and self-adjoint. Then we have

\begin{theorem}
$dF$ is a Fredholm operator of index $|x| - |y|$ between Banach spaces.
\end{theorem}

\subsubsection*{A parameter space Banach manifold} Let $E_g:= End_{sym, g}(TM)$ denote the subset of endomorphisms $End(TM)$ consisting of symmetric self-adjoint ones with respect to the Riemannian metric $g.$ Also, let $\mathcal{T}_g$ denote the set of positive definite elements in $E_g.$ Let $e = (e_1, e_2, \cdots)$ be a sequence of positive real numbers. For $s \in \Gamma(End_{sym,g})(TM),$ we define a norm by 
\begin{equation}\nonumber
\|s\|_e:= \sum\limits_{k \geq 0} e_k \sup\limits_{p \in M} | \nabla^k s(p) |,
\end{equation}
where
\begin{equation}\nonumber\\
| \nabla^k s (p)| : = \max \bigl\{ \nabla_{v_1} \cdots \nabla_{v_k} s (p) \mid \| v_1 \| = \cdots = \| v_k \| =1, v_1, \cdots v_k \in T_p M \bigr\}. 
\end{equation}
We write $C^{e}(E_g) := \bigl\{s \in C^{\infty}(E_g) \mid \|s\|_{\epsilon} < \infty \bigr\}.$ Then we have:
\begin{lemma}\label{l2lem}
 $\big(C^{e}(E_g), \| \cdot \| \big)$ is a Banach space unless it is empty. Moreover, there is a sequence $e$ such that $C^{e}(E_g)$ is dense in $L^2(E_g)$.
\end{lemma}

\begin{lemma}
$C^{e}(\mathcal{T}_g)$ is an open neighborhood of $id \in C^e(E_g)$ and there is an isomorphism $T_XC^{e}(\mathcal{T}_g) \simeq C^{e}(E_g)$ for any $X \in C^{e}(\mathcal{T}_g).$
\end{lemma}

We consider the following map of Banach manifolds:
\begin{equation}\nonumber
\begin{split}
\Phi : C^{e}(\mathcal{T}_g) \times & \mathcal{P}^{1,2}(x,y) \longrightarrow \bigcup\limits_{\gamma \in \mathcal{P}^{1,2}(x,y)} L^2(\gamma^* TM),\\
(A, \gamma) & \longmapsto 
\begin{cases}
 \dot{\gamma} + A(\nabla \gamma) \circ \gamma,  \quad \text{ for the time-independent case, }\\
 \dot{\gamma} + \frac{A (\nabla h)}{\sqrt{ 1+ |\dot{h}| |A (\nabla h)|^2}} \circ \gamma, \text{ for the time-dependent case.}
\end{cases}
\end{split}
\end{equation}

\subsection{Transversality and compactness of the moduli space}

\subsubsection*{Transversality}When we try to achieve transversality for the moduli space of trajectories, we will refer to the following theorem.
\begin{theorem}\label{schtransv}([Sch] Proposition 2.24)
Suppose that
\begin{enumerate}[label = (\roman*)]
\item 0 is a regular value of $\Phi,$
\item $\Phi_A := \Phi(A, \cdot)$ is Fredholm of index $l$ for each $A \in C^{e}(\mathcal{T}_g),$
\end{enumerate}
then there exists a generic subset $S$ of $C^{e}(\mathcal{T}_g)$ such that $\Phi_A$ is surjective for each $A \in S.$ In particular, $\Phi^{-1}(0)$ is a submanifold of $\mathcal{P}^{1,2}(x,y)$ of dimension $l$ for each $A \in S$ by the implicit function theorem for Banach manifolds.
\end{theorem}

For time-independent Morse function $h,$ there is a free action by $\mathbb{R}$ on $\widehat{\mathcal{M}}(h;x,y)$ by 
$\gamma (\cdot) \mapsto \gamma( \cdot + \tau)$ for $\tau \in \mathbb{R}.$ Then we denote $\mathcal{M}(h;x,y) := \widehat{\mathcal{M}}(h;x,y) / \mathbb{R}.$

\begin{corollary}
There is a generic choice of Riemannian metric such that the space $\mathcal{M}(h;x,y)$ for time-independent $h$ it is a smooth manifold of of dimension $|x| - |y| - 1.$ When $h$ is time-dependent, it is a smooth manifold of dimension $|x| - |y|.$
\end{corollary}

\subsubsection*{Compactness theorem} We remark that our moduli space of trajectories is a metric space, and thus compactness is equivalent to sequential compactness. We say that a compact subset $K \subset \widehat{\mathcal{M}}(h;x,y)$ is \textit{compact up to broken trajectories} if either (i) any subset $\{u_n\}_n \subset K,$ has a convergent subsequence in $K$ or (ii) there exist critical points of $h,$ $x = x_0, x_1, \cdots, x_l =y$ with $|x| > |x_1| > \cdots |y|,$ $v_i \in \mathcal{M}(h;x_i,x_{i+1}),$ together with time reparametrizations $\{ \tau_{m, i} \}_m \subset \mathbb{R}, \ (i =0,\cdots, l-1)$, so that we have a convergence $u_{m_k}(\cdot + \tau_{m_k,i}) \xrightarrow{C^{\infty}_{loc}} v_i$ as $k \rightarrow \infty.$

\begin{theorem}
Any sequence $\{ u_n \}_n \subset \mathcal{M}(h;x,y) \subset C^{\infty}(\overline{\mathbb{R}}, M)$ converges to $v \in C^{\infty}(\overline{\mathbb{R}}, M)$ in $C^{\infty}_{loc}$. Moreover, if $v \in \mathcal{M}(h;x,y),$ then it converges to $v$ in $W^{1,2}$.  
\end{theorem}

\begin{theorem}
$\widehat{\mathcal{M}}(h;x,y)$ is compact up to broken trajectories.
\end{theorem}

By this theorem, $\mathcal{M}(h;x,y)$ (for either time-independent or time-dependent Morse function $h$) can be compactified to $\overline{\mathcal{M}}(h;x,y)$ by adding broken trajectories, so that we have $\partial \overline{\mathcal{M}}(h;x,y) = \bigcup\limits_{\substack{z \in Crit(h),\\ |y| < |z| <|x|}} \overline{\mathcal{M}}(h;x,z) \times \overline{\mathcal{M}}(h;z,y).$ In particular, when $|x| = |y|+1, \ \overline{\mathcal{M}}(h;x,y)$ is a finite set, and when $|x| = |y| + 2,$ it is a compact 1-dimensional manifold.

\subsection{Morse chain complexes and homologies}

We assume that $g$ is chosen from the generic set of Theorem \ref{schtransv}, and consider the following free $\mathbb{Z}_2$-module associated to a triple $(M, h, g)$:
\begin{equation}\nonumber\nonumber
MC_*(M,h,g) :=  \bigoplus\limits_{x \in Crit(f)} \mathbb{Z}_2 \cdot x,
\end{equation}
where we put a grading by the Morse index, i.e., the number of negative eigenvalues of the Hessian of $h$ at the critical point. We define a map of degree $-1$ by
\begin{equation}\nonumber\nonumber
\begin{split}
d : MC_*(M,h,g) \longrightarrow MC_{*-1}(M,h,g)\\
Crit(f) \ni x \longmapsto \sum\limits_{\substack{y \in Crit(f), \\ |y| = |x| -1 }} \#_2 \overline{\mathcal{M}}(h;x,y) \cdot y.
\end{split}
\end{equation}
Since $\overline{\mathcal{M}}(h;x,y)$ is a finite set when $|x| - |y|  = 1,$ this map is well-defined. In fact, we have $d \circ d = 0,$ which can be checked from the following observation:
\begin{equation}\nonumber\nonumber
\begin{split}
d \circ d (x) &= \sum\limits_{\substack{ x, y \in Crit(h), \\ |z| = |x| -1, \\ |y| = |x| -2 }} \#_2 \big(\overline{\mathcal{M}}(h;x,z) \times \overline{\mathcal{M}}(h;z,y) \big) \cdot y  \\
&= \sum\limits_{\substack{ y \in Crit(h), \\ |y| = |x| -2}} \#_2 \big(\partial \overline{\mathcal{M}}(h;x,y)\big) \cdot y  = 0,
\end{split}
\end{equation}
where the last equality follows from the fact that any compact connected 1-dimensional manifold is either a circle or a closed interval, and thus the number of its boundary components is 0 modulo 2. 

\begin{definition}
We call $\big( MC_*(M,h,g), d \big)$ the \textit{Morse chain complex} associated to $(M,h,g).$ We define the \textit{Morse homology} of a triple $(M,h,g)$ by
\begin{equation}\nonumber\nonumber
MH_*(M,h,g) : = \frac{\ker d}{\text{im } d}.
\end{equation}
\end{definition}

\begin{notation}
For simplicity, we sometimes omit one or two components of the triple $(M,h,g)$ in the notations of Morse chain complexes and homologies without changing the meanings.  
\end{notation}

\begin{theorem}
$HM_*(M,h,g)$ is independent of the choice of $(h,g)$ up to isomorphism.
\end{theorem}

\section{Homotopy constructions}

We begin our study of the noncompact case by introducing homotopical notions. Some of them are just restatements of known results, written in the forms which we will take advantage of in later sections. 

\subsection{A noncompact setting}

\subsubsection*{Morse functions with boundary conditions} Let $M$ be a compact manifold with a possibly nonempty boundary $\partial M$ and $h$ a Morse function on it, satisfying
\begin{equation}\label{inwdircon}
-\nabla h \text{ is in the inward direction along } \partial M.
\end{equation}

\begin{figure}[h!]
  \includegraphics[width=4.5cm]{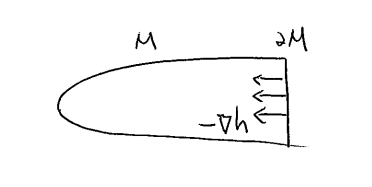}
  \caption{The inward direction condition}
  \label{}
\end{figure}

Throughout this paper, we will call this the {\textit{inward direction condition.}} This condition requires the gradient vector field to be transversal to the boundary. It also prevents the gradient vector field from vanishing, so that there are no critical points at the boundary. 

\begin{remark}
Trajectories connecting two critical points cannot intersect $\partial M$; hence the discussion of the previous section for compact manifolds is available.
\end{remark}

\subsubsection*{Compact exhaustions} We are interested in a noncompact manifold $W$ equipped with a \textit{compact exhaustion} (or an \textit{exhaustion} in short) by which we mean the following data:

\begin{itemize}[ label = \textbullet]
\item $\bigl\{(M_a, h_a, g_a)\bigr\}_{a \in \mathbb{Z}_{\geq 0}}$ a family of compact submanifolds of $W$ with Morse functions and Riemannian metrics, 
\item $\bigcup\limits_{a \in \mathbb{Z}_{\geq 0} } M_a = W,$
\item $M_a \subseteq \text{int}({M_{a+1}}), \text{ for each } a \in Ob(\mathcal{I}),$
\item Each $h_a$ satisfies the condition (\ref{inwdircon})
\item $\{\overline{h}^b_a : M_j \rightarrow \mathbb{R}\}_{a \leq b},$ a family of Morse functions, (called \textit{extensions}),
\item Each $\overline{h}^b_a$ satisfies the condition (\ref{inwdircon}),
\item $\overline{h}^b_a |_{M_a} \equiv h_a, \text{for each pair } a \leq b$ (in particular, $\overline{h}^a_a \equiv h_a, \text{for each } a$).
\end{itemize}

\begin{figure}[h!]
  \includegraphics[width=6.5cm]{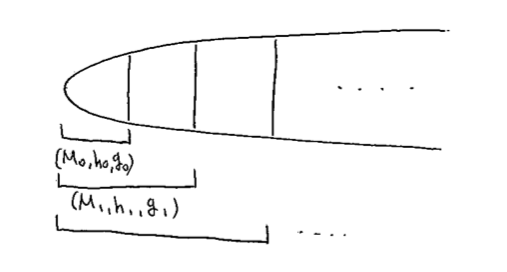}
  \caption{A compact exhaustion}
  \label{}
\end{figure}

\subsection{1-homotopies}

We introduce {\textit{1-homotopies}} and the related analytical statements. They can be dealt with the standard Morse theory for the time-dependent case, and the reader can refer to [Sch] or [AD] for proofs.

\begin{definition}\label{1htaxo}
Given an exhaustion $\bigl\{ (M_a, h_a, g_a)\bigr\}_{a \in \mathbb{Z}_{\geq 0}}$ of $W,$ by \textit{1-homotopies}, we mean a family of pairs of time-dependent Morse functions and Riemannian metrics 
\begin{equation}\nonumber
\begin{cases}
h_f : {\mathbb{R}} \times M_{tf}   \rightarrow \mathbb{R}\\ 
g_f : {\mathbb{R}} \rightarrow Met(M_{tf}), \ {f \in Mor(\mathcal{I})},
\end{cases}
\end{equation}
(where $Met(M_{\bullet})$ is the set of all Riemannian metrics on $M_{\bullet}$) that satisfy the following axioms: 
\begin{itemize}[label = \textbullet]
\item (Stability at the ends) There exists $R_f >0$ such that 
\begin{equation}\nonumber
h_f(s, \cdot) = \begin{cases}
\overline{h}^{tf}_{sf} \text{ if } s < -R_f,\\
h_{tf}  \text{ if } s > R_f,\\
\end{cases}
\end{equation}
\begin{equation}\nonumber
g_f(s) = \begin{cases}
\overline{g}^{tf}_{sf} \text{ if } s < -R_f,\\
g_{tf} \text{ if } s > R_f,\\
\end{cases}
\end{equation}
\item (Constant homotopies) $(h_{id_{a}}, g_{id_{a}})$ is the constant homotopy for each $a \in Ob(\mathcal{I}).$
\item (Extensions) There exists a family of time-dependent Morse functions and Riemannian metrics 
\begin{equation}
\begin{cases}
\overline{h}^k_f : {\mathbb{R}} \times M_k \rightarrow \mathbb{R}\\
\overline{g}^k_f : {\mathbb{R}} \rightarrow Met(M_{k}), \  tf \leq k, \ f \in Mor(\mathcal{I}),
\end{cases}
\end{equation}
\item (Trivial extensions) $\overline{h}^{tf}_f \equiv h_f \text{ for each } f \in Mor(\mathcal{I}),$
\item (Stability for extensions) For the same $R_f >0$ as above, and for all $k \geq tf,$ we have
\begin{equation}\nonumber
\overline{h}^k_f(s, \cdot) = \begin{cases}
\overline{h}^k_{sf}  \text{ if } s < -R_f,\\
\overline{h}^k_{tf}  \text{ if } s > R_f,\\
\end{cases}
\end{equation}
\begin{equation}\nonumber
\overline{g}^k_f(s, \cdot) = \begin{cases}
\overline{g}^k_{sf}  \text{ if } s < -R_f,\\
\overline{g}^k_{tf}  \text{ if } s > R_f,\\
\end{cases}
\end{equation}
\item (Inward direction condition) $-\nabla_{g_f} h_f$ and $-\nabla_{\overline{g}^{k}_f}\overline{h}^k_f$ are in the inward direction along the boundaries $\partial M_{tf}$ and $\partial M_k,$ respectively for all $f \in Mor(\mathcal{I})$ and $k$ with $k \geq tf.$
\item (Regularity) $h_f$ and $\overline{h}^k_f$ (simply denoted by $h_{\bullet}$ in what follows) are regular in the following sense: for any $p \in M_{tf}$ with $\nabla h (s,p) = 0,$ the operator $\frac{\partial}{\partial s} + H^2h(s,p) : W^{1,2}(T_pM) \rightarrow L^2(T_pM)$ is onto, where $H^2h(s,p) \in End(T_pM)$ denotes the Hessian of $h(s, \cdot)$ at p. 
\end{itemize}
\end{definition}

\begin{remark}
It is a standard fact that the regularity can always be achieved by a local perturbation of the 1-homotopy $(h_f, g_f)$ near critical points.
\end{remark}

\subsubsection*{The moduli space of trajectories} Given a 1-homotopy $\{h_f\}_{f \in Mor(\mathcal{I})}$ on $W,$ we consider the space of time-dependent trajectories for $x \in Crit(h_{sf})$ and $y \in Crit(h_{tf}),$
\begin{equation}\nonumber
\begin{split}
\mathcal{M}(h_f ; x,y) := \lbrace u : \overline{\mathbb{R}} \rightarrow {M}_{tf} \mid \dot{u} + \frac{\nabla h_f}{\sqrt{1+ |\dot{h}_f|^2 |\nabla h_f|^2 }} & \circ u = 0, \\ & u(-\infty) = x, u(\infty) = y \rbrace.
\end{split}
\end{equation}
Indeed, it is the zero set of the following map between Banach manifolds:
\begin{equation}\nonumber
\begin{split}
F_f : \mathcal{P}^{1,2}(x,y) & \longrightarrow \bigcup\limits_{\gamma \in \mathcal{P}^{1,2}(x,y)} L^2(\gamma^* TM_{tf}),\\
u & \longmapsto \dot{u} + \frac{\nabla h_f}{\sqrt{1+ |\dot{h}_f|^2|\nabla h_f|^2}} \circ u.
\end{split}
\end{equation}
Similar to the time-independent case, we have $\mathcal{M}(h_f ; x,y) \subset C^{\infty}(\overline{\mathbb{R}}, M_{tf}).$

\subsubsection*{Transversality and compactness of the moduli space} We state the transversality and compactification theorems for 1-homotopies.
\begin{theorem}\label{trasv1h}
There exists a generic choice of ${\mathbb{R}}$-family of Riemannian metrics $\{g_f\}$$_{f \in Mor(\mathcal{I})}$ on $M_{tf}$ for which $\mathcal{M}(h_f ; x,y)$ is a smooth manifold of dimension $|x|-|y|.$
\end{theorem}

\begin{theorem}\label{trasv2hh}
$\mathcal{M}(h_f ; x,y)$ can be compactified to $\overline{\mathcal{M}}(h_f ; x,y)$, so that 
\begin{equation}\nonumber
\begin{split}
\partial \overline{\mathcal{M}}(h_f ; x,y) = \bigcup\limits_{\substack{x' \in Crit(h_{tf}), \\ |x'| = |x| - 1}} \mathcal{M}(h_{sf} ; x,x') & \times \mathcal{M}(h_{f} ; x',y ) \\ 
& \cup \bigcup\limits_{\substack{y' \in Crit(h_{tf}), \\ |y'| = |y| +1}} \mathcal{M} (h_{f} ; x,y') \times \mathcal{M}(h_{tf} ; y',y).
\end{split}
\end{equation}
\end{theorem}

As a result, we have
\begin{corollary}
If $|x| = |y|,$ then $ \overline{\mathcal{M}}( h_f ; x,y) = \mathcal{M}( h_f ; x,y),$ and it is a finite set.
\end{corollary}
\begin{proof}
When $|x | = |y|,$ no breaking can take place for reasons of degree as is suggested by Theorem \ref{trasv2hh}. So $\mathcal{M}( h_f ; x,y)$ is already compact and it is zero dimensional by Theorem \ref{trasv1h}.

\end{proof}

\subsection{Continuation homomorphisms}

For the extension,
\begin{equation}\nonumber
(M_a, h_a) \hookrightarrow (M_b, \overline{h}^b_a),
\end{equation}
we have the inclusion
\begin{equation}\nonumber
Crit(h_a) \subset Crit(\overline{h}^b_a).
\end{equation}
For each $f \in Mor(\mathcal{I}),$ we define an injective $\mathbb{Z}_2$-linear map:
\begin{equation}\nonumber
\begin{split}
\iota_f : MC_*(M_{sf},h_{sf}) &\rightarrow MC_*(M_{tf}, \overline{h}^{tf}_{sf}),\\
x \in Crit(h_{sf}) &\mapsto x \in Crit(\overline{h}_{sf}^{tf}).
\end{split}
\end{equation}

Then we have:

\begin{corollary}
$\iota_f$ is a chain map.
\end{corollary}

\begin{proof}
This follows from the fact that there exists no outward trajectories from a critical point of $h_{sf}$ to the one of $\overline{h}^{tf}_{sf}$ which lies outside $M_{sf}$ due to condition (\ref{inwdircon}).
\end{proof}

\subsubsection*{Continuation homomorphisms} For $f \in Mor(\mathcal{I}),$ we define a $\mathbb{Z}_2$-linear map
\begin{equation}\nonumber
\begin{split}
\overline{\varphi}_f : MC_*(M_{sf}, h_{sf}) & \rightarrow MC_*(M_{tf},h_{tf}),\\
x \in Crit(\overline{h}^{tf}_{sf}) & \mapsto \sum\limits_{\substack{y \in Crit(h_{tf}), \\ |x| = |y| } } \#_2 \overline{\mathcal{M}}(h_f ; x,y) \cdot y.
\end{split}
\end{equation}
\begin{lemma}
$\overline{\varphi}_f$ is a chain map.
\end{lemma}
\begin{proof}
By Theorem \ref{trasv1h}, $\overline{\varphi}_f$ is well-defined. It is a standard procedure to show $\overline{\varphi}_f$ is a chain map, which follows from Theorem \ref{trasv2hh}. 
\end{proof}

We define
\begin{equation}\nonumber
\varphi_f : MC_*(M_{sf}, h_{sf}) \rightarrow MC_*(M_{tf}, h_{tf})
\end{equation}
by $\varphi_f := \overline{\varphi}_f \circ \iota_f$ and call the family $\{\varphi_f\}_{f \in Mor(\mathcal{I})}$ the \textit{continuation maps}.

The following proposition will be proved in a later section, in a more general context.

\begin{proposition}
$\varphi_f$ is a chain map for each $f \in Mor(\mathcal{I}).$ Moreover, the maps induced on homologies satisfy $\varphi_{g \circ f *}= \varphi_{g*} \circ \varphi_{f*}$ for all morphisms $f,g$ of $\mathcal{I}$ with $tf = sg.$
\end{proposition}

\subsubsection*{Concatenations of homotopies} Assume that we are given a transversal family of 1-homotopies $\{h_f\}_{f \in Mor(\mathcal{I})}.$ Let $f,g$ be morphisms of $\mathcal{I}$ with $tf = sg.$ From the pair $h_f$ and $h_g,$ we can construct a new 1-homotopy $h_f \#_{\rho} h_g,$ called the \textit{concatenation} of $h_f$ and $h_g,$ where we require $\rho > \max\{R_f, R_g\}$:\begin{equation}\nonumber
h_f \#_{\rho} h_g (s, \cdot) := \begin{cases}
\overline{h}_f^{tg}(s, \cdot) &\text{ if } s < -\max \{R_f, R_g \},\\
{h}_g(s, \cdot) &\text{ if } s > \max \{R_f, R_g \}.\\
\end{cases}
\end{equation}
We define the \textit{concatenation} of Riemannian metrics $g_f$ and $g_g,$ similarly. Its well-definedness and smoothness of the concatenation are guaranteed by the stability-at-the-ends condition. In Theorem \ref{1gt}, we show that it is transversal when $\rho$ is sufficiently large.

\subsubsection*{Gluing theorem} We can glue two trajectories, an important consequence of which is the following theorem.
\begin{theorem}\label{1gt}
There exists $\rho_0 >0$ such that for every $\rho > \rho_0, \ \mathcal{M}(h_f \#_{\rho} h_g ; x,y)$ is a smooth manifold of dimension $|x| - |y|,$ and when $|x| = |y| = |z|,$ we have a bijection between finite sets
\begin{equation}\nonumber
\mathcal{M}(h_f \#_{\rho} h_g; x,y) \simeq \bigcup\limits_{\substack{z \in Crit({h}_{sg}),\\ |x| = |y| = |z| }} \mathcal{M}({h}_f; x,z) \times \mathcal{M}(h_g ;z,y).
\end{equation}
\end{theorem}

\begin{proof}
This is a special case of later results, such as  Theorem \ref{gluthm} and Theorem \ref{lsurj}.
\end{proof}

For the proof of Theorem \ref{1gt}, we use the following lemma .

\begin{lemma}\label{lem411}
We have a bijection $\mathcal{M}(\overline{h}^{tg}_f; x,z) \simeq \mathcal{M}({h}_{f}; x,z).$
\end{lemma}

\begin{proof}
By the inward direction condition, no trajectory of $\overline{h}^{tg}_f$ can escape from $M_{tf}$ in $M_{tg}.$
\end{proof}

Similarly to how we defined $\varphi_f$ and $\varphi_g,$ the 1-homotopy $h_f \#_{\rho} h_g$ from $sf$ to $tg$ gives rise to a $\mathbb{Z}_2$-linear map for sufficiently large $\rho >0,$ 
\begin{equation}\nonumber
\varphi_{h_f \#_{\rho} h_g} : MC_*(M_{sf}, h_{sf}) \rightarrow MC_*(M_{tg}, h_{tg}).
\end{equation}

\begin{corollary}\label{corgt}
Let $\rho_0$ be as in Theorem \ref{1gt}. Then for any $\rho \geq \rho_0,$ the following diagram commutes.
\begin{equation}\nonumber
\begin{tikzcd}
MC_*(M_{sf}, h_{sf}) \arrow{d}[swap]{\iota_f}  \arrow{rr}{\varphi_{f}} & & MC_*(M_{tf}, h_{tf})  \arrow{d}{\varphi_{g}} \\
MC_*(M_{tg}, \overline{h}_{sf}^{tg}) \arrow{rr}{\overline{\varphi}_{h_f \#_{\rho} h_g}}
              & &MC_*(M_{tg},h_{tg}).
\end{tikzcd}
\end{equation}
Hence we have 
\begin{equation}\nonumber
\varphi_{h_f \#_{\rho} h_g} =  {\varphi}_{g} \circ {\varphi}_{f}.
\end{equation}
\end{corollary}
\begin{proof}
By Theorem \ref{1gt}, we have $\overline{\varphi}_{h_f \#_{\rho} h_g} = {\varphi}_{g} \circ \overline{\varphi}_{\overline{h}^{tg}_f}.$ Composition of $\overline{\varphi}_{h_f  \#_{\rho} h_g }$ with $\iota_f$ amounts to restricting its domain to $MC_*(M_{sf}, h_{sf}).$ Since $\overline{\varphi}_{\overline{h}^{tg}_{f}}|_{MC_*(M_{sf}, h_{sf})} = \varphi_{h_f}$ due to (\ref{inwdircon}), which in particular means that its image is in $MC_*(M_{tf},h_{tf}),$ we have:
\begin{equation}\nonumber
\varphi_{h_f \#_{\rho} h_g} = \overline{\varphi}_{h_f \#_{\rho} h_g} \circ \iota_f =  (\overline{\varphi}_{h_g} \circ \overline{\varphi}_{h^{tg}_f}) \circ \iota_{f} = \overline{\varphi}_{g} \circ \varphi_{f} = {\varphi}_{g} \circ {\varphi}_{f}.
\end{equation} 
\end{proof}

\begin{corollary}\label{cordire}
We have $\varphi_{g \circ f *} = \varphi_{g*} \circ \varphi_{f*}.$
\end{corollary}

\begin{proof}
We consider a chain homotopy $h_{g \circ f} \sim h_f \#_{\rho} h_g.$ By the standard Morse theory, we have $\varphi_{g \circ f *}= \varphi_{h_f \#_{\rho} h_g*}.$ Then $\varphi_{g \circ f *} = \varphi_{g*} \circ \varphi_{f*}$ follows from Corollary \ref{corgt}.
\end{proof}

\begin{definition}\label{def1htp}
We call the following collection \textit{1-homotopy data} and denote it by $\mathfrak{H}^1$.
\begin{itemize}[label = \textbullet]
\item $\bigl\{ (M_a, h_a, g_a)\bigr\}_{a \in Ob(\mathcal{I})},$ a compact exhaustion $\mathcal{E}$ of $W,$
\item $\bigl\{(h_f, g_f)\bigr\}_{f \in Mor(\mathcal{I})},$ 1-homotopies,
\item $\{\rho_{f,g}>\rho_{f,g,0} \},$ gluing parameters for each $(f,g) \in Mor(\mathcal{I})^{\times 2}$ with $tf=sg$.
\end{itemize}
\begin{definition}\label{strict1h}
We say that 1-homotopy data are {\textit{strict}} if the continuation maps satisfy
\begin{equation}\nonumber
\varphi_{f_k \circ f_{k-1} \circ \cdots \circ f_{1}} = \varphi_{f_{k}} \circ \cdots \circ \varphi_{f_1}.
\end{equation}
for all $(f_1, \cdots, f_{k}) \in N(\mathcal{I})_k$ and $k \geq 1.$
\end{definition}

\end{definition}

Observe that Corollary \ref{cordire} says that the data $ \big( \{MH_*(h_a)\}_{a \in \mathbb{Z}_{\geq 0}}, \{\varphi_{f*}\}_{f \in Mor(\mathcal{I})} \big)$ is a direct system of $\mathbb{Z}_2$-modules. 

\begin{definition}\label{defmh}
We define the \textit{Morse homology} of a pair $(W, \mathfrak{H}^1)$ by the direct limit
\begin{equation}\nonumber
MH_*(W, \mathfrak{H}^1) : = \lim\limits_{\longrightarrow} MH_*(h_a).
\end{equation}
\end{definition}

\subsubsection*{Morphisms of 1-homotopy data} Let $\mathfrak{H}^1, {\mathfrak{H}'}^1 \supset \mathcal{E}$ be two 1-data with the same exhaustion, $\mathcal{E} = \bigl\{(M_a, h_a, g_a)\bigr\}_{a \in \mathbb{Z}_{\geq 0}}.$ By a {\textit{morphism of 1-homotopy data}}, we mean a family of pairs $\bigl\{(\widehat{h}_{a,b}, \widehat{g}_{a,b})\bigr\}_{a \leq b}$ that consists of homotopies connecting $(h_a, g_a)$ and $(h_b, g_b),$ satisfying suitable conditions analogous to the axioms in Definition \ref{1htaxo}. Then, by the same process as for $\{\varphi_f\}_{f \in Mor(\mathcal{I})},$ and $\mathfrak{H}^1$ (or ${\mathfrak{H}'}^1$) described earlier in this subsection, we obtain a family of chain maps.
\begin{equation}\nonumber
\bigl\{ \widehat{\varphi}_{a,b} : MC_*(h_a, g_a) \rightarrow MC_*(h_b, g_b)\bigr\}_{a \leq b},
\end{equation}
and the homology-level maps,
\begin{equation}\nonumber
\bigl\{ \widehat{\varphi}_{a,b*} : MH_*(h_a, g_a) \rightarrow MH_*(h_b, g_b)\bigr\}_{a \leq b}.
\end{equation}

\begin{proposition}\label{1huptoisom}
$MH_*(W, \mathfrak{H}^1)$ is independent of the choice of 1-homotopy data $\mathfrak{H}^1$ up to isomorphism.
\end{proposition}

\begin{proof}
Recall that in Section 2 we defined the strict functor $H\mathscr{F}_{\mathfrak{H}^1} : {\mathcal{I}} \rightarrow Ch(\mathbb{Z}_2)$ by $H\mathscr{F}_{\mathfrak{H}^1}(a) = MH(h_a)$ for each $a \in Ob(\mathcal{I}),$ and $H\mathscr{F}_{\mathfrak{H}^1}(a \rightarrow b) = \varphi_{(a \rightarrow b)*}$ for each pair $a \leq b.$ We now consider another pair of strict functors
\begin{equation}\nonumber
\widehat{H\mathscr{F}}, \ \overset{\circ}{H\mathscr{F}}_{\mathfrak{H}^1} : {\mathcal{I}} \rightarrow Ch(\mathbb{Z}_2)
\end{equation}
given by
\begin{equation}\nonumber
\begin{split}
\widehat{H\mathscr{F}}(a) =  
\begin{cases}
MH_*(h_c) \text{ if } a = 2c,\\
MH_*(h'_c) \text{ if } a = 2c + 1,\\
\end{cases}
\overset{\circ}{H\mathscr{F}}_{\mathfrak{H}^1}(a) = 
\begin{cases}
MH_*(h_c) \text{ if } a = 2c,\\
MH_*(h_c) \text{ if } a = 2c + 1,\\
\end{cases}
\end{split}
\end{equation}
and
\begin{equation}\nonumber
\widehat{H\mathscr{F}}(a \rightarrow b) =
\begin{cases}
\varphi_{(a \rightarrow b)*} \text{ if } &(a,b) = (2c, 2c'),\\
\varphi'_{(a \rightarrow b)*} \text{ if } &(a,b) = (2c+1, 2c'+1),\\
\widehat{\varphi}_{a,b*} \ \quad \text{ if } &(a,b) = (2c, 2c'+1) \text{ or } (2c+1, 2c'), \\
\end{cases}
\end{equation}
\begin{equation}\nonumber
\overset{\circ}{H\mathscr{F}}_{\mathfrak{H}^1}(a \rightarrow b) =
\begin{cases}
\varphi_{(c \rightarrow c')*}, \text{ where } (a,b) \text{ is one of the 4 cases: } \\ 
\quad \quad  (2c, 2c'), (2c+1, 2c'+1), (2c, 2c'+1), (2c+1, 2c'),
\end{cases}
\end{equation}
respectively. Then there is a natural transformation $\mathcal{T} : \overset{\circ}{H\mathscr{F}}_{\mathfrak{H}^1} \Rightarrow \widehat{H\mathscr{F}}$ given by
\begin{equation}\nonumber
\mathcal{T}(a) = 
\begin{cases}
id_{MH_*(h_c)} &\text{ if } a = 2c,\\
\widehat{\varphi}_{c,c*} &\text{ if } a = 2c + 1.
\end{cases}
\end{equation}
The fact that $\widehat{\varphi}_{c,c*}$ is an isomorphism implies that $\mathcal{T}$ is a natural isomorphism. Observe that we have the following maps,
\begin{equation}\nonumber
 \lim\limits_{\longrightarrow} \widehat{H\mathscr{F}}  \xrightarrow{\simeq} \lim\limits_{\longrightarrow}\overset{\circ}{H_*\mathscr{F}}_{\mathfrak{H}^1}
\end{equation}
given by the universal properties of direct limits. 

On the other hand, it is easy to check that 
$\lim\limits_{\longrightarrow}H\mathscr{F}_{\mathfrak{H}^1} \simeq \lim\limits_{\longrightarrow}\overset{\circ}{H\mathscr{F}}_{\mathfrak{H}^1},$
which impies
\begin{equation}\nonumber
\lim\limits_{\longrightarrow}H\mathscr{F}_{\mathfrak{H}^1} \simeq \lim\limits_{\longrightarrow}\widehat{{H_*\mathscr{F}}}_{\mathfrak{H}^1},
\end{equation}
Doing the same thing with ${\mathfrak{H}'}^1,$ we conclude that 
\begin{equation}\nonumber
\lim\limits_{\longrightarrow}H\mathscr{F}_{{\mathfrak{H}'}^1} \simeq \lim\limits_{\longrightarrow}\widehat{H\mathscr{F}} \simeq \lim\limits_{\longrightarrow}H\mathscr{F}_{{\mathfrak{H}}^1},
\end{equation}
which finishes the proof.
\end{proof}

\section{Parametrized homotopy data}

In this section, we generalize the notion of 1-homotopies to higher degree levels.

\subsection{Parametrized homotopies}

Suppose that 1-homotopy data $\mathfrak{H}^1$ are given on a noncompact manifold $W.$ To each simplex $\sigma_k \in N(\mathcal{I})_k$ with $k \geq 1,$ we assign higher degree homotopy data.
\begin{definition}\label{khtpy}
For $k \geq 1,$ we consider a pair of maps:
\end{definition}
\begin{equation}\nonumber
\begin{cases}
H_{\sigma_k} : [0,1]^{k-1} \times {\mathbb{R}}  \times M_{t \sigma_k} \rightarrow \mathbb{R},\\
G_{\sigma_k} : [0,1]^{k-1} \times {\mathbb{R}} \rightarrow Met(M_{t \sigma_k}).
\end{cases}
\end{equation}
(As before, $Met(M_{\bullet})$ denotes the set of all Riemannian metrics on $M_{\bullet}$.) We call a family $\bigl\{(H_{\sigma_k}, G_{\sigma_k})\bigr\}_{\sigma_k} \in N(\mathcal{I})_k$ $k$\textit{-homotopies} if they satisfy the following axioms.
\begin{itemize}[label=\textbullet]
\item (Compatibility with 1-homotopy data) If $k=1,$ they coincide with the pairs from $\mathfrak{H}^1.$
\item (Stability at the ends) There exists $R_{\sigma_k} >0$ such that
\begin{equation}\label{hsigmar1}
H_{\sigma_k}(\vec{t}, \cdot ,s) = \begin{cases}
\overline{h}^{t \sigma_k}_{s\sigma_k}(\cdot) \text{ if } s < - R_{\sigma_k},\\
h_{t\sigma_k}(\cdot) \text{ if } s > R_{\sigma_k},
\end{cases}
\end{equation} 
\begin{equation}\label{hsigmar2}
G_{\sigma_k}(\vec{t}, s) = \begin{cases}
\overline{g}^{t \sigma_k}_{s\sigma_k} \text{ if } s < - R_{\sigma_k},\\
g_{t\sigma_k} \text{ if } s > R_{\sigma_k},
\end{cases}
\end{equation}
\item (Extensions) For all $l \geq t \sigma_k,$ we have an \textit{extension} $(\overline{H}^l_{\sigma_k}, \overline{G}^l_{\sigma_k}).$
\item (Trivial extensions) $(\overline{H}^{t \sigma_k}_{\sigma_k}, \overline{G}^{t \sigma_k}_{\sigma_k}) = (H_{\sigma_k}, G_{\sigma_k}).$
\item (Double extensions) $(\overline{\overline{H}^{l}_{\sigma_k}}^{l'}, \overline{\overline{G}^{l}_{\sigma_k}}^{l'}) = (\overline{H}^{l'}_{\sigma_k}, \overline{G}^{l'}_{\sigma_k}), \text{ for each pair } l \leq l'.$
\item (Stability near the faces) $({H}_{\sigma_k},{G}_{\sigma_k})$ is constant sufficiently near the faces of $[0,1]^{k-1}$ along the normal directions.
\item  (Stability for extensions) For the same ${R}_{\sigma_k} >0$ as above, we put analogous conditions to extensions as (\ref{hsigmar1}) and (\ref{hsigmar2}).
\item (Inward direction condition) $- \nabla_{G_{\sigma_k}} H_{\sigma_k}|_{\partial M_{t \sigma_k}}$ and $- \nabla_{\overline{G}^l_{\sigma_k}} \overline{H}^l_{\sigma_k}|_{\partial M_{l}}$ are in the inward direction, for $ (\vec{t}, s) \in  [0,1]^{k-1} \times {\mathbb{R}}.$ and all $l \geq t\sigma_k.$ 
\item (Regularity) $H_{\sigma_k}$ and $\overline{H}^l_{\sigma_k}$ are regular;  $H_{\bullet}$ for any $p$ with $\nabla_{G_{}^{\vec{t}}}H_{\bullet}^{\vec{t}}(s, p) = 0, \text{ for all } s \in \overline{\mathbb{R}},$ then the operator $\frac{\partial}{\partial s} + H^2(H_{\bullet}^{\vec{t}})(p,\cdot):  W^{1,2}(T_pM_{\bullet}) \rightarrow L^2(T_pM_{\bullet})$ is onto, where $ H^2(H_{\bullet}^{\vec{t}})(p,\cdot)$ is the Hessian of $H_{\bullet}^{\vec{t}}$ at $p \in M_{\bullet}.$ (Here $H_{\bullet}$ denotes either $H_{\sigma_k}$ or $\overline{H}^l_{\sigma_k}$.)

\end{itemize}

\begin{notation}
We write $(H^{\vec{t}}_{\sigma_k}, G^{\vec{t}}_{\sigma_k})$ for $\big(H_{\sigma_k}(\vec{t}, \cdot), G_{\sigma_k}(\vec{t}, \cdot)\big).$
\end{notation}

\begin{notation}\label{not1st}
More generally, let us call any pair of $[0,1]^{k-1} \times \mathbb{R}$-family of Morse functions and Riemannian metrics a $k$-homotopy if it is subject to the above axioms, and not necessarily indexed by the simplicial set $N(\mathcal{I})$.  
\end{notation}

\subsubsection*{The moduli space of trajectories} For a given $k$-homotopy $\ (H_{\sigma_k}, G_{\sigma_k}),$ and critical points $x \in Crit(h_{s\sigma_k}),$ $y \in Crit(h_{t \sigma_k}),$ we consider the set
\begin{equation}\nonumber
\begin{split}
\mathcal{M}(H_{\sigma_k}; x,y) := \Bigl\{ (\vec{t}, u) \mid \vec{t} \in & [0,1]^{k-1}, \ u : \overline{\mathbb{R}} \rightarrow M_{t \sigma_k}\\
&\dot{u} + \frac{\nabla_{G_{\sigma_k}^{\vec{t}}}H_{\sigma_k}^{\vec{t}}}{\sqrt{ 1+ |\dot{H}_{\sigma_k}^{\vec{t}}|^2 |\nabla_{G_{\sigma_k}^{\vec{t}}}H_{\sigma_k}^{\vec{t}}|^2 }} \circ u = 0 \Bigr\}.
\end{split}
\end{equation} 
In other words, $\mathcal{M}(H_{\sigma_k}; x,y)$ is the zero set of the following map between Banach manifolds:
\begin{equation}\label{eqnban}
\begin{split}
F_{\sigma_k} : [0,1]^{k-1} \times \mathcal{P}^{1,2}(x,y)  \longrightarrow \bigcup\limits_{\gamma \in \mathcal{P}^{1,2}(x,y)} L^2(\gamma^* TM_{t\sigma_k}),\\
F_{\sigma_k}(\vec{t}, u) := \dot{u} + \frac{\nabla_{{G_{\sigma_k}^{\vec{t}}}} H_{\sigma_k}^{\vec{t}}}{\sqrt{ 1+ |\dot{H}_{\sigma_k}^{\vec{t}}|^2 |{\nabla_{G^{\vec{t}}_{\sigma_k}} H^{\vec{t}}_{\sigma_k}}|^2 }} \circ u.
\end{split}
\end{equation}
We will denote the second term of the right hand side of (\ref{eqnban}) by $\widetilde{H} \circ u.$

\subsubsection*{Local description} Recall that the Banach manifold $\mathcal{P}^{1,2}(x,y)$ can be covered by local coordinates $\bigl\{W^{1,2}(w^* \mathcal{O}), \exp_w\bigr\}_{w \in C^{\infty}_{x,y}({\overline{\mathbb{R}}}, M_{t \sigma_k})}.$ Using this, we can locally describe $F_{\sigma_k}$ as
\begin{equation}\nonumber
\begin{split}
F_{\sigma_k}^{loc}: [0,1]^{k-1} \times W^{1,2}(w^* \mathcal{O}) \longrightarrow L^2(w^* TM_{t \sigma_k}),
\\
(\vec{t}, \xi) \longmapsto \nabla_2 \exp_h \big(F_{\sigma_k}^{\vec{t}}\big(\exp_w(\xi)\big)\big).
\end{split}
\end{equation}

\subsection{Linearized operators}

The linearization of $F_{\sigma_k}$ can be computed in a way similar to that of the unparametrized case. The only difference is that now the operator depends explicitly on the parameter $\vec{t},$ so that the derivatives in the directions of $\frac{\partial}{\partial{t_i}}$, $i = 1, \cdots, k-1$ appear as well.
\begin{equation}\nonumber
\begin{split}
dF_{\sigma_k} : \mathbb{R}^{k-1} \times W^{1,2}(\mathbb{R}, \mathbb{R}^n) \rightarrow L^2(\mathbb{R}, \mathbb{R}^n)\\
(\tau, \xi) \mapsto \sum\limits_{i} dF_{1,i}(\vec{t}, u ) \cdot \tau_i + dF_2(\vec{t}, u )(\vec{\tau}, \xi),
\end{split}
\end{equation}
where 
\begin{equation}\label{df1df2}
\begin{split}
dF_{1,i}(\vec{t}, u  ) & = \frac{\partial}{\partial t_i} \widetilde{H} \circ u,\\
\big( dF_2(\vec{t}, u ) \big) (\vec{\tau}, \xi) & = \frac{\partial}{\partial s} \xi + A(\vec{t})(\xi), \text{ for } i =1, \cdots, k-1.
\end{split}
\end{equation}
Here $dF_2(\vec{t}, u)$ is written similarly to (\ref{aaaaa}). In particular, $A(\vec{t})$ is an analogous operator to $A$ in (\ref{aaaaa}), but with the additional parametrization.

\subsubsection*{Fredholm theory} For a given $k$-homotopy $(H_{\sigma_k}, G_{\sigma_k}),$ with $\sigma_k \in N(\mathcal{I})_k, \ k \geq 2,$  we have $(k-1)$-homotopies when restricted to the faces of the cube $[0,1]^{k-1}.$
\begin{equation}\nonumber
\begin{split}
 F^{i, +}_{\sigma_k} := & F_{\sigma_k}|_{[0,1]^{k-1}|_{t_i=0}},\\  
 F^{i, -}_{\sigma_k} := & F_{\sigma_k}|_{[0,1]^{k-1}|_{t_i=1}}\\
& : [0,1]^{k-2} \times \mathcal{P}^{1,2}(x,y) \rightarrow \bigcup_{\gamma \in \mathcal{P}^{1,2}(x,y)} L^2(\gamma^* TM_{t\sigma_k}), \\
& \text{ for } i = 1, \cdots, k-1.
\end{split}
\end{equation}
\begin{assumption}\label{assp}
The linearizations
\begin{equation}\nonumber
d F^{i, \pm}_{\sigma_k} (\vec{t}, u) : \mathbb{R}^{k-2} \times W^{1,2}(u^*TM_{t \sigma_k}) \rightarrow L^2(u^*TM_{t \sigma_k})
\end{equation}
are surjective Fredholm operators of the same index $|x| - |y| +k -2$ for all $i.$
\end{assumption}

\begin{proposition}
Under Assumption \ref{assp}, $dF_{\sigma_k}(\vec{t}, u)$ is a Fredholm operator of index $|x| - |y| + k-1.$
\end{proposition}

\begin{proof}
We have for all $i,$
\begin{equation}\nonumber
\begin{split}
\text{ind}\big( dF_{\sigma_k}(\vec{t},u) \big) & = \text{ind} \big( dF_{\sigma_k}\big( (t_i=0), u \big) \big) \text{ (by continuity of the index)}\\
&  = \ker \big( dF_{\sigma_k}\big( (t_i=0), u \big) \big)  \text{ (by surjectivity of } d F^{i, \pm}_{\sigma_k}(\vec{t}', u) )\\
& = 1 + \ker \big( d F^{i, +}_{\sigma_k}(\vec{t}', u) \big) \text{ (by stability near the face) }\\
& = |x| - |y| +k-1.
\end{split}
\end{equation}
On the other hand, $dF_{\sigma_k}\big( (t_i=0), u \big)$ is Fredholm, so $\text{dim} \ \text{coker} \ \big( dF_{\sigma_{k}} (\vec{t}, u) \big)$ is finite. Then it follows that $\text{dim} \ker \big( dF_{\sigma_{k}} (\vec{t}, u) \big) = \text{ind} \big( dF_{\sigma_{k}} (\vec{t}, u) \big) - \ \text{dim} \ \text{coker}\big(  dF_{\sigma_{k}} (\vec{t}, u)   \big)$ is also finite, and the Fredholm property of $dF_{\sigma_k}(\vec{t},u)$ follows.
\end{proof}

\begin{remark}\label{rmk56}
In Section 9, we construct a special class of higher homotopies in an inductive way. This process naturally allows Assumption \ref{assp} to hold, so we obtain the corresponding Fredholm property.
\end{remark}

\subsubsection*{Transversality} We consider the space of maps $Ban \big( [0,1]^{k-1} \times {\mathbb{R}}, C^e(\mathcal{T}_{g_{t \sigma_k}}) \big), $ equipped with the natural Banach structure with respect to the norm $\|\mathcal{A}\|:= \sup\limits_{\|(s, \vec{t})\| = 1} \|\mathcal{A}(s, \vec{t})\|$ for $\mathcal{A} \in Ban \big( [0,1]^{k-1} \times {\mathbb{R}}, C^e(\mathcal{T}_{g_{t \sigma_k}}) \big).$ For a $k$-homotopy $(H_{\sigma_k}, G_{\sigma_k})$ for $\sigma_k \in N(\mathcal{I})_k,$ we define the following map of Banach manifolds:
\begin{equation}\nonumber
\begin{split}
\Phi : Ban \big( [0,1]^{k-1} \times {\mathbb{R}}, & \ C^e(\mathcal{T}_{g_{t \sigma_k}}) \big) \times [0,1]^{k-1}\times \mathcal{P}^{1,2}(x,y)  \longrightarrow  \bigcup\limits_{\gamma \in \mathcal{P}^{1,2}(x,y)} (\gamma^* TM_{t \sigma_k}),\\
(\mathcal{A}, \vec{t}, u) &\longmapsto \dot{u} + \frac{\mathcal{A}(s,\vec{t})\Big( \nabla_{G^{\vec{t}}_{\sigma_k} }H_{\sigma_k}^{\vec{t}} \Big)}{\sqrt{ 1 + |\dot{H}_{\sigma_k}^{\vec{t}}|^2 |\mathcal{A}(s, \vec{t}) \Big(  \nabla_{G^{\vec{t}}_{\sigma_k}}H_{\sigma_k}^{\vec{t}} \Big) |^2}} \circ u.
\end{split}
\end{equation}

We restrict $\Phi$ to the local coordinate of $u \in \mathcal{P}^{1,2}(x,y)$ as before to have

\begin{equation}\nonumber
\Phi^{loc} : Ban \big([0,1]^{k-1} \times {\mathbb{R}}, C^e(\mathcal{T}_{g_{t \sigma_k}}) \big) \times [0,1]^{k-1} \times W^{1,2}(u^* \mathcal{O}) \longrightarrow L^2(u^* TM_{t \sigma_k}).
\end{equation}

\begin{proposition}\label{transprop}
0 is a regular value of $\Phi$ and $\Phi_{\mathcal{A}} := \Phi^{loc}(\mathcal{A},\cdot)$ is a Fredholm operator of index $|x| -|y| + k-1$ for any $\mathcal{A} \in Ban \big({\mathbb{R}} \times [0,1]^{k-1}, C^e(\mathcal{T}_{g_{t \sigma_k}}) \big).$
\end{proposition}

\begin{proof}
Consider the differential:
\begin{equation}\nonumber
\begin{split}
d \Phi(\mathcal{A}, \xi) : T_{\mathcal{A}} Ban \big(&[0,1]^{k-1} \times  \mathbb{R}, C^e(\mathcal{T}_{g_{t \sigma_k}}) \big) \times W^{1,2}(u^*TM_{t \sigma_k}) \rightarrow L^2(u^* TM_{t \sigma_k}),\\
& d\Phi(\mathcal{A}, \xi) (B, \eta) :=  d_1 \Phi(\mathcal{A}, \xi)(B) + d_2 \Phi(\mathcal{A}, \xi)(\eta).
\end{split}
\end{equation}
We already know about the second part $d_2 \Phi(\mathcal{A}, \xi),$ (it is Fredholm, etc.). Since $\text{im}\big( d_2 \Phi(\mathcal{A}, \xi) \big)$ is of finite codimension, so is $d \Phi(\mathcal{A}, \xi).$

Let $c \in \text{im}\big( d_2 \Phi(\mathcal{A}, \xi) \big)^{\perp} \subset L^2(u^* TM_{t \sigma_k})$ be a vector that satisfies 
\begin{equation}\label{dphic}
\langle d \Phi(\mathcal{A}, \xi) (B, \eta), c \rangle_{L^2} = 0
\end{equation}
for \textit{all} $(B, \eta) \in \mathcal{A}^* C^e(E_{g_{t \sigma_k}}) \times W^{1,2}(u^* TM_{t \sigma_k}).$ To have surjectivity of $d \Phi (\mathcal{A}, \xi),$ it suffices to show that $c \equiv 0.$ Notice that (\ref{dphic}) implies that 

\begin{equation}\label{dphic2}
\langle d_1 \Phi(\mathcal{A}, \xi) (B), c \rangle_{L^2} = 0,
\end{equation}
for all $B \in T_{\mathcal{A}} Ban \big( [0,1]^{k-1} \times \mathbb{R}, C^e (\mathcal{T}_{g_{t\sigma_k}}) \big) \simeq \mathcal{A}^* \big( T C^e ( \mathcal{T}_{g_{t\sigma_k}}) \big) \simeq \mathcal{A}^*\big( C^e(E_{g t\sigma_k}) \big).$ In particular, this holds for $B$ such that  
\begin{enumerate}[label=(\roman*)]
\item The fibers $B_{\mathcal{A}(s,\vec{t})}$ are independent of $(s,\vec{t})$; they form a constant section,
\item $B_{\mathcal{A}(s,\vec{t})} \in L^2(E_{g_{t \sigma_k}}),$ which is possible by Lemma \ref{l2lem}.
\item $B_{\mathcal{A}(s,\vec{t})}$ has a sufficiently small support near some $s_0 \in \mathbb{R}$ with $|s_0| < R_{\sigma_k}.$
\end{enumerate}

So (\ref{dphic2}) implies the following local expression at $s \in \mathbb{R},$
\begin{equation}\nonumber
\langle d_1\Phi(\mathcal{A}, \xi) \big(B(s_0, \cdot) \big), c(s_0) \rangle =0
\end{equation}
at $T_pM_{t \sigma_k}$ with $p = \exp_{\gamma}\xi(s_0).$ By the choice of $B$ and $s_0,$ this leads to exactly the same situation of [Sch] Proposition 2.30 in the unparametrized case, hence we have $c(s_0)=0.$ Since $c \in \text{im}\big( d_2 \Phi(\mathcal{A}, \xi) \big)^{\perp}$ and $\langle d_2\Phi(\mathcal{A}, \xi)(\eta), c \rangle = 0,$ we know $c$ satisfies an ODE $\dot{c}(s_0)+ X(s_0) c(s_0) = 0,$ for some trivialization. Then the uniqueness for a solution of an ODE implies that $c(s')=0$ for some $s' \in \mathbb{R}$ if and only if $c \equiv 0.$ Thus, we have $c \equiv 0$ and conclude that $d \Phi(\mathcal{A}, \xi)$ is surjective at $(\mathcal{A}, \xi) \in \Phi^{-1}(0).$
\end{proof}

The following theorem and corollary are due to Theorem \ref{schtransv} and Proposition \ref{transprop}. We make a remark for the corollary. When we try to get a homotopy of Riemannian metrics $G_{\sigma_k},$ for transversality it is possible to choose one, so that the pair $(H_{\sigma_k}, G_{\sigma_k})$ satisfies the axioms of Definition \ref{khtpy}. The stability-at-the-end axiom and the inward direction axiom are to be checked. The rest of them are almost trivial. The former can be achieved by (iii) in the proof of Proposition \ref{transprop}, i.e., the surjectivity of the linearization follows by a local consideration of metrics. The latter is also ensured by a local change of the metric near the boundary, which is always possible.

\begin{theorem}\label{paramtrasv}
There is a generic choice of $\mathcal{A} \in Ban \big([0,1]^{k-1} \times {\mathbb{R}}, C^e(\mathcal{T}_{g,_{t \sigma_k}}) \big)$ such that $\Phi_{\mathcal{A}}^{-1}(0)$ is a smooth manifold (with corners) of dimension $|x| - |y| + k-1.$
\end{theorem}

\begin{corollary}\label{givenk}
Given a parametrized Morse function $H_{\sigma_k},$ there is a generic choice of parametrized Riemannian metrics $G_{\sigma_k}$ such that $\mathcal{M}(H_{\sigma_k}; x, y)$ is a smooth manifold of dimension $|x| - |y| +k-1.$\end{corollary}

\begin{remark}
We will see in Section 11 that there exists a special type of parametrized homotopy data $\bigl\{(\mathcal{H}_{\sigma}, \mathcal{G}_{\sigma})\bigr\}$ with respect to which we obtain some algebraic structure.
\end{remark}

For a given pair $(H_{\sigma_k}, G_{\sigma_k})$ with $\sigma_k \in N(\mathcal{I})_k, \ k \geq 1$ which is chosen as in
Corollary \ref{givenk}, we consider the following linear map

\begin{equation}\nonumber
\begin{split}
\overline{\varphi}_{{\sigma_k}} : MC_*(M_{t\sigma_k}, \overline{h}^{t\sigma_k}_{s\sigma_k}) \rightarrow MC_*(M_{t\sigma_k}, h_{t\sigma_k}),\\
x \mapsto \sum\limits_{|x| + |y| +k -1 =0} \#_2 \mathcal{M}(H_{\sigma_k}; x, y) \cdot y,
\end{split}
\end{equation}
which is well-defined again by Corollary \ref{givenk}.

We also consider the obvious inclusion map:
\begin{equation}\nonumber
\iota_{\sigma_k} : MC_*(h_{s \sigma_k}) \hookrightarrow MC_*(\overline{h}^{t\sigma_k}_{s\sigma_k}).
\end{equation}

\begin{lemma}\label{lem2} We have
\begin{enumerate}[label=(\roman*)]
\item $\iota_{\sigma_k} = \iota_{\partial_i \sigma_k}$ for all $i,$
\item $\iota_{\sigma_k}$ is a chain map.
\end{enumerate}
\end{lemma}
\begin{proof}
Both (i) and (ii) follow from the definition of $\iota_{\bullet}$ and the inward direction condition put on $- \nabla_{g_{s\sigma_k}} h_{s\sigma_k}.$
\end{proof}

\begin{definition}

The \textit{higher continuation map}
\begin{equation}\nonumber
\varphi_{\sigma_k} :  MC_*(M_{s\sigma_k}, {h}_{s\sigma_k})  \rightarrow MC_*(M_{t\sigma_k}, h_{t\sigma_k}),
\end{equation}
associated to a pair $({H}_{\sigma_k}, G_{\sigma_k})$ is defined by the composition : $\varphi_{\sigma_k} : = \overline{\varphi}_{\sigma_k} \circ \iota_{\sigma_k}.$
\end{definition}

\section{Parametrized gluing constructions}

In this section, we begin our discussion of gluing constructions with extra parameters. We consider corresponding concatenations, pregluings, and their linear versions.

\subsection{Concatenations of homotopies}

Suppose that we are given two parametrized Morse homotopies $(H_{\sigma_k}, G_{\sigma_k})$ and $(H_{\sigma_l}, G_{\sigma_l})$ with $t \sigma_k= s \sigma_l.$ We define their \textit{concatenation} for $\rho> \max \{ R_{\sigma_k}, R_{\sigma_l} \},$
\begin{equation}\nonumber
\begin{cases}
H_{\sigma_k} \#_{\rho} H_{\sigma_l} : [0,1]^{k+l-2} \times \mathbb{R} \times M_{t\sigma_l} \rightarrow \mathbb{R},\\
G_{\sigma_k} \#_{\rho} G_{\sigma_l} :  [0,1]^{k+l-2} \times \mathbb{R} \rightarrow Met(M_{t \sigma_l}),
\end{cases}
\end{equation}
by
\begin{equation}\nonumber
\begin{split}
H_{\sigma_k} \#_{\rho} H_{\sigma_l}(t_1, \cdots, t_{k-1}, t'_1, & \cdots, t'_{l-1}, s, \cdot)\\ 
&:=
\begin{cases}
\overline{H}^{t \sigma_l}_{\sigma_k} (t_1, \cdots, t_{k-1}, s + \rho, \cdot) & \text{ if } s < -\rho,\\
\overline{h}^{t \sigma_l}_{t \sigma_k}(\cdot) & \text{ if } -\rho \leq s \leq \rho,\\
H_{\sigma_l} (t'_1, \cdots, t'_{l-1}, s -\rho, \cdot) & \text{ if } s > \rho.\\
\end{cases}
\end{split}
\end{equation}
and
\begin{equation}\nonumber
\begin{split}
G_{\sigma_k} \#_{\rho} G_{\sigma_l}(t_1, \cdots, t_{k-1}, t'_1, & \cdots, t'_{l-1}, s)\\ 
&:=
\begin{cases}
\overline{G}^{t \sigma_l}_{\sigma_k} (t_1, \cdots, t_{k-1}, s + \rho) & \text{ if } s < -\rho,\\
\overline{g}^{t \sigma_l}_{t \sigma_k} & \text{ if } -\rho \leq s \leq \rho,\\
G_{\sigma_l} (t'_1, \cdots, t'_{l-1}, s -\rho, ) & \text{ if } s > \rho.\\
\end{cases}
\end{split}
\end{equation}
(Here $Met(M_{t \sigma_l})$ denotes the set of all Riemannian metrics on $M_{t \sigma_l}.$) These are well-defined by the stability axiom of Definition \ref{khtpy}.

For $m \geq t\sigma_l,$ we require the extension to be given as follows:
\begin{equation}\nonumber
\begin{cases}
\overline{(H_{\sigma_k} \#_{\rho} H_{\sigma_l})}^m := \overline{H_{\sigma_k}}^m \#_{\rho} \overline{H_{\sigma_l}}^m,\\
\overline{(G_{\sigma_k} \#_{\rho} G_{\sigma_l})}^m := \overline{G_{\sigma_k}}^m \#_{\rho} \overline{G_{\sigma_l}}^m.
\end{cases}
\end{equation}

Observe that $H_{\sigma_k} \#_{\rho} H_{\sigma_l}$ itself is a parametrized homotopy. That is, it satisfies the axioms of Definition \ref{khtpy}.
\begin{lemma}
$( H_{\sigma_k} \#_{\rho} H_{\sigma_l}, G_{\sigma_k} \#_{\rho} G_{\sigma_l} ) $ is a $(k+l-1)$-homotopy (in the sense of Notation \ref{not1st}).
\end{lemma}

\begin{notation}
$\overline{H}^{t\sigma_l}_{\sigma_k} \#_{\rho} H_{\sigma_l}$ may be a more appropriate notation than $H_{\sigma_k} \#_{\rho} H_{\sigma_l},$ but we keep it for simplicity whenever no confusion arises.
\end{notation}

\subsection{Pregluings and its linear versions}

We introduce our notations:

\begin{notation}\label{not1to5}
\begin{enumerate}[label = (\roman*)]

\item We write $u_{\rho}(\cdot)$ for $u(\cdot + \rho) $ and similarly for $v_{-\rho}.$

\item $\alpha_u \in C^{\infty}(\overline{\mathbb{R}}, M_{t \sigma_k}), \ \alpha_v \in C^{\infty}(\overline{\mathbb{R}}, M_{t \sigma_l})$ are smooth maps that satisfy $\alpha_u(-\infty) = x, \ \alpha_u(\infty) = y = \alpha_v(-\infty), \ \alpha_v(\infty) = z,$ and $u$ and $v$ are in the Banach coordinate near $\alpha_u$ and $\alpha_v,$ that is, they are in the images of the exponential maps evaluated at $\alpha_u$ and $\alpha_v,$ respectively, for some vector fields. $\alpha_u$ and $\alpha_v$ can be chosen so that they stabilize near the ends, where $\alpha_u \#_{\rho} \alpha_v$ denotes the concatenation that is well-defined and smooth by this stabilization condition. 

\item  $\beta^{-}, \ \beta^{+} : \mathbb{R} \rightarrow [0,1]$ are smooth cut-off functions given by 
\begin{equation}\nonumber
\begin{split}
\beta^{-}(s) =  
\begin{cases}
1 \text{ if } s < 0,\\
0 \text{ if } s > 1,\\
\end{cases}
\beta^{+}(s) =  
\begin{cases}
0 \text{ if } s < 0,\\
1 \text{ if }s > 1.\\
\end{cases}
\end{split}
\end{equation}
Let $\beta : \mathbb{R} \rightarrow [0,1]$ be a smooth cut-off function given by
\begin{equation}\nonumber
\beta(s) = \begin{cases}
1 &\text{ if } |s| \leq \frac{1}{2},\\
0 &\text{ if } |s| \geq 1.
\end{cases}
\end{equation}
We write $\beta_{\rho}(s)$ for $\beta(s+ \rho).$

\end{enumerate}
\end{notation}

\subsubsection*{Pregluings} For a given $(\vec{t}, u) \in \mathcal{M}(H_{\sigma_k}; x, y), \ (\vec{t}', v) \in \mathcal{M}(H_{\sigma_l}; y, z),$ and for $\rho \geq \max \{ R_{\sigma_k}, R_{\sigma_l} \},$ we define their \textit{pregluing} by
\begin{equation}\nonumber\nonumber
(\vec{t}, u) \#^o_{\rho} (\vec{t}', v) := \big(  \vec{t}, \vec{t'}, u \#_{\rho}^o v \big),
\end{equation}
where $u \#^o_{\rho} v : \overline{\mathbb{R}} \rightarrow M_{t \sigma_l}$ is an element of $\mathcal{P}^{1,2}(x,z)$ for $M_{ t \sigma_l}$ given by
 \begin{equation}\nonumber\nonumber
u \#_{\rho}^o v (s): = 
\begin{cases}
u(s+\rho) & \text{if } s < -\max \{ R_{\sigma_k}, R_{\sigma_l} \},\\
\exp_{\alpha_u \#_{\rho} \alpha_v} \big[\beta^{-}\exp^{-1}_{\alpha_u} (u^{}_{\rho}) + \beta^{+} \exp^{-1}_{\alpha_v}(v_{-\rho})&\big](s) \\ & \text{if } |s| \leq \max \{ R_{\sigma_k}, R_{\sigma_l} \},\\
v (s-\rho) & \text{if } s > \max \{ R_{\sigma_k}, R_{\sigma_l} \}.
\end{cases}
\end{equation}

\subsubsection*{The linear version of pregluings} For $\xi \in W^{1,2}(u^*TM)$ and $ \zeta \in W^{1,2}(v^*TM),$ we define their \textit{linear pregluing} as follows
\begin{equation}\nonumber
(\xi \#_{\rho} \zeta)(s):= \begin{cases}
\xi(s +\rho) &\text{ if } s \leq -1,\\
\nabla_2 \exp_{\exp^{-1}_y(u \#^o_{\rho}v),} [\beta^{-} \nabla_2 \exp_{\exp^{-1}_y(u_{\rho})}^{-1} \xi_{\rho}& \\ \quad \quad + \beta^{+} \nabla_2 \exp_{\exp^{-1}_y(v_{-\rho})}^{-1} \zeta_{-\rho}](s) &\text{ if } -1 \leq s \leq 1,\\
\zeta(s - \rho) &\text{ if } s \geq 1.\\
\end{cases}
\end{equation}
Here $\nabla_2 \exp$ is considered at $\exp^{-1}_y \big(u \#^o_{\rho}v(s)\big),$ $\exp^{-1}_y\big(u(s + \rho)\big),$ and $\exp^{-1}_y\big(v( s -\rho)\big),$ respectively. 

We extend this to the parametrized case:
for $(\vec{\tau}, \xi) \in \mathbb{R}^{k-1} \times W^{1,2}(u^*TM_{t \sigma_k}),$ and $(\vec{\tau}', \zeta) \in \mathbb{R}^{l-1} \times W^{1,2}(v^*TM_{t \sigma_l}),$ we define their \textit{parametrized linear pregluing} as follows:
\begin{equation}\nonumber
(\vec{\tau}, \xi)\#^o_{\rho}(\vec{\tau}', \zeta)(s) := \big(\vec{\tau}, \vec{\tau}',(\xi \#_{\rho} \zeta)(s)\big).
\end{equation}

For $\gamma \in C^{\infty}_{x,z} \subset \mathcal{P}^{1,2}(x,z),$ $(\vec{\tau}_1, \xi_1), (\vec{\tau}_2, \xi_2) \in \mathbb{R}^{k-1} \times W^{1,2}(\gamma^*TM),$ we define the {inner product} between them by
\begin{equation}\nonumber
\langle (\vec{\tau}_1, \xi_1),(\vec{\tau}_2, \xi_2) \rangle^{0,2}_{\gamma} := \langle \xi_1, \xi_2 \rangle_{\gamma}^{0,2} + \vec{\tau}_1 \cdot \vec{\tau}_2,
\end{equation}
where $\langle, \rangle$ is the Riemannian metric on $M.$
Also, we define the {$W^{1,2}$-norm:}
\begin{equation}\nonumber
\langle(\vec{\tau}_1, \xi_1),(\vec{\tau}_2, \xi_2)\rangle_{\gamma}^{1,2} := \langle \nabla_s \xi, \nabla_s \xi \rangle_{\gamma}^{0,2} + \langle (\vec{\tau}_1, \xi_1), (\vec{\tau}_2, \xi_2) \rangle_{\gamma}^{0,2}
\end{equation}
The metrics $\| \cdot \|^{0,2}, \| \cdot \|^{1,2}$ are defined using these norms.

\subsubsection*{Gluing theorem} Let $\sigma_k \in N(\mathcal{I})_k$ and $\sigma_l \in N(\mathcal{I})_l$ be two simplices such that $t \sigma_k = s \sigma_l.$ We assume that we are given $(k-1)$- and $(l-1)$-homotopies, $(H_{\sigma_k}, G_{\sigma_k})$ and $(H_{\sigma_l}, G_{\sigma_l}),$ respectively. For critical points 
\begin{equation}\label{gtst}
x \in Crit(h_{s \sigma_k}) \subset Crit(\overline{h}^{t \sigma_{l}}_{s \sigma_k}), \ \
y \in  Crit(h_{t \sigma_k}) \subset Crit(\overline{h}^{t \sigma_{l}}_{t \sigma_k}), \text{and} \ \
z \in Crit(h_{t \sigma_l})
\end{equation}
with $|x| - |y| +k -1 =0$ and $|y| - |z| + l -1 =0,$ we consider the parametrized moduli space of trajectories $\mathcal{M}(H_{\sigma_k}; x, y) $ and $\mathcal{M}(H_{\sigma_l}; y,z),$ respectively. 

\begin{lemma}
For all $x \in Crit(h_{s\sigma_k}),$ $y \in Crit(h_{t \sigma_k})$, and $l \geq t \sigma_k,$ we have an isomorphism 
\begin{equation}\nonumber
\mathcal{M}(\overline{H}^l_{\sigma_k}; x,y) \simeq \mathcal{M}(H_{\sigma_k}; x, y).
\end{equation}
\end{lemma}

\begin{proof}
The inward direction condition implies the assertion of the lemma, which is similar to Lemma \ref{lem411} 
\end{proof}

We now state one of our main theorems which we will prove in Sections 7 and 8.
\begin{theorem}[Gluing theorem]\label{gluthm}
For each $\rho> \max \{ R_{\sigma_k}, R_{\sigma_l} \},$ there exist $\rho_{\sigma_k, \sigma_l, 0} > 0$ and a map,
\begin{equation}\nonumber
\chi_{\rho} : \mathcal{M}(H_{\sigma_{k}}; x, y) \times \mathcal{M}(H_{\sigma_l}; y, z) \rightarrow \mathcal{M}(H_{\sigma_{k}} \#_{\rho} H_{\sigma_l} ; x,z)
\end{equation}
such that for every $\rho > \rho_{\sigma_k, \sigma_l, 0},$ $\chi_{\rho}$ is a bijection of finite sets.
\end{theorem}

For $(\vec{t}, u) \in \mathcal{M}(H_{\sigma_k}; x, y),$ and $(\vec{t'}, v) \in \mathcal{M}(H_{\sigma_l}; y ,z),$ we consider their pregluing for $\rho > \rho_{\sigma_k, \sigma_l, 0}$ by $(\vec{t}, u) \#^o_{\rho} (\vec{t}', v) : = (\vec{t}, \vec{t}', u \#^o_{\rho}{v}).$ Here we denote $w_{\rho}:= u \#_{\rho}^o v,$ and trivialize the vector bundle $w_{\rho}^*TM_{t \sigma_l}$ over $\mathbb{R}$ as $w_{\rho}^*TM_{t \sigma_l} \simeq \mathbb{R} \times \mathbb{R}^n,$ where the right hand side is the trivial rank $n$ vector bundle over $\mathbb{R}.$ Fixing such a trivialization, we can identify the Banach spaces: 
\begin{equation}\nonumber
W^{1,2}(w_{\rho}^*TM_{t\sigma_l}) \simeq W^{1,2}(\mathbb{R}, \mathbb{R}^n), \ L^{2}(w_{\rho}^*TM_{t \sigma_l}) \simeq L^{2}(\mathbb{R}, \mathbb{R}^n).
\end{equation}  

For a small neighborhood around the zero section
\begin{equation}\nonumber
B \subset W^{1,2}(\mathbb{R}, \mathbb{R}^n) \ \big(\simeq W^{1,2}(w_{\rho}^*TM_{t \sigma_l})\big),
\end{equation}
we define an operator 
\begin{equation}\nonumber
\mathcal{F}_{\rho} : [0,1]^{k+l -2} \times B \rightarrow L^{2}(\mathbb{R}, \mathbb{R}^n) \ \big(\simeq L^2(w_{\rho}^*TM_{t\sigma_l})\big)
\end{equation}
by 
\begin{equation}\nonumber
\mathcal{F}_{\rho}(\vec{t}, \vec{t}', Y) := \Big( \frac{d}{ds} + \nabla_{G^{\vec{t}}_{\sigma_k} \#_{\rho} G^{\vec{t}}_{\sigma_l}(s + \rho)} \big( H^{\vec{t}}_{\sigma_k} \#_{\rho} H^{\vec{t}'}_{\sigma_l}(s+ \rho) \big) \Big)\big(\exp_{w_{\rho}}(Y)\big).
\end{equation}
We write $\mathcal{F}_{\rho}^{(\vec{t}, \vec{t}')}(\cdot) := \mathcal{F}_{\rho}( \vec{t}, \vec{t}', \cdot ),$
\begin{equation}\nonumber
\mathcal{F}_{\rho}^{(\vec{t}, \vec{t}')} : B \subset W^{1,2}(\mathbb{R}, \mathbb{R}^n) \rightarrow L^2(\mathbb{R}, \mathbb{R}^n).
\end{equation}

\subsubsection*{The linearization $d \mathcal{F}_{\rho}$} We consider the linearization of $\mathcal{F}_{\rho}$
\begin{equation}\nonumber
\mathcal{L}_{\rho} : \mathbb{R}^{k+l-2} \times W^{1,2}(\mathbb{R}, \mathbb{R}^n) \rightarrow L^2(\mathbb{R}, \mathbb{R}^n)
\end{equation}
defined by  $\mathcal{L}_{\rho} : = d \mathcal{F}_{\rho}|_{(\vec{t}, \vec{t}', 0)}. $ We also write $\mathcal{L}_{\rho}^{(\vec{t}, \vec{t}')}$ for $ \mathcal{L}_{\rho}(0, \cdot) = d \mathcal{F}_{\rho}^{(\vec{t}, \vec{t}')}|_{0}(\cdot).$
For $\vec{t}'' := (t''_1, \cdots, t''_{k+l-2}), \ \vec{t}''_{\vec{\tau}} := (\tau_1 t_1'', \cdots, \tau_{k+l-2} t''_{k+l-2}),$ we then have
\begin{equation}\nonumber
\begin{split}
\mathcal{L}_{\rho} (\vec{\tau}'', 0) = \sum_{i =1}^{k+l-2} \frac{\partial}{\partial t_i''}\Big|_{\vec{t}''=0} \mathcal{F}_{\rho}&^{(\vec{t}, \vec{t}') + \vec{t}''_{\vec{\tau}}} \big(\exp_{w_{\rho}}(0) \big)\\
 = \sum_{i} \frac{\partial}{\partial t''_i} \Big|_{\vec{t}''=0} \Big( &\frac{\partial w_{\rho}}{\partial s} + \nabla_{...(s+\rho)} \big( H_{\sigma_{k}}^{\vec{t} + (\tau_1 t_1'', \cdots, \tau_{k-1} t''_{k-1} )} \\
 & \quad \quad \#_{\rho} H_{\sigma_l}^{\vec{t}' + (\tau_k t''_k, \cdots, \tau_{k+l-2} t''_{k+l-2}) }(s+\rho) \circ w_{\rho} \big) \Big)\\
 = \sum_{i} \frac{\partial}{\partial t''_i}  \nabla&_{{\partial_{t''_i}} ...(s+\rho)} \big( H_{\sigma_{k}}^{\vec{t} + (\tau_1 t_1'', \cdots, \tau_{k-1} t''_{k-1} )} \\ 
 \quad \quad \#_{\rho} & H_{\sigma_l}^{\vec{t}' + (\tau_k t''_k, \cdots, \tau_{k+l-2} t''_{k+l-2}) }(s+\rho) \circ w_{\rho} \big) \Big|_{\vec{t}''=0}  \\
 + \sum_{i} \nabla_{ ...(s+\rho)} \big( & \frac{\partial}{\partial t''_i} \Big( H_{\sigma_{k}}^{\vec{t} + (\tau_1  t_1'', \cdots, \tau_{k-1} t''_{k-1} )} \\
 \#_{\rho} H_{\sigma_l}&^{\vec{t}' + (\tau_k t''_k, \cdots, \tau_{k+l-2} t''_{k+l-2}) }(s+\rho) \circ w_{\rho} \Big) \big)\Big|_{\vec{t}''=0} = \vec{\tau}'' \cdot {V}_{\rho},
\end{split}
\end{equation}
where $\vec{\tau}'' : = (\vec{\tau}, \vec{\tau}') \in \mathbb{R}^{k+l-2}.$ Applying each $\frac{\partial}{\partial t_i}$ produces a $\tau_i$ factor, yielding the expression $\vec{\tau}'' \cdot V_{\rho}$ for some vector field $V_{\rho},$ which explains the last equality (Notation: we denote $\vec{\tau}'' \cdot V_{\rho} := (\tau_1 \cdot V_{\rho,1} , \cdots,  \tau'_{l-1} \cdot V_{\rho, k+l-2})$). The asymptotic condition at the ends implies that $V_{\rho}(s) = 0$ for $s$ with $s+\rho$ sufficiently small or large.

Since $\mathcal{L}_{\rho}$ is linear, it can be written as follows
\begin{equation}\label{llinear}
\begin{split}
\mathcal{L}_{\rho}(\vec{\tau}'', \cdot) &=  \mathcal{L}_{\rho}(\vec{\tau}'',0)+ \mathcal{L}_{\rho}(0, \cdot)\\
&= \vec{\tau}'' \cdot V_{\rho} +  d \mathcal{F}_{\rho}^{(\vec{t}, \vec{t}')}|_{0}(\cdot) = \vec{\tau}'' \cdot {V}_{\rho} + \mathcal{L}_{\rho}^{(\vec{t}, \vec{t}')}(\cdot).
\end{split}
\end{equation} 

\begin{lemma}\label{lefrop} 
$\mathcal{L}_{\rho}$ is a Fredholm operator of index $|x| - |z| + k+l -2.$
\end{lemma}
\begin{proof}
In (\ref{llinear}), $\mathcal{L}_{\rho}^{(\vec{t}, \vec{t}')}$ is of the form that we already know about. For example, it is Fredholm and its index is $|x| - |z|$. Since coker$\mathcal{L}_{\rho}^{(\vec{t}, \vec{t}')}$ is finite dimensional, so is coker$\mathcal{L}_{\rho}.$ Consider 
\begin{equation}\nonumber
\ker{\mathcal{L}_{\rho}} := \bigl\{(\vec{\tau}'', \xi) \in \mathbb{R}^{k+l-2} \times W^{1,2} (\mathbb{R}, \mathbb{R}^{n}) \mid \vec{\tau}'' \cdot V_{\rho} + \mathcal{L}_{\rho}^{(\vec{t}, \vec{t}')}(\xi) = 0\bigr\}.
\end{equation}
Then there are two cases:
\begin{enumerate}[label = (\roman*)]
\item If $V_{\rho} \notin \text{im}\mathcal{L}_{\rho}^{(\vec{t}, \vec{t}')}$, then we have $\ker \mathcal{L}_{\rho} = \ker \mathcal{L}_{\rho}^{(\vec{t}, \vec{t}')},$ which is finite dimensional.
\item If $V_{\rho} \in \text{im}\mathcal{L}_{\rho}^{(\vec{t}, \vec{t}')},$ then $V_{\rho} = \mathcal{L}_{\rho}^{(\vec{t}, \vec{t}')}(\xi')$ for some $\xi' \in W^{1,2}(\mathbb{R}, \mathbb{R}^{n}),$ and $\mathcal{L}_{\rho}^{(\vec{t}, \vec{t}')} ( \vec{\tau}'' \cdot \xi' + \xi) = 0;$ in this case, we know
\begin{equation}\nonumber
\ker \mathcal{L}_{\rho} = \text{im} \xi + \ker \mathcal{L}_{\rho}^{(\vec{t}, \vec{t}')},
\end{equation}
which is again finite dimensional.
\end{enumerate}
Consider 
\begin{equation}\nonumber
\mathcal{L}^{c}_{\rho}(\vec{\tau}'', \cdot)  := c( \vec{\tau}'' \cdot {V}_{\rho} )+ \mathcal{L}_{\rho}^{(\vec{t}, \vec{t}')}(\cdot)
\end{equation}
for $c \in [0,1].$ Notice that when $c=1,$ it coincides with $\mathcal{L}_{\rho}.$ Then we have
\begin{equation}\nonumber
\text{ind}\mathcal{L}^{c}_{\rho}(\vec{\tau}'', \cdot)  = \text{ind} \mathcal{L}^{0}_{\rho}(\vec{\tau}'', \cdot)  =  \text{ind} \mathcal{L}_{\rho}^{(\vec{t}, \vec{t}')}(\cdot) + k + l -2 = |x| - |z| + k+l -2,
\end{equation}
where the first equality follows from continuity of the index and the second one from the fact that the dimension of the kernel increases by $k+l-2.$ We conclude that $\mathcal{L}_{\rho}$ is a Fredholm operator of index $|x| - |z| + k+l -2.$
\end{proof} 
 
\section{Proof of the gluing theorem I}

Sections 7 and 8 are devoted to the proof of Theorem \ref{gluthm}. Many statements in these sections can be thought of as generalizations of those in [Sch] and [AD]. For instance, Proposition \ref{invest} and Theorem \ref{lsurj} generalize  Proposition 2.69 and Proposition 2.50 of [Sch], respectively in straightforward ways. We always assume the setting  (\ref{gtst}) of Theorem \ref{gluthm} unless otherwise mentioned. We remark that $ \text{[AD]} $ is written for Floer theory, but most of its results are applicable in our case as well. The basic analysis is essentially the same.

We consider a subspace of $W^{1,2}(\mathbb{R}, \mathbb{R}^n)$
\begin{equation}\nonumber
\begin{split}
W_{\rho} : = \bigl\{Y \#_{\rho} Y'  \in W^{1,2}(\mathbb{R},  \mathbb{R}^n) \mid (\vec{\tau}, Y) \in \ker dF_{\sigma_{k}},  \ & (\vec{\tau}', Y') \in \ker dF_{\sigma_l} \\ 
&\text{ for some } \vec{\tau} \in \mathbb{R}^{k-1}, \vec{\tau}' \in \mathbb{R}^{l-1}\  \bigr\}.
\end{split}
\end{equation}
We denote by $W_{\rho}^{\perp}$ the complement of $W_{\rho}$ in $W^{1,2}(\mathbb{R}, \mathbb{R}^n).$ Namely, we set

\begin{equation}\nonumber
W_{\rho}^{\perp} := \bigl\{ Y \in W^{1,2}(\mathbb{R}, \mathbb{R}^n) \mid \int_{\mathbb{R}} \langle Y, Z \rangle ds = 0 \  
\text{for all } Z \in W_{\rho} \bigr\}.
\end{equation}

\begin{proposition}\label{invest}
There exist $\rho_0 >0$ and $C>0$ such that
\begin{equation}\nonumber
\| \mathcal{L}_{\rho} (\vec{\tau}, Y) \|_{L^2} \geq C \big( \|Y\| _{W^{1,2}} + \|\vec{\tau}\| \big),
\end{equation} 
for all $\rho \geq \rho_0,$ $Y \in W_{\rho}^{\perp},$ and $\vec{\tau} \in \mathbb{R}^{k+l-2}.$
\end{proposition}

\begin{proof}
Suppose not, then there exist sequences $\rho_m \rightarrow \infty$ and $(\vec{\tau}_m, Y_m) \in \mathbb{R}^{k+l-2} \times W_{\rho}^{\perp}$ such that
\begin{equation}\nonumber
\begin{split}
& \|\vec{\tau}_m\| + \| Y_m \|_{W^{1,2}} = 1 \ \text{ for all } m,\\ 
& \lim\limits_{m \rightarrow \infty} \| \mathcal{L}_{\rho_{m}} (\vec{\tau}_m, Y_m)\|_{L^2} =0.
\end{split}
\end{equation}
Let $\beta : \mathbb{R} \rightarrow [0,1]$ be a smooth cut-off function of  Notation \ref{not1to5} (iii). We write $r_m $ for $\frac{1}{2}\rho_m$ and $\beta^{r_m}(s)$ for $\beta(\frac{s}{r_m}).$ Recall that we denoted for $y \in Crit(h_{t \sigma_k}),$
\begin{equation}
dF_y : = \frac{\partial}{\partial s} + A_y : W^{1,2}(\mathbb{R}, \mathbb{R}^n) \rightarrow L^2(\mathbb{R}, \mathbb{R}^n).
\end{equation}
Here $A_y$ is the family of endomorphisms $A$ of (\ref{aaaaa}) evaluated at the critical point $y.$ We can trivially extend this to $\mathbb{R}^{k+l-2} \times W^{1,2}(\mathbb{R}, \mathbb{R}^n),$ and still denote by $dF_y$ by abuse of notation. That is, we denote $dF_y(\vec{\tau}, Y) := dF_y(Y)$. We estimate the $L^2$-norm of $dF_y(\vec{\tau}_m, \beta^{r_m} Y_m)$ as follows:
\begin{equation}\nonumber
\begin{split}
\|dF_y&(\vec{\tau}_m, \beta^{r_m} Y_m)\|_{L^2}\\ 
&= \|dF_y(\beta^{r_m} Y_m)\|_{L^2} = \|(\dot{\beta}^{r_m}) Y_m + \beta^{r_m}  \cdot dF_y (Y_m) \|_{L^2}\\
&\leq \| \dot{\beta}^{r_m} Y_m + \beta^{r_m} \cdot dF_y (Y_m)  \|_{L^2} \leq  \| \dot{\beta}^{r_m} (Y_m) \|_{L^2} + \| \beta^{r_m} \cdot dF_y (Y_m)  \|_{L^2}\\
&\leq \frac{1}{r_{m}} \| \beta(\frac{1}{r_m}) Y_m\|_{L^2} + \|dF_y (Y_m)\|_{L^2({[-r_m, r_m]})}\\
& \leq  {}^{\textcircled{1}} \frac{1}{r_m} \cdot C + {}^{\textcircled{2}} \|dF_{\rho_m}(\vec{\tau}_m, Y_m)\|_{L^2} + {}^{\textcircled{3}} \| \big( dF_y - dF_{\rho_m} \big) (\vec{\tau}_m ,Y_m) \|_{L^2 ([-r_m, r_m])}.
\end{split}
\end{equation}
Here we write $C$ for $\max\limits_{s \in [-1,1]} |\dot{\beta}(s)|$.
For the first two terms, when $m$ tends to $\infty,$ we have 
\begin{equation}\nonumber
\textcircled{1} =  \frac{1}{r_m} \cdot C \rightarrow 0,
\end{equation}
and 
\begin{equation}\nonumber
 \textcircled{2} = \|dF_{\rho_m}(\vec{\tau}_m, Y_m)\|_{L^2} \rightarrow 0.
\end{equation}
We estimate the third term of the last line:
\begin{equation}\nonumber
\begin{split}
\textcircled{3}  = \| (dF_y - & dF_{\rho_m})(Y_m) \|_{L^2([-r_m, r_m])} \\ 
&\leq \|dF_y(Y_m) - \sum\limits_{i}dF_{1,i} \cdot \tau_i - dF_2(Y_m)\|_{L^2([-r_m, r_m])}\\
& \leq \sum\limits_i \|dF_{1,i} \cdot \tau_i\|_{L^2([-r_m, r_m])} + \|(dF_y - dF_2)(Y_m)\|_{L^2([-r_m, r_m])}\\
& \leq \sum\limits_{i} \Big|\Big| \frac{\partial \widetilde{H}}{\partial t_i} \Big|\Big|_{L^2([-r_m,r_m])} \cdot |\tau_i| + \sup\limits_{s \in [-r_m, r_m]} \|A_y - A_{\rho_m}\| \cdot \|Y_m\|_{L^2}\\
& \leq  \sum\limits_{i} \Big|\Big| \frac{\partial \widetilde{H}}{\partial t_i} \Big|\Big|_{L^2([-r_m, r_m])} + \sup\limits_{s \in [-r_m, r_m]} \|A_y - A_{\rho_m}\|.
\end{split}
\end{equation}
Here $\widetilde{H}$ and $A_{\bullet}$ are as in (\ref{df1df2}). 

Recall that for $\rho_m > \max \{R_{\sigma_k}, R_{\sigma_l}\},$ we have
\begin{equation}\nonumber
H_{\sigma_k} \#_{\rho_m} H_{\sigma_l}(\vec{t}, s, \cdot):= \begin{cases}
\overline{H}^{t \sigma_l}_{\sigma_k}(\vec{t}, s+\rho_m, \cdot) &\text{ if } s \leq 0,\\
H_{\sigma_l}(\vec{t}, s-\rho_m, \cdot) &\text{ if } s \geq 0.
\end{cases}
\end{equation}

We have 
\begin{equation}\nonumber
\begin{split}
\sum\limits_{i} & \Big|\Big| \frac{\partial \widetilde{H}}{ \partial t_i} \Big|\Big|_{L^2([-r_m, r_m])} \rightarrow 0 \text{ as } m \rightarrow \infty,
\end{split}
\end{equation}
since $H_{\sigma_k}$ and $H_{\sigma_l}$ are constants in $\vec{t}$ near the ends. Similarly, we have
\begin{equation}\nonumber
\begin{split}
 \sup\limits_{s \in [-r_m, r_m]} \|A_y - A_{\rho_m}\| \rightarrow 0 \text{ as } m \rightarrow \infty.
\end{split}
\end{equation}

Thus we conclude that
\begin{equation}\nonumber
\lim\limits_{n \rightarrow \infty} {\| dF_y (\vec{\tau}, \beta^{r_m} Y_m)\|}_{L^2} = 0
\end{equation}

On the other hand, we have
\begin{equation}\label{concl}
\lim\limits_{n \rightarrow \infty} {\| dF_y (\vec{\tau}, \beta^{r_m} Y_m)\|}_{L^2} = \lim\limits_{n \rightarrow \infty} \| \big( \frac{\partial }{\partial s} + A_y \big) ( {\beta}^{r_m}  Y_m) {\|}_{L^2} = \lim\limits_{n \rightarrow \infty} \| \beta^{r_m} {Y_m} {||}_{W^{1,2}},
\end{equation}
where the last equality is due to the fact that $A_y$ is an isometry, which follows from the condition that $A_y$ is nondegenerate (c.f. [Sch]):

\begin{equation}\nonumber
\frac{\partial}{\partial s} + A_y : W^{1,2} (\mathbb{R}, \mathbb{R}^n) \xrightarrow{\simeq} \text{im} \big( \frac{\partial}{\partial s} + A_y \big) \hookrightarrow L^2(\mathbb{R}, \mathbb{R}^n).
\end{equation}

(\ref{concl}) implies that
\begin{equation}\label{w12locconv}
Y_m \xrightarrow{W^{1,2}_{loc}} 0.
\end{equation}

Consider sequences $\vec{\tau}_m = (\vec{\tau}_{1,m}, \vec{\tau}_{2,m}) \in \mathbb{R}^{k-1} \times \mathbb{R}^{l-1}, $ $Y_m \in W^{1,2}(w_{\rho_m}^*TM_{t \sigma_l}),$ and the following estimates:

\begin{equation}\nonumber
\begin{split}
\|dF_{\sigma_k}&(\vec{\tau}_{1,m}, \beta^{-}_{1-\rho_m} Y_{{m,\rho_m}})\|_{L^2} \leq \| \sum\limits_{m} dF_{1,m}(\vec{t})+ dF_2(\vec{t})(\vec{\tau}_{1,m}, \beta^{-}_{1-\rho_m} Y_{m, 1- \rho_m})\|_{L^2}\\
&\leq \big|\big| \sum\limits_{i} \frac{\partial \widetilde{H}}{\partial t_i}  + \frac{\partial}{\partial s} (\beta_{1-\rho_m} Y_{n, 1- \rho_m}) + A(\vec{t})(\beta_{1-\rho_m} Y_{m, -\rho_m})  \big|\big|_{L^2} \\
&\leq \sum\limits_{i} \big|\big| \frac{\partial \widetilde{H}}{\partial t_i} (1- \beta^-_{1-\rho_m}) + \dot{\beta}^{-}_{1-\rho_m} Y_{m, -\rho_m} + \beta_{1-\rho_m}^- dF_2(\vec{t}, u)(Y_{m, -\rho_m})  \big|\big|_{L^2} \\
&\leq \sum\limits_{i} \big|\big| \frac{\partial \widetilde{H}}{\partial t_i} (1- \beta^-_{1-\rho_m}) \big|\big|_{L^{2}} + \| \dot{\beta}^{-}_{1-\rho_m} Y_{m, -\rho_m} \|_{L^2} + \|\beta_{1-\rho_m}^- dF_2(\vec{t}, u)(Y_m)  \big|\big|_{L^2} \\
&\leq \sum\limits_{i} \big|\big| \frac{\partial \widetilde{H}}{\partial t_i} \big|\big|_{L^{2}} + \| Y_{m, -\rho_m} \|_{L^2([-2,1])} + \| dF_2(\vec{t}, \vec{t}', u\#_{\rho_m} v)(\vec{\tau}_m, Y_m)  \big|\big|_{L^2}.\\
\end{split}
\end{equation}
Each term of the last line vanishes as $m$ tends to $\infty.$

As a consequence, we have
\begin{equation}\nonumber
\|dF_{\sigma_k}(\vec{\tau}_{1,m}, \beta^{-}_{1-\rho_m} Y_{m, -\rho_m})\|_{L^2} \rightarrow 0 \text{ as } m \rightarrow \infty,
\end{equation}
which implies that
\begin{equation}\label{l941}
(\vec{\tau}_{1,m}, \beta^{-}_{1-\rho_m} Y_{m, -\rho_m}) \xrightarrow{W^{1,2}}  (\vec{\tau}', Y'),
\end{equation}
for some $ (\vec{\tau}', Y')  \in \ker dF_{\sigma_k}$
by Lemma 9.4.11 of [AD]. Similarly, we have:
\begin{equation}\label{l942}
\|dF_{\sigma_l}(\vec{\tau}'_{2,m}, \beta^{+}_{1+\rho_m} Y_{m, \rho_m})\|_{L^2} \rightarrow 0 \text{ as } m \rightarrow \infty,
\end{equation}
and
\begin{equation}\nonumber
(\vec{\tau}_{2,m}, \beta^{+}_{1+\rho_m} Y_{m, \rho_m})  \xrightarrow{W^{1,2}}  (\vec{\tau}'', Y'')
\end{equation}
for some $(\vec{\tau}'', Y'') \in \ker dF_{\sigma_l}.$

\begin{lemma}
(\ref{l941})  and (\ref{l942}) imply the following:
\begin{equation}\nonumber
\|\beta_1^- Y_m \|_{W^{1,2}}, \ \|\beta_1^+ Y_m \|_{W^{1,2}} \rightarrow 0.
\end{equation}
\end{lemma}
\begin{proof}
The proof is essentially the same as that of Lemma 9.4.12 of [AD], which is for the Floer case. In our situation, (\ref{l941}) says that we have
\begin{equation}\nonumber
\beta^{-}_{1- \rho_m} Y_{m, 1-\rho_m} \xrightarrow{W^{1,2}} Y',
\text{ so }
\beta^{-}_{1} Y_{m} \xrightarrow{W^{1,2}} Y'(\cdot  + \rho_m),
\end{equation}
as $m$ tends to $\infty.$ Also, we have
\begin{equation}\nonumber
\beta^{-}_{1} Y_{m} \xrightarrow{W^{1,2}(-\infty, -1]} Y' \#_{\rho_m} 0,
\end{equation}
since Supp ($\beta^{-}_{1} Y_{m} $) is contained in $(-\infty, 0].$ Furthermore, from (\ref{w12locconv}) we have 
\begin{equation}\nonumber
\beta^{-}_{1} Y_{m} \xrightarrow{W^{1,2}} Y' \#_{\rho_m} 0.
\end{equation}
Observe that $\beta^{-}_{1} Y_{m} - Y' \#_{\rho_m} 0$ and $Y' \#_{\rho_m} 0$ are of the $L^2$-class. Then by H\"{o}lder inequality, we have
\begin{equation}\label{ineqhoe}
\begin{split}
\Big| \int_{\mathbb{R}} \langle \beta^{-}_1 Y_m - Y' \#_{\rho_m} 0, & Y' \#_{\rho_m} 0 \rangle ds \Big| \\
&\leq \| \beta^{-}_{1} Y_{m} - Y' \#_{\rho_m} 0 \|_{L^2} \cdot \| Y' \#_{\rho_m} 0 \|_{L^2} \rightarrow 0.
\end{split}
\end{equation}
On the other hand, we have
\begin{equation}\label{inequ2}
\begin{split}
\int_{\mathbb{R}} \langle \beta^{-}_1 Y_m, & Y' \#_{\rho_m} 0 \rangle ds \\
&= \int_{\mathbb{R}} \langle Y_m, Y' \#_{\rho_m} 0 \rangle ds - \int_{\mathbb{R}} \langle (1- \beta^{-}_1)  Y_m, Y' \#_{\rho_m} 0 \rangle ds \rightarrow 0.
\end{split}
\end{equation}
This requires some explanation: the first term on the right hand side vanishes because $Y' \#_{\rho_n}0 \in W_{\rho_m},$ with $({0}, 0)$ being trivially an element of $\ker dF_{\sigma_l}.$ Similarly to that of (\ref{ineqhoe}), we can show that the second term tends to zero by using the fact that its support is contained in $[-1,0],$ and the H\"{o}lder inequality with $(\ref{w12locconv})$. Comparing (\ref{ineqhoe}) and (\ref{inequ2}), we know
\begin{equation}\nonumber
\int_{\mathbb{R}} \langle Y' \#_{\rho_m} 0, Y' \#_{\rho_m} 0 \rangle ds \rightarrow 0.
\end{equation}
Hence we have
\begin{equation}\nonumber
\int_{-\infty}^{-1} \langle Y'_{\rho_m}, Y'_{\rho_m} \rangle ds = \int_{-\infty}^{-1+\rho_m} \langle Y', Y' \rangle ds \rightarrow 0,
\end{equation}
and $Y' =0.$ Now it follows that $\beta_1^{-} Y_m \xrightarrow{W^{1,2}} 0,$ Similarly, we can show that $\beta_{-1}^{+} Y_m \xrightarrow{W^{1,2}} 0.$
\end{proof}

\begin{lemma}
We have $\vec{\tau}_m \rightarrow 0.$
\end{lemma}
\begin{proof-sketch}
Lemma \ref{lefrop} with the setting of Theorem \ref{gluthm} implies that $\mathcal{L}_{\rho}$ is of index 0. The transversality assumption leads to its surjectivity, so it is also injective. Then similarly to (the earlier part of) the proof of Proposition \ref{invest}, we can show that $\|\mathcal{L}_{\rho}(\vec{\tau}_m, \beta^{-}_{1 - \rho_m} Y_{-\rho_m}) \|$ tends to zero. Since $\mathcal{L}_{\rho}$ is an isomorphism, the assertion of the lemma follows.
\end{proof-sketch}

Now we have a contradiction:
\begin{equation}\nonumber
\begin{split}
1&= \lim\limits_{m \rightarrow \infty} \big( \| \vec{\tau}_m \| + \| Y_m \|_{W^{1,2}} \big) 
= \lim\limits_{m \rightarrow \infty}  \| Y_m \|_{W^{1,2}} \\
& \leq \lim\limits_{m \rightarrow \infty} \big( \| \beta^-_{1} Y_m \|_{W^{1,2}} + \| \beta^+_{-1} Y_m \|_{W^{1,2}} + \| (1- \beta^-_1 - \beta^+_{-1}) Y_m \|_{W^{1,2}}  \big)  = 0,
\end{split}
\end{equation}
where the last equality holds since we know $ \beta^-_{1} Y_m, \ \beta^+_{-1} Y_m \rightarrow 0, \text{ supp}(1- \beta^-_1 - \beta^+_{-1}) \subset [-2,2],$ and (\ref{w12locconv}). This proves Proposition \ref{invest}.
\end{proof}

Proposition \ref{invest} implies that
\begin{equation}\nonumber
\ker \mathcal{L}_{\rho} \cap \big( \mathbb{R}^{k+l-2} \times W_{\rho}^{\perp} \big) = \{ 0 \}.
\end{equation}
Hence we have

\begin{equation}\nonumber 
\mathbb{R}^{k+l-2} \times W^{1,2}(\mathbb{R}, \mathbb{R}^n) \simeq  \ker \mathcal{L}_{\rho} \oplus \big( \mathbb{R}^{k+l-2} \times W_{\rho}^{\perp} \big).
\end{equation}

Let $\text{Proj}^{\mathbb{R}^{k+l-2} \times W^{1,2} (w_{\rho}^{*}TM_{t \sigma_l})}_{\ker \mathcal{L}_{\rho} } : \mathbb{R}^{k+l-2} \times W^{1,2} (w_{\rho}^{*}TM_{t \sigma_l}) \rightarrow {\ker \mathcal{L}_{\rho}}$
be the projection map to the subspace $\ker \mathcal{L}_{\rho}.$ We denote the map obtained by composing the pregluing map $\#_{\rho}$ with this projection by
\begin{equation}\nonumber
\phi_{\rho} := \text{Proj}^{\mathbb{R}^{k+l-2} \times W^{1,2} (w_{\rho}^{*}TM_{t \sigma_l})}_{\ker \mathcal{L}_{\rho} } \circ \#_{\rho} : \ker dF_{\sigma_k} \times \ker dF_{\sigma_l} \rightarrow \ker \mathcal{L}_{\rho}.
\end{equation}
The following theorem holds for general cases, and not under the setting (\ref{gtst}) only.
\begin{theorem}\label{lsurj}
Suppose that $dF_{\sigma_k}$ and  $dF_{\sigma_l}$ are surjective. Then so is $\mathcal{L}_{\rho}$ for sufficiently large $\rho>0.$ Moreover, $\phi_{\rho}$ is an isomorphism.
\end{theorem}

\begin{proof}
Recall that $dF_{\sigma_k}$ and  $dF_{\sigma_l}$ are Fredholm, and we have
\begin{equation}\nonumber
\begin{split}
\text{ind} dF_{\sigma_k} &= \text{dim} \ker dF_{\sigma_k},\\
\text{ind} dF_{\sigma_l} &= \text{dim} \ker dF_{\sigma_l}
\end{split}
\end{equation}
from the assumption of surjectivity for the both operators. Then it follows that
\begin{equation}\nonumber
\begin{split}
\text{dim} \ker \mathcal{L}_{\rho} &\leq  \text{dim} \ker dF_{\sigma_k} +  \text{dim} \ker dF_{\sigma_l} =  \text{ind} dF_{\sigma_k} + \text{ind} dF_{\sigma_l}\\
&= |x| - |y| + k-1 + |y| -|z| + l-1 = |x| -|z| +k+l-2\\  &= \text{ind} \mathcal{L}_{\rho} \leq \text{dim} \ker \mathcal{L}_{\rho}.
\end{split}
\end{equation}
Hence we get:
\begin{equation}\nonumber
\text{dim} \ker \mathcal{L}_{\rho} = \text{dim} \ker dF_{\sigma_k} +  \text{dim} \ker dF_{\sigma_l}.
\end{equation}
Suppose that $\phi_{\rho}$ is not surjective. Then there exists nonzero $(\vec{\tau}', Y') \in \text{ker} \mathcal{L}_{\rho}$ with $(\vec{\tau}', Y') \in \text{im}(\#_{\rho})^{\perp},$ so that we have 
\begin{equation}\nonumber
\langle ( \vec{\tau}', Y'),\big((\vec{\tau}_1, \vec{\tau}_2), Y_1 \#_{\rho} Y_2 \big) \rangle =0,
\end{equation}
for any $(\vec{\tau}_1, Y_1) \in \mathbb{R}^{k-1} \times W^{1,2}(u^*TM_{t \sigma_k})$ and $(\vec{\tau}_2, Y_2) \in \mathbb{R}^{l-1} \times W^{1,2}(v^*TM_{t \sigma_l}).$
But we assumed $(\vec{\tau}', Y') \in \ker \mathcal{L}_{\rho},$ and Proposition \ref{invest} implies that $\|\vec{\tau}' \| = \|Y'\|_{L^2} = 0.$ Hence $(\vec{\tau}', Y') =0,$ which is a contradiction. Therefore $\phi_{\rho}$ is surjective.
\end{proof}

\section{Proof of the gluing theorem II}

In this section, we provide the second half of a proof for Theorem \ref{gluthm}. Again, we generalize the discussions in [AD] and [Sch] by adding extra parametrizations by cubes. For example, Lemma 8.2 is a straightforward generalization of Lemma 11.4.54 of [AD].

For $\rho > {\rho_0},$ and 
\begin{equation}\nonumber
\mathcal{L}_{\rho} : \mathbb{R}^{k+l-2} \times  W^{1,2}(\mathbb{R} ,\mathbb{R}^n) \simeq \ker \mathcal{L}_{\rho} \oplus (\mathbb{R}^{k+l-2} \times W_{\rho}^{\perp}) \rightarrow L^2(\mathbb{R}, \mathbb{R}^n)
\end{equation}
surjective, we define 
\begin{equation}\nonumber
\mathcal{G}_{\rho} :L^{2}(\mathbb{R}, \mathbb{R}^n) \rightarrow \mathbb{R}^{k+l-2} \times W_{\rho}^{\perp} \hookrightarrow  \mathbb{R}^{k+l-2} \times W^{1,2}(\mathbb{R}, \mathbb{R}^n)
\end{equation}
by $\mathcal{G}_{\rho}:= \mathcal{L}_{\rho}|_{\mathbb{R}^{k+l-2} \times W_{\rho}^{\perp}}^{-1},$ so that $\mathcal{L}_{\rho} \circ \mathcal{G}_{\rho} = id_{L^2(\mathbb{R}, \mathbb{R}^{n})}.$

We recall the contraction mapping theorem for Banach spaces.
\begin{theorem}\label{ctrct} (Contraction mapping theorem)
Let $F : X_1 \rightarrow X_2$ be a map of Banach spaces which can be written as 
\begin{equation}\nonumber
F = F(0) + dF + N
\end{equation}
for some $N: X_1 \rightarrow X_2.$ Suppose we have $G : X_2 \rightarrow X_1$ such that
\begin{enumerate}[label = (\roman*)]
\item $dF \circ G = id_Y,$
\item $\|G \circ F(0)\| \leq \frac{\epsilon}{2},$ where $\epsilon := \min (r, \frac{C}{5}),$
\item $\|G \circ N (x) - G \circ N(y) \| \leq C \big( \|x \| + \|y\| \big) \|x-y\|$ for some $C>0$ and for all $x, y \in B(0,r)$ for some $r>0.$
\end{enumerate}
Then there exists a unique $z \in B(0, \epsilon) \cap \text{im} G$ such that $F(z) = 0.$  Moreover, we have $\|z\| < 2 \| G \circ F(0)\|.$
\end{theorem}

Recall that we defined $\mathcal{F}_{\rho} : [0,1]^{k+l-2} \times W^{1,2}(\mathbb{R}, \mathbb{R}^n) \rightarrow L^2(\mathbb{R}, \mathbb{R}^n)$ by
\begin{equation}\nonumber
\mathcal{F}_{\rho}(\vec{t}, \vec{t}', Y) := \Big( \frac{d}{ds} + \nabla_{G^{\vec{t}}_{\sigma_k} \#_{\rho} G^{\vec{t}}_{\sigma_l}(s + \rho)} \big( H^{\vec{t}}_{\sigma_k} \#_{\rho} H^{\vec{t}'}_{\sigma_l}(s+ \rho) \big) \Big)(\exp_{w_{\rho}}Y).
\end{equation}
 
We consider an extension of $\mathcal{F}_{\rho}$

\begin{equation}\nonumber
\overline{\mathcal{F}}_{\rho} : \mathbb{R}^{k+l-2} \times W^{1,2}(\mathbb{R}, \mathbb{R}^n) \rightarrow L^2(\mathbb{R}, \mathbb{R}^n),
\end{equation}

so that
\begin{equation}\nonumber
\overline{\mathcal{F}}_{\rho} |_{[0,1]^{k+l-2} \times W^{1,2}(\mathbb{R}, \mathbb{R}^n)} = \mathcal{F}_{\rho}.
\end{equation}
By abuse of notation, it will sometimes be denoted by $\mathcal{F}_{\rho}.$

For $(\vec{t}, \vec{t}') \in [0,1]^{k+l-2},$ we write
\begin{equation}\nonumber
\mathcal{F}_{\rho}(\underline{\vec{t}}, \underline{\vec{t}'}, Y) := \overline{\mathcal{F}}_{\rho}(\underline{\vec{t}} + \vec{t}, \underline{\vec{t}'} + \vec{t}', Y).
\end{equation}
We then rewrite $\overline{\mathcal{F}}_{\rho}$ as follows
\begin{equation}\nonumber
\begin{split}
\overline{\mathcal{F}}_{\rho}(\underline{\vec{t}}, \underline{\vec{t}'}, Y) & = \overline{\mathcal{F}}_{\rho} (0,0,0) + d \overline{\mathcal{F}}_{\rho}^{(\vec{t}, \vec{t}')}|_{0}(Y) + \overline{\mathcal{N}}_{\rho}(Y)\\
& = \mathcal{F}_{\rho}(\vec{t}, \vec{t}', 0) + \mathcal{L}_{\rho}(Y) + \overline{\mathcal{N}}_{\rho}(Y),
\end{split}
\end{equation}
for some $\overline{\mathcal{N}}_{\rho}.$

In our case, $\overline{\mathcal{F}}_{\rho}$ corresponds to $F$ and the above $\mathcal{G}_{\rho}$ to $G$ of Theorem \ref{ctrct}. Recall that $\mathcal{L}_{\rho}:= d \overline{\mathcal{F}}_{\rho}|_{(0,0,0)}.$

{\it{Checking the condition (i) of Theorem \ref{ctrct}}}. We have $\mathcal{L}_{\rho} \circ \mathcal{G}_{\rho} = id_{L^2(\mathbb{R}, \mathbb{R}^n)}.$

{\it{Checking the condition (ii) of Theorem \ref{ctrct}}}. For $\rho$ sufficiently large, we have 
\begin{equation}\nonumber
\| \mathcal{G}_{\rho} \circ \overline{\mathcal{F}}_{\rho}(0,0,0) \|_{\mathbb{R}^{k+l-2} \times W^{1,2}(\mathbb{R}, \mathbb{R}^n)} \leq  C^{-1} \|\overline{\mathcal{F}}_{\rho}(0,0,0) \|_{L^2} < \frac{\epsilon}{2},
\end{equation}
where the first inequality (with the constant $C$) is from Proposition \ref{invest}, while the second one is from the following observation:
\begin{equation}\nonumber
\begin{split}
\overline{\mathcal{F}}_{\rho}(0,0,0) & = \mathcal{F}_{\rho}^{(\vec{t}, \vec{t}')}(0) \\ 
& = \frac{d w_{\rho}}{ds} + \nabla_{G_{\sigma_{k}}^{\vec{t}} \#_{\rho} G_{\sigma_l}^{\vec{t}'}}(H_{\sigma_{k}}^{\vec{t}} \#_{\rho} H_{\sigma_l}^{\vec{t}'}) \circ w_{\rho} \rightarrow 0 \text{ as } \rho \rightarrow \infty.
\end{split}
\end{equation}

\begin{lemma}\label{chktwo}
There exist constants $K>0$ and $\rho_0>0$ such that for every $\rho > \rho_0,$ we have
\begin{equation}\nonumber
\|d \mathcal{N}_{\rho}|_{(\vec{t}, \vec{t}', Y(\rho))} \|^{op} \leq K \big( \|Y\|_{W^{1,2}(\mathbb{R}, \mathbb{R}^n)} + \| (\vec{t}, \vec{t}') \| \big),
\end{equation}
where the left hand side is with respect to the operator norm, and $K$ is independent of the choices of $Y, x,$ and $y.$
\end{lemma}
\begin{proof}

Consider the linearization $d \overline{\mathcal{F}}_{\rho}(\vec{t}, \vec{t}', Y)(\vec{\tau}'', Z).$ We have
\begin{equation}\nonumber
\begin{split}
\|d \overline{\mathcal{F}}_{\rho}(\vec{t}'', Y) - & d \overline{\mathcal{F}}_{\rho}(0,0)\|^{op}\\
& \leq \| d \overline{\mathcal{F}}_{\rho}(\vec{t}'', Y) - d \overline{\mathcal{F}}_{\rho}(\vec{t}'', 0) \|^{op} + \| d \overline{\mathcal{F}}_{\rho}(\vec{t}'', 0) - d \overline{\mathcal{F}}_{\rho}(0, 0) \|^{op}.
\end{split}
\end{equation}

\begin{equation}\nonumber
\big( d \overline{\mathcal{F}}_{\rho}(\vec{t}'', Y) \big)(\vec{\tau}'', Z) = \big( d \overline{\mathcal{F}}_{\rho}(\vec{t}'', Y) \big)(0, Z)  + \big( d \overline{\mathcal{F}}_{\rho}(\vec{t}'', Y) \big)(\vec{\tau}'', 0),
\end{equation}
where
\begin{equation}\nonumber
\begin{split}
& d \overline{\mathcal{F}}_{\rho} (\vec{t}'', Y) =: d \overline{\mathcal{F}}_{\rho}^{\vec{t}''}(Y),\\
& d \overline{\mathcal{F}}_{\rho} (\vec{t}'', Y) = \sum\limits_{i} \vec{\tau}'' \cdot V_{\rho_{n,i}}(\vec{t}'', Y).
\end{split}
\end{equation}
It follows that
\begin{equation}\label{str1}
\begin{split}
 \| d \overline{\mathcal{F}}_{\rho} & (\vec{t}'', Y)(\vec{\tau}'', Z) - d \overline{\mathcal{F}}_{\rho}(\vec{t}, 0)(\vec{\tau}'', Z) \|_{L^2}\\ 
& \leq \| d \overline{\mathcal{F}}_{\rho}^{\vec{t}''}(Y) (Z) - d \overline{\mathcal{F}}_{\rho}^{\vec{t}''}(0) (Z) \|_{L^2} + \sum\limits_i \| \tau''_i \cdot V_{\rho, i} (\vec{t}'', Y) - \tau''_i \cdot V_{\rho, i} (\vec{t}'', 0) \|_{L^2}.
\end{split}
\end{equation}

\begin{lemma}
$\| d \overline{\mathcal{F}}_{\rho}^{\vec{t}''}(Y)(Z) - d \overline{\mathcal{F}}_{\rho}^{\vec{t}''}(0)(Z) \|_{L^2} \leq k_1 \|Y\|_{W^{1,2}} \|Z\|_{W^{1,2}},$
\end{lemma}
\begin{proof}
We can refer to Lemma 9.4.8 in [AD], where the inequality is proved for the Floer case. But the Morse version is essentially the same, since the underlying analysis of Fredholm operators applies to the both cases.
\end{proof}

By Lemma 13.8.1 of [AD], we have
\begin{equation}\label{str3}
\sum\limits_{i} |\tau_i| \cdot \| V_{\rho, i}(\vec{t}'', Y) - V_{\rho,i}(\vec{t}'',0) \|_{L^2} \leq K \|Y\|_{W^{1,2}},
\end{equation}
and this follows from the fact that $V_{\rho, i}(\vec{t}, \cdot)$ is of $C^1$-class. Hence (\ref{str1}) implies that

\begin{equation}\nonumber
\begin{split}
  \| d \overline{\mathcal{F}}_{\rho}&  (\vec{t}'', Y)(\vec{\tau}'', Z) -   d \overline{\mathcal{F}}_{\rho} (\vec{t}'', 0) (\vec{\tau}'', Z)  \|_{L^2}\\
& \leq k_1 \|Y\|_{W^{1,2}}\|Z\|_{W^{1,2}} + k_2 \| Y \|_{W^{1,2}}  \| \vec{\tau}'' \| \\
& \leq (k_1 + k_2) \| Y \|_{W^{1,2}} \sqrt{\|Z\|^2 + \|\vec{\tau}''\|^2}.
\end{split}
\end{equation}
for some $k_1, k_2 >0.$
Hence we have
\begin{equation}\label{str4}
\| d \overline{\mathcal{F}}_{\rho}(\vec{t}'', Y) - d \overline{\mathcal{F}}_{\rho}(\vec{t}'', 0)  \|^{op} \leq (k_1 + k_2) \|Y\|_{W^{1,2}}.
\end{equation}
Recall that 
\begin{equation}\nonumber
d\overline{\mathcal{F}}_{\rho}(\vec{t}'', 0)(\vec{\tau}'', Z) = (d \mathcal{F}_{\rho})_0^{\vec{\underline{t}''} + \vec{t}'' }(Z) + \sum\limits_i \tau_i'' V_{\rho, i}(\underline{\vec{t}}''+ \vec{t}'', 0).
\end{equation}
For $\|d \overline{\mathcal{F}}_{\rho}(\vec{\underline{t}''}, 0) - d \overline{\mathcal{F}}_{\rho}(0,0)\|^{op},$ we have
\begin{equation}\nonumber
\begin{split}
\| d\overline{\mathcal{F}}_{\rho}& (\vec{t}'', 0) (\vec{\tau}'', Z) -  d \overline{\mathcal{F}}_{\rho}(0,0)(\vec{\tau}'', Z)\| \\
& \leq {}^{\textcircled{1}} \| (d\mathcal{F}_{\rho})^{\underline{\vec{t}}'' + \vec{t}''}_0 (Z) -  (d\mathcal{F}_{\rho})^{ \vec{t}''}_0 (Z) \| + {}^{\textcircled{2}} \| \sum\limits_{i} \tau_i'' \cdot \big( V_{\rho, i} (\underline{\vec{t}}'' + \vec{t}'', 0)- V_{\rho, i}(\underline{\vec{t}''}) \big) \|,
\end{split}
\end{equation}
and
\begin{equation}\nonumber
\begin{split}
\textcircled{1} = \| (d \mathcal{F}_{\rho})&^{\underline{\vec{t}}'' + \vec{t}''}_0 (Z) -  (d\mathcal{F}_{\rho})^{ \vec{t}''}_0 (Z) \| \\
&= \| A(\underline{\vec{t}}'')(Z) - A(0)(Z) \|_{L^2} \leq k_4  \| \underline{\vec{t}}''\|  \| Z \|_{L_2} \leq k_4 \|\underline{\vec{t}}''\| \| Z \|_{W^{1,2}},
\end{split}
\end{equation}
where $k_4 = \sup\limits_{\substack{\vec{t} \in [0,1]^{k-1}, \\  i = 1, \cdots, k-1.}} \frac{\partial A}{\partial \underline{t}_i}$
and 
\begin{equation}\nonumber
\textcircled{2} = \| \sum\limits_{i} \tau_i'' \cdot \big( V_{\rho, i} (\underline{\vec{t}}'' + \vec{t}'', 0)- V_{\rho, i}(\vec{t}'') \big) \| \leq \cdots \leq k_5 \|\vec{\tau}''\| \| \underline{\vec{t}}'' \|
\end{equation}
similarly as (\ref{str3}) for some $k_5>0.$ 

Thus we have 
\begin{equation}\nonumber
\begin{split}
\| d \overline{\mathcal{F}}_{\rho} & (\vec{t}'', 0) (\vec{\tau}'', Z) -  d \overline{\mathcal{F}}_{\rho}(0,0)(\vec{\tau}'', Z) \| \\ 
& \leq k_4 \|\underline{\vec{t}}''\| \| Z \|_{W^{1,2}} + k_5 \|\vec{\tau}''\| \| \underline{\vec{t}}'' \| \leq k_6 \| \underline{\vec{t}}'' \| \sqrt{  \|Z\|^2 + \| \vec{\tau}'' \|^2   },
\end{split}
\end{equation}
so that
\begin{equation}\label{str2}
\| d\overline{\mathcal{F}}_{\rho} (\vec{t}'', 0) - d \overline{\mathcal{F}}_{\rho}(0,0)\|^{op} \leq k_6 \| \underline{\vec{t}}'' \|,
\end{equation}
where $k_6:= k_4+k_5.$
From (\ref{str4}) and (\ref{str2}), we now conclude that
\begin{equation}\nonumber
\begin{split}
\| d\overline{\mathcal{F}}_{\rho} (\vec{t}'', Y) - d \overline{\mathcal{F}}_{\rho}(0,0) \|^{op} & \leq (k_1 + k_2) \| Y \|_{W^{1,2}} + k_6 \| \underline{\vec{t}}'' \| \\
& \leq K \big( \| Y\|_{W^{1,2}} + \| \underline{\vec{t}}'' \| \big),
\end{split}
\end{equation}
where $K:= \max (k_1 + k_2, k_6),$ which finishes the proof of Lemma \ref{chktwo}.
\end{proof}

{\it{Checking the condition (iii) of Theorem \ref{ctrct}}}. For $(\vec{t}''_1, Y_1), (\vec{t}''_2, Y_2) \in B(0,r),$ we have
\begin{equation}\nonumber
\begin{split}
\|\mathcal{G}_{\rho} & \circ \mathcal{N}_{\rho}(\vec{t}''_1, Y_1) - \mathcal{G}_{\rho} \circ \mathcal{N}_{\rho}(\vec{t}''_2, Y_2) \|_{W^{1,2}}\\ 
&\leq {C}^{-1} \cdot \|\mathcal{N}_{\rho}(\vec{t}''_1, Y_1) - \mathcal{N}_{\rho}(\vec{t}''_2, Y_2)\|_{L^2} \\ 
& \leq C^{-1} \cdot \| d \mathcal{N}_{\rho} \|^{op} \cdot \|(\vec{t}''_1, Y_1)-(\vec{t}''_2, Y_2)\| \\  
& \leq {C}^{-1} \cdot \sup\limits_{ (\vec{t}'', Y) \in B(0, \epsilon) } K \big(\| Y \|_{W^{1,2}} + \|\vec{t}''\|\big) \cdot \|(\vec{t}''_1, Y_1)-(\vec{t}''_2, Y_2)\| \\ 
& \leq C^{-1} \cdot K \cdot \big(\|(\vec{t}''_1, Y_1)\|  + \|(\vec{t}''_2, Y_2)\|\big) \cdot \|(\vec{t}''_1, Y_1)-(\vec{t}''_2, Y_2)\|.
\end{split}
\end{equation}

Then by Theorem \ref{ctrct}, we have

\begin{theorem}\label{glucont}
There exists unique $ \big(\underline{\vec{t}}_{\rho}, \underline{\vec{t}'}_{\rho}, \Gamma(\rho)\big) \in \text{Im }\mathcal{G}_{\rho} \cap B \big( 0, \epsilon \big)$ such that
\begin{equation}\nonumber
\overline{\mathcal{F}}_{\rho}\big(\Gamma(\rho) \big) = \Big( \frac{d}{ds} + \nabla_{G_{\sigma_k}^{\vec{t} + \underline{\vec{t}}_{\rho}} \#_{\rho} G_{\sigma_l}^{\vec{t} + \underline{\vec{t}}_{\rho}'}}H_{\sigma_k}^{\vec{t} + \underline{\vec{t}}_{\rho}} \#_{\rho} H_{\sigma_l}^{\vec{t} + \underline{\vec{t}}_{\rho}'} \Big) \big(\exp_{w_{\rho}}\big( \Gamma(\rho) \big) \big) =0.
\end{equation}
In other words, we have $\big( \vec{t} + \underline{\vec{t}}_{\rho}, \vec{t}' + \underline{\vec{t'}}_{\rho}, \exp_{w_{\rho}} \big( \Gamma ( \rho ) \big) \big) \in \mathcal{M}( H_{\sigma_{k}} \#_{\rho} H_{\sigma_{l}}; x, z).$
\end{theorem}

\begin{definition}
We define the \textit{gluing of the two parametrized trajectories} $(\vec{t}, u)$ and $(\vec{t}', v)$ for the gluing parameter $\rho$ by
\begin{equation}\nonumber
(\vec{t}, u) \#_{\rho} (\vec{t}', v) := \big( \vec{t} + \underline{\vec{t}}_{\rho}, \vec{t}' + \underline{\vec{t}'}_{\rho}, \exp _{w_{\rho}} \big( \Gamma (\rho) \big) \big).
\end{equation}
\end{definition}

\begin{proposition}\label{prp2}
We have
\begin{enumerate}[label = (\roman*)]
\item $d \mathcal{F}_{\rho}|_{\Gamma(\rho)}$ is surjective.
\item $\lim_{\rho \rightarrow \infty} \| \Gamma(\rho) \|_{\mathbb{R}^{k+l-2} \times W^{1,2}(\mathbb{R}, \mathbb{R}^n)}=0.$
\end{enumerate}
\end{proposition}

\begin{proof}
\begin{enumerate}[label=(\roman*)]
\item It follows from $d \mathcal{F}_{\rho}|_{\Gamma(\rho)} =  d \mathcal{F}_{\rho} \circ (d \exp_{w_{\rho}})^{-1} $ and the fact that $d \exp_{w_{\rho}} $ is an isomorphism.
\item We have
\begin{equation}\nonumber
\begin{split}
\|\Gamma(\rho)\|_{W^{1,2}} &\leq 2 \| \mathcal{G}_{\rho} \circ \mathcal{F}_{\rho} (0) \|_{W^{1,2}} \leq 2 C^{-1} \cdot \| \mathcal{F}_{\rho}(0)  \|_{L^2}\\ 
&= 2 C^{-1} \cdot \| \mathcal{F}(w_{\rho})  \|_{L^2} \rightarrow 0 \text{ as } \rho \rightarrow \infty
\end{split}
\end{equation}
\end{enumerate}
\end{proof}

In the above setting, we define the following map

\begin{equation}\nonumber
\begin{split}
\chi_{\rho} : \bigcup\limits_{\substack{y \in Crit(h_{t \sigma_k}), \\ |y| - |x| = k-1, \\ |z| - |y| = l-1}} \mathcal{M}(H_{\sigma_k} ; x, y) & \times \mathcal{M}(H_{\sigma_l}; y, z) \rightarrow \mathcal{M}(H_{\sigma_k} \#_{\rho} H_{\sigma_l}; x, z),\\
\big( (\vec{t}, u), (\vec{t}', v) \big) & \mapsto \big(  \vec{t} + \underline{\vec{t}}_{\rho}, \vec{t}' + \underline{\vec{t}'}_{\rho}, \exp _{w_{\rho}} \big( \Gamma (\rho) \big)  \big).
\end{split}
\end{equation}

The following can be studied as rigorously as Proposition 11.5.4 of [AD], but we omit its detail in this paper.
\begin{proposition}
\begin{equation}\nonumber
\begin{split}
& \lim_{\rho \rightarrow \infty} \chi_{\rho}\big( (\vec{t}, u), (\vec{t}',v) \big)(s-\rho) = \big( \vec{t} + \underline{\vec{t}}, \vec{t}' + \underline{\vec{t}'}, u(s) \big),\\
& \lim_{\rho \rightarrow \infty} \chi_{\rho}\big( (\vec{t}, u), (\vec{t}',v) \big)(s+\rho) = \big( \vec{t} + \underline{\vec{t}}, \vec{t}' + \underline{\vec{t}'}, v(s) \big) \text{ in } C^{\infty}_{loc},
\end{split}
\end{equation}
where $\underline{\vec{t}} := \lim_{\rho \rightarrow \infty} \underline{\vec{t}}_{\rho}$ and $\underline{\vec{t}'}:= \lim_{\rho \rightarrow \infty} \underline{\vec{t}'}_{\rho}$.
\end{proposition}

\begin{theorem}\label{thminjsurj}
There exists $\rho_{0} > 0$ such that for every $\rho > \rho_0,$ we have
\begin{enumerate}[label = (\roman*)]
\item $\chi_{\rho}$ is injective.
\item $\chi_{\rho}$ is surjective.
\end{enumerate}
\end{theorem}

\begin{proof-sketch}
We explain why this is true.

(i) Suppose not, then there exist sequences $\rho_n>0, \ (\vec{t}_n, \vec{t}'_n, u_n, v_n),$ and \\  
$ (\vec{t}''_n, \vec{t}'''_n, u_n, v_n)$ with 

\begin{equation}\nonumber
\begin{cases}
\rho_n \rightarrow \infty \text{ as } n \rightarrow \infty,\\ 
(\vec{t}_n, \vec{t}'_n, u_n, v_n) \neq (\vec{t}''_n, \vec{t}'''_n, u_n, v_n), \text{ for each }n,\\
\chi_{\rho} (\vec{t}_n, u_n, \vec{t}'_n, v_n)  = \chi_{\rho} (\vec{t}''_n, u_n, \vec{t}'''_n, v_n), \text{ for each }n,
\end{cases}
\end{equation}

so that from Proposition \ref{prp2}, we have
\begin{equation}\nonumber
\begin{split}
\big(\vec{t}_n, \vec{t}'_n, u(s) \big) = \lim\limits_{\rho \rightarrow \infty} \chi_{\rho} (\vec{t}_n, \vec{t}'_n, & u_n, v_n)(s-\rho)\\ 
&= \lim\limits_{\rho \rightarrow \infty} \chi_{\rho} (\vec{t}''_n, \vec{t}'''_n, u_n, v_n)(s + \rho) = \big(\vec{t}''_n, \vec{t}'''_n, u'(s) \big).
\end{split}
\end{equation}
It follows that $u = u',$ and similarly $v = v',$ which is a contradiction.

(ii) We claim that $\chi_{\rho_n}$ is surjective for all sufficiently large $n.$ Suppose not, then there exists a sequence $l_n = (\vec{t}_n, \vec{t}'_n, l_n) \notin \text{im}\chi_{\rho_n},$ for all $ n.$ Consider
\begin{equation}\nonumber
\overline{l}_n := (\vec{t}_n, \vec{t}'_n, l_n) \in \mathcal{M}(H_{\sigma_k} \#_{\rho} H_{\sigma_l}; x, z).
\end{equation}
Then there exists a sequence $\rho_n >0$ such that
\begin{equation}\nonumber
\begin{split}
& l_n(s - \rho_n) \xrightarrow{n \rightarrow \infty} u(s) \text{ in } C^{\infty}_{loc} \text{ with } (\vec{t},u) \in \mathcal{M}(H_{\sigma_k}; x, y) \text{ for some } \vec{t} \in [0,1]^{k-1}, \\
& l_n(s + \rho_n) \xrightarrow{n \rightarrow \infty} v(s) \text{ in } C^{\infty}_{loc} \text{ with } (\vec{t}',v) \in \mathcal{M}(H_{\sigma_l}; y, z) \text{ for some } \vec{t}' \in [0,1]^{l-1}.
\end{split}
\end{equation}
Here the convergences are guaranteed by  Lemma \ref{aathm}. Then there exists $Y_n \in  W^{1,2}(w_{\rho_n}^*TM_{t \sigma_l}) $ such that $\exp_{w_{\rho_{n}}}(Y_n) = l_n$ and $\| Y_n \|_{\infty}, \ \|Y_n\|_{W^{1,2}} \rightarrow 0.$ (See the comments in the sketch of proof of Theorem \ref{thm15}.) For $n$ sufficiently large, $K$ in Proposition \ref{chktwo} is independent of $\rho_n,$ and the radius $\epsilon$ in the contraction mapping theorem is independent of $n.$ As $\| Y_n \|_{W^{1,2}} \rightarrow 0,$ we can use the uniqueness argument for the solution (for sufficiently large $n$) to have $Y_n = \gamma(\rho_n).$ Thus we have $l_n = \chi_{\rho_n}\big( (\vec{t}_n - \underline{\vec{t}}_n, u), (\vec{t}'_n - \underline{\vec{t}'}_n,v) \big)$ for all sufficiently large $n$ and $\underline{\vec{t}}_n$ and $\underline{\vec{t}'}_n$ as in Theorem \ref{glucont}, which is a contradiction.
\end{proof-sketch}

We now have a bijection between the finite sets:
\begin{equation}\nonumber
\bigcup\limits_{\substack{y \in Crit(h_{t \sigma_k}), \\ |y| - |x| = k-1, \\ |z| - |y| = l-1}} \mathcal{M}(H_{\sigma_k} ; x, y)  \times \mathcal{M}(H_{\sigma_l}; y, z) \simeq \mathcal{M}(H_{\sigma_k} \#_{\rho} H_{\sigma_l}; x, z).
\end{equation}
That is, we have proved Theorem \ref{gluthm}. \qed

\ \

Recall that we defined the maps $\overline{\varphi}_{{\sigma_k}}: MC_*(\overline{h}^{t \sigma_l}_{s \sigma_k}) \rightarrow MC_*(h_{t \sigma_l}),$ by counting the order of the parametrized moduli space of solutions, from which we obtained the higher continuation maps:
\begin{equation}\nonumber
\varphi_{{\sigma_k}} := \overline{\varphi}_{{\sigma_k}} \circ \iota_{\sigma_k} : MC_*(h_{s \sigma_k}) \rightarrow MC_*(h_{t \sigma_l}).
\end{equation}

\begin{corollary}\label{corlater}
Let $\rho_0$ be as in Theorem \ref{thminjsurj}. Then for each $\rho \geq \rho_0,$ we have
\begin{equation}\nonumber
\overline{\varphi}_{H_{\sigma_k} \#_{\rho} H_{\sigma_l}} = \overline{\varphi}_{{\sigma_l}} \circ \overline{\varphi}_{\overline{H}^{t\sigma_l}_{\sigma_k}},
\end{equation}
and the following diagram commutes.
\begin{equation}\nonumber
\begin{tikzcd}
MC_*(M_{s\sigma_k}, h_{s\sigma_k}) \arrow{d}[swap]{\iota_{\sigma_k}}  \arrow{rr}{\varphi_{{\sigma_k}}} & & MC_*(M_{t\sigma_k}, h_{t\sigma_k})  \arrow{d}{\varphi_{{\sigma_l}}} \\
MC_*(M_{t\sigma_l}, \overline{h}_{s\sigma_k}^{t\sigma_l}) \arrow{rr}{\overline{\varphi}_{H_{\sigma_k} \#_{\rho} H_{\sigma_l}}}
              & &MC_*(M_{t\sigma_l},h_{t\sigma_l})
\end{tikzcd}
\end{equation}
In particular, we have 
\begin{equation}\nonumber
\varphi_{H_{\sigma_k} \#_{\rho} H_{\sigma_l}} =  {\varphi}_{{\sigma_l}} \circ {\varphi}_{{\sigma_k}}.
\end{equation}
\end{corollary}
\begin{proof}
By Theorem \ref{gluthm}, we have $\overline{\varphi}_{H_{\sigma_k} \#_{\rho} H_{\sigma_l}} = \overline{\varphi}_{{\sigma_l}} \circ \overline{\varphi}_{\overline{H}^{t\sigma_l}_{\sigma_k}}$ by counting the orders of both sides of the bijection. The inward direction condition tells us that the composition of $\overline{\varphi}_{H_{\sigma_k} \#_{\rho} H_{\sigma_l} }$ with $\iota_{\sigma_k}$ restricts its domain to $MC_*(M_{s\sigma_k}, h_{s\sigma_k}).$ Then we have
\begin{equation}\nonumber
\overline{\varphi}_{\overline{H}^{t\sigma_l}_{\sigma_k}}|_{MC_*(M_{s\sigma_k}, h_{s\sigma_k})} = \varphi_{H_{\sigma_k}},
\end{equation}
again by the inward direction condition and its image lands in $MC_*(M_{t\sigma_k},h_{t\sigma_k}).$ Thus we have:
$\varphi_{H_{\sigma_k} \#_{\rho} H_{\sigma_l}} = \overline{\varphi}_{H_{\sigma_k} \#_{\rho} H_{\sigma_l}} \circ \iota_{\sigma_{k}} =  (\overline{\varphi}_{{\sigma_l}} \circ \overline{\varphi}_{\overline{H}^{t\sigma_l}_{\sigma_k}} ) \circ \iota_{\sigma_k} = \overline{\varphi}_{\sigma_l} \circ \varphi_{\sigma_k} = {\varphi}_{{\sigma_l}} \circ {\varphi}_{{\sigma_k}}.
$
\end{proof}

\section{Higher Morse homotopy data}

In this section, we construct a special class of parametrized homotopy data. From now on, we distinguish between the notations for homotopies that we use for the case of these special homotopies: $\mathcal{H}_{\bullet},$ and those we use for general case: ${H}_{\bullet}.$ 

\subsection{Transversal higher homotopies}

The finite dimensionality of chain complexes $\bigl\{MC_*(h_a)\bigr\}_{a \in \mathbb{Z}_{\geq 0}}$ and the genericity of transversal homotopy data allow us to have the notion of {\textit{simultaneous}} transversality. That is, we can consider homotopy data that transversality holds at the same time for \textit{all} linearized operators at trajectories connecting pairs of critical points. We call a family of data with this property \textit{transversal homotopy data.}

Suppose that $\{H^t_{\sigma_k}\}_{t \in [0,1]}$ is a smooth 1-parameter family of transversal $k$-homotopies. The existence of such a family (for given homotopies at the ends) is not always guaranteed. Notice that $\{H^t_{\sigma_k}\}_{t \in [0,1]}$ (if it exists) is a \textit{transversal} $(k+1)$-homotopy. 

\begin{lemma}
The $(k+1)$-homotopy $\{H^t_{\sigma_k}\}_{t \in [0,1]}$ is transversal, and for all $x \in Crit(h_{s\sigma_k}),\ y \in Crit(h_{t \sigma_k})$ with $ |y| - |x| = k,$ we have
\begin{equation}\nonumber\nonumber
\mathcal{M}\big(\{H^t_{\sigma_k}\}_{t \in [0,1]}; x , y \big) = \emptyset.
\end{equation}
\end{lemma}

\begin{proof}
Transversality follows from the fact that $H^t_{\sigma_k}$ for each $t$ is transversal, hence the expected dimension (that is, the dimension when transversality assumed) of the moduli space $\mathcal{M}\big(\{H^t_{\sigma_k}\}_{t \in [0,1]}; x , y \big) $ is $|x| - |y| -k = 0$. Suppose that there exists $ (\vec{t}, u) \in \mathcal{M}\big(\{H^t_{\sigma_k}\}_{t \in [0,1]}; x , y \big).$ Then there exists $t' \in [0,1]$ such that $( \vec{t},u) \in \mathcal{M}(H^{t'}_{\sigma_k}; x , y).$ However, transversality of $H^{t'}_{\sigma_k}$ implies that this is impossible, since the expected dimension of the moduli space is $-1$.
\end{proof}

\begin{notation}
Let $H$ and $H'$ be two $k$-homotopies. In our notation,
\begin{equation}\nonumber
H \leadsto H'
\end{equation}
means a 1-parameter family of transversal homotopies connecting $H$ and $H'.$ It is a transversal $(k+1)$-homotopy at the same time.
\end{notation}

\begin{example}\label{egth}  We give examples of (1-parameter families of) transversal homotopies. In what follows, homotopies of Riemannian metrics are omitted, but they should be regarded as being included implicitly in all cases. Let $H_{\sigma_k},$ $H_{\sigma_{k_1}},$ $H_{\sigma_{k_2}},$ and $H_{\sigma_{k_3}} $ denote transversal $k$-, $k_1$-, $k_2$- and $k_3$-homotopies, respectively.
\begin{enumerate}[label=(\alph*)]
\item \textit{Identity homotopies.} The 1-parameter family $H_{\sigma_k}^t := H_{\sigma_k}, \text{for every } t \in [0,1]$ is transversal.
\begin{equation}\nonumber\nonumber
H_{\sigma_k} \leadsto H_{\sigma_k}.
\end{equation}
\item \textit{Inverse homotopies.} Let $\{H^t_{\sigma_k} \}_{t \in [0,1]}$ be a transversal family, then so is $\{H^{1-t}_{\sigma_k} \}_{t \in [0,1]}.$
\item \textit{Reparametrizations of the cubes.}
Consider a reparametrization map $c: [0,1]^{k-1} \rightarrow \prod^{k-1}_{i=1} [c^0_i, c^1_i], \ \vec{t} \mapsto \big(c_1(t_1), \cdots c_{k-1}(t_{k-1})\big)$ with $\frac{\partial c_i (\vec{t})}{\partial t_i} >0$ for each $i,$ $c_i(0) = c^0_i, $ and $ c_i(1) = c^1_i.$ If $H_{\sigma_k}(\vec{t,} \cdot)$ is transversal, then so is $H_{\sigma_k}\big(c(\vec{t}), \cdot \big).$
\item \textit{Increasing gluing parameters I.} Let 
\begin{equation}\nonumber\nonumber
H_{\sigma_{k_1}} \#_{\rho} H_{\sigma_{k_2}} \leadsto H_{\sigma_{k_1}} \#_{\rho'} H_{\sigma_{k_2}}
\end{equation}
with $\rho' > \rho > \rho_{\sigma_{k_1}, \sigma_{k_2},0}$ be a 1-parameter family of glued $(k_1+k_2-1)$-homotopies which is given by smoothly increasing the gluing parameter from $\rho$ to $\rho'.$ 
\item \textit{Gluing with the identity.}
let $H_{\sigma_k} \leadsto H_{\sigma_k}'$ be a 1-parameter family of transversal $k$-homotopies along which $R_{\sigma_k}$ remains the same. We can consider a family of homotopies
\begin{equation}\nonumber\nonumber
H_{\sigma_k} \#_{\rho} H_{\sigma_{k_1}} \leadsto H_{\sigma_k}' \#_{\rho} H_{\sigma_{k_1}}
\end{equation}
which is a 1-parameter family of transversal (if $\rho$ is sufficiently large) $(k+k_1 -1)$-homotopies naturally obtained from $H_{\sigma_k} \leadsto H_{\sigma_k}'.$
\item \textit{Increasing gluing parameters II.} Let $\rho_1$ and $\rho_1'$ be positive real numbers such that $\rho_1' > \rho_1 > \rho_{1,0}.$
Fix $\rho_2>0$ which is sufficiently large so that Theorem \ref{lsurj} holds for the pair $H_{\sigma_{k_1}} \#_{\rho_1} H_{\sigma_{k_2}}$ and $H_{\sigma_{k_3}}.$
Consider a family of homotopies
\begin{equation}\nonumber\nonumber
(H_{\sigma_{k_1}} \#_{\rho_1} H_{\sigma_{k_2}}) \#_{\rho_2} H_{\sigma_{k_3}} \leadsto (H_{\sigma_{k_1}} \#_{\rho'_1} H_{\sigma_{k_2}}) \#_{\rho_2} H_{\sigma_{k_3}} 
\end{equation}
which is obtained by smoothly increasing the gluing parameter $\rho_1$ to $\rho_1',$ keeping $\rho_2$ the same.
One can further assume that $\rho_2$ is large enough so that $(H_{\sigma_{k_1}} \#_{\rho} H_{\sigma_{k_2}}) \#_{\rho_2} H_{\sigma_{k_3}}$ is transversal for all $\rho \in [\rho_1, \rho_1'].$
\item \textit{Translation homotopies.} 
Let $\tilde{s}(t) \in \mathbb{R}$ be a smooth function with $\tilde{s}(0)=0$ and $\tilde{s}'(t)>0.$ Then $H^t_{\sigma_k}(s, \cdot):=H_{\sigma_k}\big(s+\tilde{s}(t), \cdot \big)$ is a transversal family of homotopies
\begin{equation}\nonumber
H_{\sigma_k}(s, \cdot) \leadsto H_{\sigma_k}\big(s+\tilde{s}(1), \cdot \big),
\end{equation}
since for each $t,$ $H_{\sigma_k} \big(s + \tilde{s}(t) \big)$ is transversal. 
\item \textit{Reordering homotopies.} Reordering concatenations is possible via a 1-parameter family of transversal $(k_1 + k_2 +k_3 -2)$-homotopies
\begin{equation}\nonumber
H_{\sigma_{k_1}} \#_{\rho_1} (H_{\sigma_{k_2}} \#_{\rho_2} H_{\sigma_{k_3}}) \leadsto (H_{\sigma_{k_1}} \#_{\rho'_1} H_{\sigma_{k_2}}) \#_{\rho'_2} H_{\sigma_{k_3}}
\end{equation}
under the condition that the gluing parameters $\rho'_1, \rho_2' >0$ are sufficiently large. In fact, it is an easy exercise to check that any such family can be written as repeated concatenations of (f)- and (g)-type homotopies only.
\end{enumerate}
\end{example}

\subsection{Constructions of higher data}

We construct the data $\bigl\{ (\mathcal{H}_{\sigma}, \mathcal{G}_{\sigma})\bigr\}_{\sigma \in N(\mathcal{I})}.$ For simplicity, we will omit homotopy of Riemannian metrics from our notation when we consider homotopy data. All constructions can apply to Riemannian metrics as well.

Starting from the boundary conditions:
\begin{equation}\nonumber
\begin{cases}
H_{(0,1,2)}|_{t=0} = H_{(0,1)} \#_{\rho_{(0,1),(1,2)}} H_{(1,2),} \ H_{(0,1,2)}|_{t=1}  = H_{(0,2)},\\
G_{(0,1,2)}|_{t=0} = G_{(0,1)} \#_{\rho_{(0,1),(1,2)}} G_{(1,2),} \ G_{(0,1,2)}|_{t=1}  = G_{(0,2)},
\end{cases}
\end{equation}
with $\rho_{(0,1), (1,2)} >\rho_{(0,1), (1,2),0},$ we can construct: 
\begin{equation}\nonumber
\begin{cases}
H_{(0,1,2)} : [0,1] \times \mathbb{R} \times M_2 \rightarrow \mathbb{R},\\
G_{(0,1,2)} : [0,1] \times \mathbb{R} \rightarrow Met(M_2),
\end{cases}
\end{equation}
by filling the inside $\{t \mid 0 < t <1\}.$ (Here $Met(M_2)$ denotes the set of all Riemannian metrics on $M_2.$) It can done in such a way that $\big(H_{(0,1,2)}, G_{(0,1,2)}\big)$ is a transversal pair by Theorem \ref{paramtrasv}.

We have
\begin{equation}\label{12htp}
\begin{cases}
\bigl\{ (\mathcal{H}_{f}, \mathcal{G}_{f}) \bigr\}_{f \in Mor(\mathcal{I})},\\
(\mathcal{H}_{(0,1,2)}, \mathcal{G}_{(0,1,2)}),
\end{cases}
\end{equation}
and these are our initial data. In this section, we construct higher degree data that extend these. However, when we try to do the same (i.e., identifying the corners and filling the inside), for instance, with $H_{(0,1,2,3)} : [0,1]^2 \times \mathbb{R} \times M_3 \rightarrow \mathbb{R}, $ its values at the corner (0,0) when we approach in the two different orders do not match:
\begin{equation}\nonumber 
\begin{split}
\big(H_{(0,1,2,3)}|_{t_1 =0}\big) |_{t_2 = 0} = & (H_{(0,1)} \#_{\rho_{\bullet}} H_{(1,2)}) \#_{\rho_{\bullet}} H_{(2,3)}\\ & \neq H_{(0,1)} \#_{\rho_{\bullet}} (H_{(1,2)} \#_{\rho_{\bullet}} H_{(2,3)}) = \big( H_{(0,1,2,3)}|_{t_2 =0} \big) |_{t_1 = 0}
\end{split}
\end{equation}
for $(t_1, t_2) \in [0,1]^{2}.$ The remaining part of this section will mostly be devoted to resolving this issue.

\begin{notation}
We fix our notations :
\begin{equation}\nonumber\nonumber
\begin{cases}
\widehat{\sigma_{l_1} \sigma_{l_2}} & := (f_{a_0, a_1}, \cdots,f_{a_{l_1+l_2}, a_{l_1+l_2+1}}),\\
 & \quad \text{for } \sigma_{l_1} = (f_{a_0,a_1}, \cdots, f_{a_{l_1-1},a_{l_1}}), \ \sigma_{l_2}= (f_{a_{l_1}, a_{l_1+1}}, \cdots, f_{a_{l_1+l_2}, a_{l_1+l_2+1}}),\\
S_{a,b} &:= \bigl\{\sigma \in N(\mathcal{I}) \mid s\sigma = a, \ t\sigma =b\bigr\}, \\
S_{a,-} &:= \bigl\{\sigma \in N(\mathcal{I}) \mid s\sigma = a \bigr\} \ (\text{similarly for } S_{-,b}),\\
\mathring{S}_{a,b} &:= \bigl\{\sigma \in N(\mathcal{I}) \mid s\sigma \geq a, \ t\sigma \leq b\bigr\}, \\
\mathring{S}_{a,-} &:= \bigl\{\sigma \in N(\mathcal{I}) \mid s\sigma \geq a \bigr\} \ (\text{similarly for } \mathring{S}_{-,b} ),\\
\overline{\sigma_k} &:= (s\sigma_k, s\sigma_k +1, \cdots, t\sigma_k) \text{ for } \sigma_k \in N(\mathcal{I})_k,\\
\langle \sigma_k \rangle &:= \bigl\{\tau \in N(\mathcal{I}) \mid s\tau  = s\sigma_k, t\tau = t\sigma_k\bigr\} \text{ for } \sigma_k \in N(\mathcal{I})_k,\\
\rho_{(\sigma; \sigma')} & := \max_{\substack{\tau_1 \in \langle \sigma \rangle \\ \tau_2 \in \langle \sigma' \rangle \\ t \tau_1 = s\tau_2}} \rho_{\tau_1, \tau_2}.
\end{cases}
\end{equation}
\end{notation}

\subsubsection*{A triple induction} We begin our induction steps to construct higher homotopy data. We assume that we are given the data of 1-homotopies $\mathfrak{H}^1.$ 

\begin{notation}
We denote by $\mathcal{D}_{a,b}$ the following data 
\begin{enumerate}[label = \textbullet]
\item $\mathcal{H}_{\sigma}$ with $\sigma \in \mathring{S}_{a,b},$
\item $\rho_{\sigma, \sigma'} \geq \rho_{\sigma, \sigma',0} $ for all $\sigma, \sigma' \in \mathring{S}_{a,b}$ with $|\sigma|, |\sigma'| \geq 1,$ and $t \sigma = s \sigma'$.
\item $(|\sigma| + |\sigma'| -1)$-homotopy $\mathcal{H}_{\sigma} \widetilde{\#}_{\rho_{(\sigma; \sigma')}} \mathcal{H}_{\sigma'}$ together with a transversal homotopy $$\mathcal{H}_{\sigma} \widetilde{\#}_{\rho_{(\sigma; \sigma')}}\mathcal{H}_{\sigma'} \leadsto \mathcal{H}_{\sigma} {\#}_{\rho_{(\sigma; \sigma')}} \mathcal{H}_{\sigma'}.$$ For $|\sigma| = |\sigma'| =1,$ $\widetilde{\#} \equiv \#$ and the homotopy $\leadsto$ is the identity one.
\end{enumerate}
\end{notation}

Note that we already have $\mathcal{D}_{0,2}$ from (\ref{12htp}). Our goal is to obtain $\coprod_{b=2}^{\infty} \mathcal{D}_{0,b}.$ 

\ \

\fbox{\begin{minipage}{33em}
\subsubsection*{Induction I} \textit{We have $\mathcal{D}_{0,2}.$ Given the data $\mathcal{D}_{0,b},$ we construct $\mathcal{D}_{0,b+1}.$ (Induction on $b$)}
\end{minipage}}

\ \

To construct $\mathcal{D}_{0,b+1},$ we specify the elements in $\mathcal{D}'_{a, b+1}$ for $a = 0, \cdots, b,$ which are given as follows. 

\begin{notation}
$ \mathcal{D}'_{a, b+1}$ consists of the following data
\begin{enumerate}[label = \textbullet]
\item $\mathcal{H}_{\sigma}$ with $\sigma \in S_{a, b+1}.$
\item $\rho_{\sigma, \sigma'} \geq \rho_{\sigma, \sigma', 0}$ for all $\sigma \in \mathring{S}_{a,b}, \sigma' \in S_{-, b+1}$ with $|\sigma|, |\sigma'| \geq 1,$ $t \sigma = s \sigma'.$
\item $(|\sigma| + |\sigma'| -1)$-homotopy $\mathcal{H}_{\sigma} \widetilde{\#}_{\rho_{(\sigma; \sigma')}} \mathcal{H}_{\sigma'}$ for all $\sigma \in \mathring{S}_{a,b}, \sigma' \in S_{-, b+1}$ with $|\sigma|, |\sigma'| \geq 1$ and $t \sigma = s \sigma'$ together with a transversal homotopy $$\mathcal{H}_{\sigma} \widetilde{\#}_{\rho_{(\sigma; \sigma')}}\mathcal{H}_{\sigma'} \leadsto \mathcal{H}_{\sigma} {\#}_{\rho_{(\sigma; \sigma')}} \mathcal{H}_{\sigma'}.$$ For $|\sigma| = |\sigma'| =1,$ we put $\widetilde{\#} \equiv \#$ and the homotopy $\leadsto$ is given by the identity one.
\end{enumerate}
\end{notation}
Note that $\mathcal{D}_{0,b+1} = \coprod_{a = 0}^{b} \mathcal{D}'_{a, b+1},$ and that $\mathcal{D}'_{b, b+1}$ is given by the 1-homotopy data $\mathfrak{H}^1.$

\ \

\fbox{\begin{minipage}{33em}
\subsubsection*{Induction II} \textit{We have $\mathcal{D}'_{b, b+1}.$ Given the data $\mathcal{D}'_{a, b+1},$ and we construct $\mathcal{D}'_{a-1, b+1}.$ (Induction on $a$)}
\end{minipage}}

\ \

To construct $\mathcal{D}'_{a-1, b+1},$ we specify the elements in $\mathcal{D}^{l}_{a-1, b+1}$ for each $l \geq 1,$ which are given as follows. 

\begin{notation}
$\mathcal{D}^{l}_{a-1, b+1}$ consists of the following data
\begin{enumerate}[label = \textbullet]
\item $\mathcal{H}_{\sigma}$ with $\sigma \in S_{a-1, b+1}$ and $|\sigma| = l,$
\item $\rho_{\sigma, \sigma'} \geq \rho_{\sigma, \sigma', 0}$ for all $\sigma \in \mathring{S}_{a-1, b}, \sigma' \in S_{-, b+1},$ $t \sigma = s \sigma', |\sigma|, |\sigma'| \geq 1,$ $|\sigma| + |\sigma'| = l.$
\item $(l-1)$-homotopy $\mathcal{H}_{\sigma} \widetilde{\#}_{\rho_{(\sigma; \sigma')}} \mathcal{H}_{\sigma'}$  for all $\sigma \in \mathring{S}_{a-1, b}, \sigma' \in S_{-, b+1}$ with $t \sigma = s \sigma', |\sigma|, |\sigma'| \geq 1,$ and $|\sigma| + |\sigma'| = l$ endowed with a transversal homotopy 
$$\mathcal{H}_{\sigma} \widetilde{\#}_{\rho_{(\sigma; \sigma')}}\mathcal{H}_{\sigma'} \leadsto \mathcal{H}_{\sigma} {\#}_{\rho_{(\sigma; \sigma')}} \mathcal{H}_{\sigma'}.$$
For $|\sigma| = |\sigma'| =1,$ we put $\widetilde{\#} \equiv \#$ and the homotopy $\leadsto$ is given by the identity one.
\end{enumerate}
So we have $\mathcal{D}'_{a-1, b+1} = \coprod_{l \geq 1} \mathcal{D}^l_{a-1, b+1},$ and $\mathcal{D}^{1}_{a-1, b+1}$ is determined by the given 1-homotopy data $\mathfrak{H}^1.$

\end{notation}

\fbox{\begin{minipage}{33em}
\subsubsection*{Induction III} \textit{We have $\mathcal{D}^1_{a-1, b+1}.$ Given the data $\mathcal{D}^{\leq l}_{a-1, b+1},$ and we construct $\mathcal{D}^{\leq l+1}_{a-1, b+1}.$ (Induction on $l$)}
\end{minipage}}

\ \

For $\sigma_{l+1} \in N(\mathcal{I})_{l+1},$ let $\sigma_{l_1} \in N(\mathcal{I})_{l_1},\sigma_{l_2} \in N(\mathcal{I})_{l_2}$ be simplices such that $t \sigma_{l_1} = s \sigma_{l_2},$ and $\widehat{\sigma_{l_1} \sigma_{l_2}} = \sigma_{l + 1}.$ Then we have

\begin{itemize}
\item $\mathcal{H}_{\sigma_{l_1}}$ (from the hypothesis of Induction I)
\item $\mathcal{H}_{\sigma_{l_2}}$ (from the hypothesis of Induction II)
\item $\rho_{\sigma_{l_1}, \sigma_{l_2}}$ (from the hypothesis of Induction III), so for each such pair $(\sigma_{l_1}, \sigma_{l_2}),$ we can consider $\mathcal{H}_{\sigma_{l_1} \#_{\rho_{\sigma_{l_1}, \sigma_{l_2}}}}\mathcal{H}_{\sigma_{l_2}}.$
\item $\mathcal{H}_{\partial_i \sigma_{l+1}}$ for every $i =1, \cdots, l$ (from the hypothesis of Induction III)
\end{itemize}

\subsubsection*{The faces of cubes} For $\sigma_l \in N(\mathcal{I})_l,$ we consider all the pairs $(\sigma_{l_1}, \sigma_{l_2})$ such that $l_1 + l_2 = l$ and $\sigma_{l_1 + l_2} := \widehat{\sigma_{l_1} \sigma_{l_2}}.$ For each such pair, we think of $\mathcal{H}_{\sigma_{l_1}} \#_{\rho_{(\sigma_{l_1}; \sigma_{l_2})}} \mathcal{H}_{\sigma_{l_2}}$ as an $(l_1 + l_2 - 2)$-homotopy, which are our initial data.

We further consider $(l_1 + l_2 -1)$-homotopies denoted by (a) $\mathcal{H}_{\sigma_{l_1}} \#'_{\rho_{(\sigma_{l_1}; \sigma_{l_2})}} \mathcal{H}_{\sigma_{l_2}}$ and (b) $\mathcal{H}'_{\partial_i \sigma_{l_1+l_2}}$ from the following processes: We fix a small number $0 < \epsilon \ll 1$. 

\subsubsection*{\emph{(a)} $\mathcal{H}_{\sigma_{l_1}} \#'_{\rho_{(\sigma_{l_1}; \sigma_{l_2})}} \mathcal{H}_{\sigma_{l_2}}$}
\begin{enumerate}
\item Contraction along the diagonal of the cube by $1- \epsilon$.
\item Expansion of the half of the faces on the side of the origin $(0, \cdots, 0)$ toward the original ones in the normal directions by the identity homotopies.
\end{enumerate}

\subsubsection*{\emph{(b)} $\mathcal{H}'_{\partial_i \sigma_{l_1+l_2}}$}
\begin{enumerate}
\item The same as (1) of (a).
\item The same as (2) of (a).
\item Filling the gaps between the expanded widths by the identity homotopies. (This is possible since we used the identity homotopies in (2).)
\end{enumerate}

\begin{figure}[h!]
  \includegraphics[width=\linewidth]{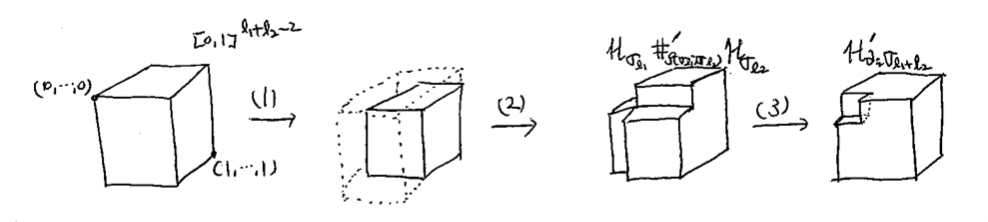}
  \caption{Homotopies of type (a) and (b)}
  \label{}
\end{figure}

\subsubsection*{Identifications of the faces} The next step is to identify faces of the forms $\mathcal{H}_{\bullet} \#_{\rho_{\bullet}}' \mathcal{H}_{\bullet}$ and $\mathcal{H}_{\partial_* \bullet}'.$ There are three kinds of such face identifications. For $\sigma_i \in N(\mathcal{I})_i, \sigma_j \in N(\mathcal{I})_j$ with $i <j$ with $\widehat{\sigma_i \sigma_{l_1 + l_2-i}} = \widehat{\sigma_j \sigma_{l_1 + l_2-j}} = \sigma_{l_1 + l_2} \in N(\mathcal{I})_{l_1 +l_2},$
\begin{enumerate}[label = (\roman*)]
\item $\mathcal{H}_{\sigma_i} \#'_{\rho_{(\sigma_i ; \sigma_{l_1 + l_2-i})}} \mathcal{H}_{\sigma_{l_1 + l_2-i}} |_{t_j = 1} \leadsto \mathcal{H}'_{\partial_j \sigma_{l_1 + l_2}}|_{t_i = 0},$
\item $\mathcal{H}'_{\partial_i \sigma_{l_1 + l_2}}|_{t_j =1} \leadsto \mathcal{H}'_{\partial_j \sigma_{l_1 + l_2}}|_{t_i =1}$
\item $\mathcal{H}_{\sigma_i} \#'_{\rho_{(\sigma_i ; \sigma_{l_1 + l_2-i})}} \mathcal{H}_{\sigma_{l_1 + l_2-i}} |_{t_j = 0} \leadsto \mathcal{H}_{\sigma_j} \#'_{\rho_{(\sigma_j ; \sigma_{l_1 + l_2-j})}} \mathcal{H}_{\sigma_{l_1 + l_2-j}} |_{t_i = 0}$
\end{enumerate}
(i) and (ii) can be done simply by identifying the faces. To deal with the case (iii), we apply the following concatenations of homotopies.
\begin{equation}\nonumber
\begin{split}
 \big( \mathcal{H}_{\sigma_i} & \#'_{\rho_{(\sigma_i ; \sigma_{l_1 + l_1-i})}} \mathcal{H}_{\sigma_{l_1 + l_2-i}} \big) \big((1-\varepsilon) \vec{t}, \cdot \big)|_{t_j = 0}\\ 
 & = \mathcal{H}_{\sigma_i} \#_{\rho_{(\sigma_i ; \sigma_{l_1+ l_2} - j)}} ( \mathcal{H}_{\sigma_{l_{j-i}}} \widetilde{\#}_{\rho_{(\sigma_{j-i}; \sigma_{l_1+ l_2 -j +1})}} \mathcal{H}_{\sigma_{l_{l_1+ l_2 -j +1}}} )(\vec{t}, \cdot)\\ 
 & \leadsto \mathcal{H}_{\sigma_i} \#_{\rho_{(\sigma_i ; \sigma_{l_1+ l_2} - j)}} ( \mathcal{H}_{\sigma_{l_{j-i}}} {\#}_{\rho_{{(\sigma_{j-i}; \sigma_{l_1+ l_2 -j +1})}}} \mathcal{H}_{\sigma_{l_{l_1+ l_2 -j +1}}})(\vec{t}, \cdot)\\ 
 & \text{ (from the hypothesis of Induction II),}\\
 & \leadsto (\mathcal{H}_{\sigma_i} \#_{P_1} \mathcal{H}_{\sigma_{l_{j-i}}}){\#}_{{P_2}} \mathcal{H}_{\sigma_{l_{l_1+ l_2 -j +1}}}(\vec{t}, \cdot)\\ 
 \end{split}
 \end{equation}
 \begin{equation}\nonumber
 \begin{split}
 &\text{ (Reordering homotopy for large $P_1, P_2$), } \\
 & \leadsto (\mathcal{H}_{\sigma_i} \#_{\rho_{(\sigma_i; \sigma_{l_{j-i}})}} \mathcal{H}_{\sigma_{l_{j-i}}}){\#}_{\rho_{(\sigma_{j-1} ; \sigma_{l_1+ l_2 -j +1})}} \mathcal{H}_{\sigma_{l_{l_1+ l_2 -j +1}}}(\vec{t}, \cdot)\\ 
 &\text{  (The inverse of growing parameter homotopy) }\\
 & \leadsto (\mathcal{H}_{\sigma_i} \widetilde{\#}_{\rho_{(\sigma_i; \sigma_{l_a})}} \mathcal{H}_{\sigma_{l_a}}){\#}_{\rho_{(\sigma_{j-1} ; \sigma_{l_1+ l_2 -j +1} )}} \mathcal{H}_{\sigma_{l_{l_1+ l_2 -j +1}}}(\vec{t}, \cdot)\\ 
 & \text{  (The hypothesis Induction I) }\\
  & =  \big( \mathcal{H}_{\sigma_{j-i}}  \#'_{\rho_{(\sigma_{j-1} ; \sigma_{l_1+l_2-j})}} \mathcal{H}_{\sigma_{l_{l_1+ l_2 -j +1}}} \big)\big((1 - \varepsilon)\vec{t}, \cdot \big) |_{t_i = 0}.
\end{split}
\end{equation}

\subsubsection*{Filling the inside} As a result, we obtain a homotopy parametrized by $\partial [0,1]^{l_1+l_2} \setminus \text{int}(V),$ where $V \simeq D^{l_1+l_2-1}$ is a closed contractible space. We need to fill $V$ with homotopies. From now any homotopy for this purpose will always mean a transversal one.

Observe that the homotopies corresponding to the points on $\partial V$ are given by some faces $\{ t_{\bullet}=0 \} $ of those of (i) to (iii). We try to fill the inside of $V$ by repeatedly using the homotopy of type $\widetilde{\#} \leadsto \#$ and expansion at the faces $\{ t_{\bullet'}=0 \}.$ We repeat this process until we end up with a closed subset $V' \subset V$ that is again contractible and the homotopies at each of the points on $\partial V$ is (up to translation homotopies) of the form $ \cdots (h_{f_1} \#_{\rho'_1} (h_{f_2}  \#_{\rho''_2} (h_{f_3}  \cdots ))) \#_{\rho'_{l_1 +l_2 - 2}} h_{h_{l_1 + l_2 -2}},$ which is independent of the variables of the cube. Here all the homotopies $h_{\bullet}$ are 1-homotopies, i.e., $f_{\bullet} \in Mor(\mathcal{I}),$ and their gluing is in an {\textit{arbitrary}} order, which varies, depending on where in $V'$ the resulting homotopy corresponds to. We can further fill the inside of $V'$ by reordering homotopies, which can be chosen as a combination of translations and increasing gluing parameters. (See Example \ref{egth} (h).) Note that such choices can be made smoothly. Finally, we obtain a homotopy of the form $( \cdots (h_{f_1} \#_{\rho''_1} h_{f_2})  \#_{\rho''_2} h_{f_3} ) \cdots )  \#_{\rho''_{l_1 +l_2 - 1}} h_{h_{l_1 + l_2 -2}}$ with a {\textit{fixed}} order (starting from the left most one and ending at the right most one) and sufficiently large gluing parameters $\rho''_1, \cdots, \rho''_{l_1 + l_2 -1}$. It is independent of the variables from the cube $[0,1]^{l_1 + l_2 -1}.$ Now we fill the empty ball with this constant homotopy.

Thus we obtain a homotopy parametrized by $ [0,1]^{l_1 + l_2 -1}$ with a generic choice of homotopies for transversality. We denote this transversal $(l_1 + l_2 -1)$-homotopy pair by $(\mathcal{H}_{\sigma_{l_1 + l_2}}, \mathcal{G}_{\sigma_{l_1 + l_2}}),$ where $\sigma_{l_1 + l_2}:= \widehat{\sigma_{l_1} \sigma_{l_2}}.$ We denote 
\begin{equation}\nonumber
\mathcal{H}_{\sigma_{l_1}} \widetilde{\#}_{\rho_{(\sigma_{l_1} ; \sigma_{l_2})}} \mathcal{H}_{\sigma_{l_2}} := \mathcal{H}_{\sigma_{l_1 + l_2}} |_{t_{l_1}= 0}.
\end{equation}
 Also, we observe that there is a homotopy $\mathcal{H}_{\sigma_{l_1}} \widetilde{\#}_{\rho_{(\sigma_{l_1} ; \sigma_{l_2})}} \mathcal{H}_{\sigma_{l_2}} \leadsto \mathcal{H}_{\sigma_{l_1}} {\#}_{\rho_{(\sigma_{l_1} ; \sigma_{l_2})}} \mathcal{H}_{\sigma_{l_2}}$ from the expansion as is depicted in Figure 4.
\begin{figure}[h!]
  \includegraphics[height=2cm]{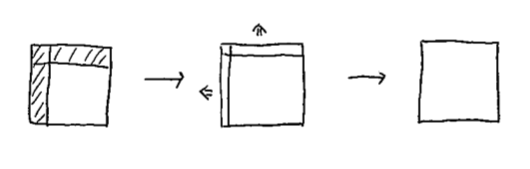}
  \caption{The homotopy $\widetilde{\#} \leadsto \#$ by expansion}
  \label{}
\end{figure}
The shaded region of the cube is where no solution is allowed for the corresponding parameters. By extending this cube and cutting out the shaded region we get the whole cube that gives rise to $\mathcal{H}_{\bullet} \#_{\rho_{\bullet}} \mathcal{H}_{\bullet}$. It is clear that each homotopy along this process is transversal, hence the homotopy $\mathcal{H}_{\sigma_{l_1}} \widetilde{\#}_{\rho_{(\sigma_{l_1} ; \sigma_{l_2})}} \mathcal{H}_{\sigma_{l_2}} \leadsto \mathcal{H}_{\sigma_{l_1}} {\#}_{\rho_{(\sigma_{l_1} ; \sigma_{l_2})}} \mathcal{H}_{\sigma_{l_2}}$ itself is transversal. Similarly, by construction of $\partial \mathcal{H}'_{\partial_i \sigma_{k}}$ we also have expansions that give rise to a transversal homotopy
\begin{equation}
\mathcal{H}'_{\partial_i \sigma_{k}} \leadsto \mathcal{H}_{\partial_i \sigma_{k}}.
\end{equation}
The following lemma immediately follows. 
\begin{lemma}\label{isotil}
We have bijections among finite sets:
\begin{equation}\nonumber
\begin{split}
\mathcal{M}(\mathcal{H}_{\sigma_k}|_{t_i=0};x,y) = \mathcal{M}(\mathcal{H}_{(\sigma_k)^1_i} & \widetilde{\#}_{\rho_{((\sigma_{k})^1_i; (\sigma_k)^2_{k-i})}} \mathcal{H}_{(\sigma_k)^2_{k-i}}; x, y) \\ 
& \simeq \mathcal{M}(\mathcal{H}_{(\sigma_k)^1_i} \#_{\rho_{((\sigma_{k})^1_i; (\sigma_k)^2_{k-i})}} \mathcal{H}_{(\sigma_k)^2_{k-i}}; x, y),\\
\mathcal{M}(\mathcal{H}_{\sigma_k}|_{t_i=1};x,y) = \mathcal{M}(\mathcal{H}'_{\partial_i \sigma_k};&x,y) \simeq \mathcal{M}(\mathcal{H}_{\partial_i \sigma_k};x,y),
\end{split}
\end{equation} 
for all $i = 1, \cdots, k-1, \ x \in Crit(\overline{h}^{t \sigma_k}_{s \sigma_k}), \ y \in Crit(h_{t\sigma_k})$ with $|y|-|x|+k-2=0.$ 
\end{lemma}

As a result, we can complete Induction III, hence Inductions II and I as well, and we obtain a collection $\bigl\{(\mathcal{H}_{\sigma}, \mathcal{G}_{\sigma})\bigr\}_{\sigma \in N(\mathcal{I})}.$
\begin{definition}
We call the following collection of data \textit{Morse higher homotopy data,} and denote by $\mathfrak{H}$:
\begin{itemize}[label = \textbullet]
\item $\bigl\{(M_a, h_a, g_a)\bigr\}_{a \in N(\mathcal{I})_0},$ an exhaustion $\mathcal{E}$ of $W$,
\item $\bigl\{(\mathcal{H}_{\sigma_k}, \mathcal{G}_{\sigma_k})\bigr\}_{\sigma_k \in N(\mathcal{I})_k},$ all $k$-homotopies and $k \geq 1,$
\item $\{\rho_{\sigma_k,\sigma_l}>\rho_{\sigma_k,\sigma_l,0} \},$ gluing parameters for all $(\sigma_k,\sigma_l) \in N(\mathcal{I})_k \times N(\mathcal{I})_l$ with $t \sigma_k=s \sigma_l$ and for all $k, l \geq 1.$
\end{itemize}
\end{definition}

\begin{definition}
We say that higher data are {\textit{strict}} if the resulting continuation maps satisfy
\begin{equation}\nonumber
\varphi_{\sigma} =0, \text{ for all } \sigma \in N(\mathcal{I}) \text{ with } |\sigma| \geq 2.
\end{equation}
\end{definition}

\begin{remark}
(i) Any higher data $\mathfrak{H}$ contain some 1-data $\mathfrak{H}^1$ as a subset. (ii) Strictness of $\mathfrak{H}$ implies that of $\mathfrak{H}^1$ (c.f. Definition \ref{strict1h}).
\end{remark}

\section{Compactification of the moduli spaces}

In this section, we fix Morse higher homotopy data $\mathfrak{H}$ and study a compactification of the corresponding parametrized moduli space of trajectories, say $\mathcal{M}(\mathcal{H}_{\sigma_k}; x,y)$ for some critical points $x, y$ and a simplex $\sigma_k \in N(\mathcal{I})_k,$ with $k \geq 1.$

\subsection{Weak convergences}

We first recall the Arzel\'{a}-Ascoli theorem:
\begin{theorem}[Arzel\'{a}-Ascoli]\label{aathm}
Let $X$ and $Y$ be metric spaces with $X$ being compact. Let $S \subset C^{\infty}(X,Y)$ be a countable family of smooth maps. Then $S$ is equicontinuous and pointwisely relatively compact if and only if it is relatively compact.
\end{theorem}

\begin{assumption}\label{assm}
We consider $X = \overline{\mathbb{R}}, \ Y = [0,1]^{k-1} \times M_{t \sigma_k}.$ Our subset of interests  $\{u_n\} \subset C^{\infty}(\overline{\mathbb{R}}, [0,1]^{k-1} \times M_{t \sigma_k})$ consists of maps that satisfy
$$\text{ $pr_1 \circ u_n(s) \in C^{\infty}(\overline{\mathbb{R}}, [0,1]^{k-1})$ stays constant along the trajectory,}$$ 
where $pr_1$ is the projection to the first component.
\end{assumption}

Let $\bigl\{( \vec{t}_n, u_n)\bigr\}_n \subset \mathcal{M}(\mathcal{H}_{\sigma_k}; x, y)$ be a subset that satisfies Assumption \ref{assm}. 

\begin{proposition}\label{equicont}
Then $\bigl\{( \vec{t}_n, u_n)\bigr\}_n$ is equicontinuous in $C^0(\overline{\mathbb{R}}, M_{t \sigma_k}).$ In other words, for any $\epsilon >0,$ there exists $\delta>0$ such that $| s-s' | < \delta$ implies \\ $d\big( \big(\vec{t}_n, u_n(s)\big), \big(\vec{t}_n, u_n(s')\big)\big) < \epsilon$ for every $ n.$
\end{proposition}
\begin{proof}
We have
\begin{equation}\label{enesti}
\begin{split}
d\big(&(\vec{t}_n, u_n(s)), (\vec{t}_n, u_n(s'))\big) = d\big(u_n(s), u_n(s')\big)\\ 
& := \inf\limits_{\substack{l : [s, s'] \rightarrow M_{t\sigma_k},\\ l(s) = u^{\vec{t}_n}_n(s),\\ l(s')= u^{\vec{t}_n}_n(s')}} \big( \int^{s'}_{s}|\dot{l}(s)|ds) \big) \leq \int^{s'}_{s}|\dot{u}_n(s)|ds \\
& \stackrel{\text{(H\"{o}lder inequality)}}{\leq} \sqrt{|s' - s|} \cdot \sqrt{\int^{s'}_{s}|\dot{u}_n(s)|^2ds} \leq \sqrt{|s' - s|} \cdot \sqrt{\int^{\infty}_{-\infty}|\dot{u}_n(s)|^2ds}\\
\end{split}
\end{equation}
\begin{equation}\nonumber
\begin{split}
\quad \quad &\leq \sqrt{|s'-s|} \cdot \sqrt{\int^{\infty}_{-\infty} \frac{|\nabla h_{\sigma_k}|^2}{{1 + |\dot{h}_{\sigma_k}|^2 |\nabla h_{\sigma_k}|^2}} \circ u ds}\\
& \leq \sqrt{|s' - s|} \cdot \sqrt{\int^{\infty}_{-\infty} \frac{|\nabla h_{\sigma_k}|^2}{\sqrt{1 + |\dot{h}_{\sigma_k}|^2 |\nabla h_{\sigma_k}|^2}} \circ u ds} =  \sqrt{|s' - s|} \cdot \sqrt{\int^{\infty}_{-\infty} - \nabla h_{\sigma_k} \cdot \dot{u} ds}\\
& {=} \sqrt{|s' - s|} \cdot \sqrt{ \int^{\infty}_{-\infty} - \frac{d}{ds} h_{\sigma_k}(\vec{t}, s, u(s)) ds + \int^{\infty}_{-\infty} \dot{h}_{\sigma_k}(\vec{t}, s, u(s)) ds}\\
&  = \sqrt{|s' - s|} \cdot \sqrt{ - h_{t \sigma_k}(y) + \overline{h}_{s \sigma_k}^{t \sigma_k}(x) +  \int^{R_{\sigma_k}}_{-R_{\sigma_k}} \dot{h}_{\sigma_k}(\vec{t}, s, u(s)) ds}\\
& \leq \sqrt{|s' - s|} \cdot \sqrt{- h_{t \sigma_k}(y) + \overline{h}_{s \sigma_k}^{t \sigma_k}(x) + 2 C \cdot R_{\sigma_k}}.
\end{split}
\end{equation}
for some $C>0.$ (Here $\dot{h}_{\bullet}$ means $\frac{\partial h_{\bullet}}{\partial s}$.)
\end{proof}

We can apply Theorem \ref{aathm} to our situation.
\begin{theorem}\label{cptwc}
$\bigl\{(\vec{t}_n, u_n)\bigr\}_n \subset \mathcal{M}({H}_{\sigma_k}; x, y)$ is relatively compact, and it has a weakly converging subsequence:
\begin{equation}\nonumber\nonumber
(\vec{t}_{n_m},u_{n_m}) \xrightarrow{C^{\infty}_{loc}} (\vec{t}_v, v), \text{ for some } \vec{t}_v \in [0,1]^{k-1} \text{ and } v \in C^{\infty}(\overline{\mathbb{R}}, M_{t \sigma_k}) \text{ as } m \rightarrow \infty.
\end{equation}
Hence, there exists $v \in C^{\infty}(\overline{\mathbb{R}}, M_{t \sigma_k})$ such that 
\begin{equation}\nonumber
u_{n_m}|_{[-R, R]} \xrightarrow{C^l([-R,R])} v|_{[-R, R]}, \text{ for all } R >0, \ l \in \mathbb{N}.
\end{equation}
\end{theorem}

\begin{proof}
Since $[0,1]^{k-1} \times M_{t\sigma_k}$ is compact, we know $\bigl\{ \big( \vec{t}_n, u_n(s) \big)\bigr\}_n$ is relatively compact pointwisely at each $s \in {\mathbb{R}}$. Then from Theorem \ref{aathm} and Proposition \ref{equicont}, it follows that $\bigl\{ \big( \vec{t}_n, u_n(s) \big)\bigr\}_n$ is relatively compact, and we have
\begin{equation}\nonumber
u_{n_m} \xrightarrow{C^0_{loc}} v \in C^0(\overline{\mathbb{R}}, M_{t \sigma_k}).
\end{equation}
That is,
\begin{equation}\nonumber
u_{n_m}|_{[-R, R]} \xrightarrow{C^0([-R,R])} v|_{[-R, R]} \text{ for all } R >0.
\end{equation}
But since $u_{n_m}$ is a solution of the first order ordinary differential equation $\dot{u}_{n_m} = - X \circ u_{n_m}$ for some vector field $X,$ it follows that
\begin{equation}\nonumber
u_{n_m} \xrightarrow{C^{l}_{loc}} v \in C^{\infty}(\overline{\mathbb{R}}, M_{t \sigma_k}),
\end{equation}
and so on for each $l \geq 1.$ Hence by Assumption \ref{assm} the assertion of the theorem follows.
\end{proof}

\begin{theorem}\label{thm15}
In the setting of Theorem \ref{cptwc}, if $(\vec{t}_v, v)$ is an element of the space of parametrized trajectories $\mathcal{M}(\mathcal{H}_{\sigma_k}; x, y),$ then we have a convergence in $W^{1,2}:$
\begin{equation}\nonumber\nonumber
(\vec{t}_{n_m}, u_{n_m}) \xrightarrow{W^{1,2}} (\vec{t}_{v}, v) \in \mathcal{M}(\mathcal{H}_{\sigma_k}; x, y).
\end{equation}
\end{theorem}
\begin{proof-sketch}
The proof is essentially the same as that of Lemma 2.39 of [Sch]. We can show uniform asymptotic exponential decay property: $ | u_{local, n_m}(s)|  \leq  c \cdot e^{- \lambda |s|}$ for some $c, \lambda >0$ and for all $s$ with $|s| \geq R_{\sigma_k}.$ In the local coordinate of the Banach manifold, we write $u_{n} = \exp_{v_n}(\xi_{n}), \ \xi_{n} \in W^{1,2}\big(v_n^*(TM_{t \sigma_k})\big),$ and then the weak convergence $\xi_n \xrightarrow{C^{\infty}_{loc}} 0$ implies the desired $W^{1,2}$-convergence.
\end{proof-sketch}

\subsection{Compactness theorem and higher continuation maps}

Now we state our main theorem in this section.

\begin{theorem}\label{cptthm}
Let $\bigl\{ (\vec{t}_n, u_n) \bigr\}_n \subset \mathcal{M}(\mathcal{H}_{\sigma_k}; x, y)$ be a sequence of parametrized solutions for $x \in Crit(\overline{h}^{t \sigma_k}_{s \sigma_k}),$ $y \in Crit(h_{t \sigma_k}),$ $\sigma_k \in N(\mathcal{I})_k,$ and  $k \geq 1.$ Then there exist: 
\begin{enumerate}[label = (\roman*)]
\item $x_0 (= x), \cdots, x_l \in Crit(\overline{h}^{t \sigma_k}_{s \sigma_k})$ and $y_0, \cdots, y_{l'} (= y) \in Crit(h_{t \sigma_k}),$
\item \begin{equation}\nonumber
\begin{split}
&u^i_{s\sigma_k} \in \mathcal{M}(\overline{h}_{s \sigma_k}; x_i, x_{i+1}), \ i = 0, \cdots, l-1,\\
&(\vec{t}_v, v) \in \mathcal{M}(\mathcal{H}_{\sigma_k}; x_l, y_{0}),\\
&u^j_{t\sigma_k} \in \mathcal{M}({h}_{t \sigma_k}; y_j, y_{j+1}), \ j = 0, \cdots, l'-1,\\
\end{split}
\end{equation}
and
\item $\{\tau^i_{s \sigma_k, n}\}_n, \ \{\tau^j_{t \sigma_k, n} \}_n \subset \mathbb{R}$
\end{enumerate} 
such that the following hold.
\begin{enumerate}[label = (\alph*)]
\item 
\begin{equation}\nonumber
\begin{split}
&u_n(\cdot + \tau^i_{s \sigma_k,n}, \cdot) \xrightarrow{C^{\infty}_{loc}}u^i_{s \sigma_k}, \ i = 0, \cdots, l-1,\\
&u_n \xrightarrow{C^{\infty}_{loc}}v,\\
&u_n(\cdot + \tau^j_{t \sigma_k,n}, \cdot) \xrightarrow{C^{\infty}_{loc}}u^j_{t \sigma_k}, \ j = 0, \cdots, l'-1,\\
\end{split}
\end{equation}
\item 
\begin{equation}\nonumber
\begin{split}
&|x_0| > |x_1| > \cdots > |x_l|,\\
&|x_l| > |y_0| - k + 1,\\
&|y_0| > |y_1| > \cdots > |y_{l'}|.
\end{split}
\end{equation}
\end{enumerate}
\end{theorem}

\begin{proof}
Theorem \ref{thm15} says that we have a subsequence $\bigl\{(\vec{t}_{n_m}, u_{n_m})\bigr\}_m$ weakly converging to $(\vec{t}_{u'}, u').$ Then we can check that $(\vec{t}_{u'}, u')$ is an element of $\mathcal{M}(\mathcal{H}_{\sigma_k}; x', y')$ for some $x'$ and $y'.$
In the same way as (\ref{enesti}), we can show $\int^{\infty}_{\infty} | \dot{u}'|^2 ds \leq - h_{t \sigma_k}(y') + \overline{h}_{s \sigma_k}^{t \sigma_k}(x') + \int^{R_{\sigma_k}}_{\sigma_k}\dot{h}_{\sigma_k}\big(\vec{t}, s, u'(s)\big) ds < \infty$ Hence we know $\dot{u}'(\pm \infty)= 0 = \frac{-\nabla h_{\sigma_k} \circ u'}{(\cdots)},$ which implies that $x' = u'(-\infty), \ y' =u'(\infty)$ are critical points of $\overline{h}^{t \sigma_k}_{s \sigma_k}$ and $h_{\sigma_k},$ respectively. If $x = x'$ and $y =y',$ then we are done. 

In general, we have $|x| \geq |x'|$ and $|y'| \geq |y|.$ Suppose that $|x| > |x'|.$ Then for $c \in [\overline{h}^{t \sigma_k}_{s \sigma_k}(x'), \overline{h}^{t \sigma_k}_{s \sigma_k}(x)], $ we consider the reparametrization $\mathcal{H}_{\sigma_k}(\cdot, \cdot + \tau_{n_m}),$ where $\tau_{n_m}< - R_{\sigma_k},$ so that we have $\mathcal{H}_{\sigma_k}(\vec{t}, \tau_{n_m}) = h_{s \sigma_k} (\tau_{n_m}) = c.$ We let $\tau_{n_m} := \tau^{\bullet}_{\sigma_k, m}.$ Then by Theorem \ref{aathm}, we have a weak convergence $\bigl\{(\vec{t}_{n_m}, u_{n_m})\bigr\}_m \xrightarrow{C^{\infty}_{loc}} (\vec{t}_{u''}, u'') \in \mathcal{M}(\overline{h}^{t \sigma_k}_{s \sigma_k}; x'', x''')$ with $|x'| \leq |x''| \leq |x'''| \leq |x|.$ We can repeat this process (and also for $y$). It must end in finite steps because of transversality of $\mathcal{H}_{\sigma_k}$ and Corollary \ref{givenk}. Otherwise, at some point the expected dimension becomes negative, which is impossible.
\end{proof}

\subsubsection*{0- and 1-dimensional cases}
\begin{corollary}\label{finsedi0}
If $|y|-|x| = k-1,$ then $\mathcal{M}(\mathcal{H}_{\sigma_k};x, y)$ is a finite set.
\end{corollary}
\begin{proof}
By construction of the homotopy data $\bigl\{ (\mathcal{H}_{\sigma}, \mathcal{G}_{\sigma})\bigr\}_{\sigma \in N(\mathcal{I})}$, we know the space of parametrized trajectories $\mathcal{M}(\mathcal{H}_{\sigma_k};x, y)$ is a manifold of dimension $|x|-|y| + k -1 =0.$ Also, Theorem \ref{cptthm} says that any sequence in $\mathcal{M}(\mathcal{H}_{\sigma_k}; x, y)$ has a subsequence that weakly converges to one of its elements, since no breaking can take place for reasons of degree. For metric spaces, sequential compactness implies compactness, and we conclude that $\mathcal{M}(\mathcal{H}_{\sigma_k};x,y)$ is compact, hence a finite set.
\end{proof}

We now apply this analytic method that we have developed until now to obtain algebraic structures. We remark that the Fredholm properties for the linearized operators are ensured to hold for the homotopy data $\bigl\{ (\mathcal{H}_{\sigma}, \mathcal{G}_{\sigma})\bigr\}_{\sigma \in N(\mathcal{I})},$ which is due to our inductive construction (c.f. Remark \ref{rmk56}).

Recall that higher continuation maps
\begin{equation}\nonumber
\varphi_{\sigma_k} : MC_*(h_{s \sigma_k}) \rightarrow MC_*(h_{t \sigma_k})
\end{equation}
are defined by the composition $\varphi_{\sigma_k} : = \overline{\varphi}_{\sigma_k} \circ {\iota}_{\sigma_k}$ of linear extensions of the generator-level assignment:
\begin{equation}\nonumber
\overline{\varphi}_{\sigma_k}(x) := \sum\limits_{\substack{y \in Crit(h_{t \sigma_k}), \\ |y| - |x| - k+1 =0 }}\#_2 \mathcal{M}({\mathcal{H}}_{\sigma_k}; x, y) \cdot y \text{ for } x \in Crit(\overline{h}^{t \sigma_j}_{s \sigma_k}),
\end{equation}
and the inclusion map:
\begin{equation}\nonumber
\iota_{\sigma_k} : MC_*(h_{s \sigma_k}) \hookrightarrow MC_*(\overline{h}^{t \sigma_k}_{s \sigma_k}).
\end{equation}
Lemma \ref{lem2} implies that it also satisfies the chain map condition:
\begin{equation}\nonumber
\partial \circ \iota_{\sigma_k} + \iota_{\sigma_k} \circ \partial =0.
\end{equation}
Then we have
\begin{proposition}
$\varphi_{\sigma}$ vanishes for all $\sigma \in N(\mathcal{I})_k$ with $k > 1 + \dim W.$
\end{proposition}

\begin{proof}
Suppose that $\varphi_{\sigma}(x) \neq 0,$ for some $x \in Crit(h_{s \sigma_k}).$ Then if $k > 1 + \dim W,$ the following inequality holds
\begin{equation}\nonumber
\dim W \geq | \varphi_{\sigma}(x) | = |x| + k-1 > |x| + \dim W \geq \dim W,
\end{equation}
which is a contradiction.
\end{proof}

We now focus on the case of $|y| - |x| = k$ when $\mathcal{M}(\mathcal{H}_{\sigma_k}; x, y)$ is a 1-dimensional manifold. We consider the following space:
\begin{equation}\nonumber
\begin{split}
\mathcal{L}(\mathcal{H}_{\sigma_k};x, y)&:= \bigcup\limits_{\substack{y' \in Crit(\overline{h}^{t \sigma_k}_{s \sigma_k}),\\ |y'| = |x| - 1 }} \mathcal{M}(\overline{h}^{t \sigma_k}_{s \sigma_k}; x, y') \times \mathcal{M}(\mathcal{H}_{\sigma_k}; y', y)\\
&\cup \bigcup\limits_{\substack{y'' \in Crit(h_{t \sigma_k}),\\ |y''| = |y| +1}} \mathcal{M}(\mathcal{H}_{\sigma_k}; x, y'') \times \mathcal{M}(h_{t\sigma_k}; y'',y),
\end{split}
\end{equation}
which is a finite set by Corollary \ref{finsedi0}. For reasons of degree, breaking must take place exactly once. Hence in Theorem \ref{cptthm} for any sequence of elements in $\mathcal{M}(\mathcal{H}_{\sigma_k}; x,y),$ it is either of the type:
\begin{enumerate}[label = (\roman*)]
\item \begin{equation}\nonumber
(\vec{t}_n, u_n) \rightarrow \big(v, (\vec{t},u) \big) \in \mathcal{M}(\overline{h}^{t\sigma_k}_{s \sigma_k}; x, y') \times \mathcal{M}(\mathcal{H}_{\sigma_k}; y', y),
\end{equation}
or of the type:
\item \begin{equation}\nonumber
(\vec{t}_n, u_n) \rightarrow \big( (\vec{t'},u'), v \big) \in \mathcal{M}(\mathcal{H}_{\sigma_k}; x, y'') \times \mathcal{M}({h}_{t \sigma_k}; y'', y)
\end{equation}
\end{enumerate}
for some critical points $y'$ and $y''.$ By identifying these limits with the elements in $\mathcal{L}(\mathcal{H}_{\sigma_k};x, y),$ we can compactify $\mathcal{M}(\mathcal{H}_{\sigma_k}; x, y),$ which is just a standard process in Morse theory. 

\begin{theorem}
If $|y|-|x| + k = 0,$ then the union of moduli spaces
\begin{equation}\nonumber
\begin{split}
\overline{\mathcal{M}}(\mathcal{H}_{\sigma_k}; x, y) : = \mathcal{M}(\mathcal{H}_{\sigma_k};x, y) \cup '  \mathcal{L}(\mathcal{H}_{\sigma_k};x, y).
\end{split}
\end{equation}
is a compact 1-dimensional manifold whose boundary is given by
\begin{equation}\nonumber
\partial\overline{\mathcal{M}}(\mathcal{H}_{\sigma_k}; x, y) : = \bigcup\limits_{i=1}^{k-1} \mathcal{M}(\mathcal{H}_{\sigma_k}|_{t_i=0};x, y) \cup \bigcup\limits_{i=1}^{k-1}\mathcal{M}(\mathcal{H}_{\sigma_k}|_{t_i =1};x,y) \cup \mathcal{L}(\mathcal{H}_{\sigma_k};x, y),
\end{equation}
which  is a finite set. (Here we used the notation $\cup'$ to describe the compactification process mentioned above.)
\end{theorem}

\begin{corollary}\label{correl}
The higher continuation maps $\{ \varphi_{\sigma_k} \}$ satisfy:
\begin{equation}\label{coreqn}
\partial \circ \varphi_{\sigma_k} + \varphi_{\sigma_k} \circ \partial + \sum_{i=1}^{k-1} \varphi_{\partial_i \sigma_k} + \sum_{i=1}^{k-1} \varphi_{(\sigma_k)_i^2} \circ \varphi_{(\sigma_k)_{k-i}^1} =0.
\end{equation}
\end{corollary}

\begin{proof}
By counting the order of the moduli space, we have
\begin{enumerate}[label= \textbullet]
\item $\partial \circ \overline{\varphi}_{\sigma_k} + \overline{\varphi}_{\sigma_k} \circ \partial$ from $\#_2 \mathcal{L}(\mathcal{H}_{\sigma_k};x, y),$
\item $\overline{\varphi}_{(\sigma_k)_i^2} \circ \overline{\varphi}_{(\sigma_k)_{k-i}^1},$ from $\#_2 \mathcal{M}(\mathcal{H}_{\sigma_k}|_{t_i=0};x,y)$ and Lemma \ref{isotil},
\item $\overline{\varphi}_{\partial_i \sigma_k}$ from $\#_2 \mathcal{M}(\mathcal{H}_{\sigma_k}|_{t_i=1};x,y)$ and Lemma \ref{isotil},
\end{enumerate}
for all $x \in Crit(\overline{h}^{t \sigma_k}_{s \sigma_k}),$ $y \in Crit(h_{t \sigma_k})$ with $|y|-|x| + k = 0,$ and $i = 1, \cdots k-1.$ By construction of higher data and from the fact that the number of boundary components of $\partial \overline{\mathcal{M}} (\mathcal{H}_{\sigma_k}; x, y)$ is 0 modulo 2, we have:
\begin{equation}\nonumber 
\partial \circ \overline{\varphi}_{\sigma_k} + \overline{\varphi}_{\sigma_k} \circ \partial + \sum_{i=1}^{k-1} \overline{\varphi}_{\partial_i \sigma_k} + \sum_{i=1}^{k-1} \overline{\varphi}_{(\sigma_k)_i^2} \circ \overline{\varphi}_{(\sigma_k)_{k-i}^1} = 0.
\end{equation}
By applying $\iota_{\sigma_k}$ from the right and using the relations
\begin{enumerate}[label= (\roman*)]
\item $ \partial \circ \iota_{\sigma_k} + \iota_{\sigma_k} \circ \partial = 0,$
\item $  \iota_{\sigma_k} = \iota_{\partial_i \sigma_k},$
\item $ (\overline{\varphi}_{(\sigma_k)_i^2}  \circ \overline{\varphi}_{(\sigma_k)_{k-i}^1}) \circ \iota_{\sigma_k} = {\varphi}_{(\sigma_k)_i^2} \circ {\varphi}_{(\sigma_k)_{k-i}^1},$
\end{enumerate}
for every $i = 1, \cdots, k-1,$ we obtain the relation (\ref{coreqn}). Here (i), (ii) follow from Lemma \ref{lem2}, and (iii) from the proof of Corollary \ref{corlater}.
\end{proof}

\section{Homotopy colimit as a chain model for Morse homology}

In this section, we introduce a chain model for the Morse homology of a pair $(W, \mathfrak{H}^1).$

\subsection{Morse homotopy coherent diagrams}

\begin{definition} 
We fix Morse homotopy data $\mathfrak{H}$ and the associated higher continuation maps $\{\varphi_{\sigma_k}\}.$ By a {\it{Morse homotopy coherent diagram,}} we mean a homotopy coherent diagram

\begin{equation}\nonumber
\mathfrak{M} : N(\mathcal{I}) \rightarrow N_{dg}\big(Ch(\mathbb{Z}_2)\big)
\end{equation}
that is given by the following assignment
\begin{enumerate}[label = \textbullet ]
\item To each object $a$ of $\mathcal{I},$ we associate a chain complex 

\begin{equation}\nonumber
\mathfrak{M}(a) := \big(MC_*(h_a), \partial \big).
\end{equation}

\item To each sequence of composable morphisms $(f_1, \cdots f_{k}),$ we associate a degree $k-1$ linear map 
\begin{equation}\nonumber
\mathfrak{M}(f_1, \cdots f_{k}) : MC_* (h_{sf_1}) \rightarrow MC_{*}(h_{tf_{k}})
\end{equation} 
given by $\mathfrak{M}(f_1, \cdots, f_{k}) := \varphi_{(f_1, \cdots, f_{k})}.$ By Corollary \ref{correl}, we know the collection $\bigl\{\mathfrak{M}(f_1, \cdots, f_k)\bigr\}_{(f_1, \cdots, f_k)}$ is subject to the relation (\ref{coreqn}).
\end{enumerate}
\end{definition}

\subsection{Morse chain complexes}

We propose a chain model for the Morse homology. 

\begin{definition}\label{mccdef}
For given Morse homotopy data $\mathfrak{H}$ and the associated Morse homotopy coherent diagram $\mathfrak{M},$ we define the {\textit{Morse chain complex}} (over $\mathbb{Z}_2$) of a pair $(W, \mathfrak{H})$ by
\begin{equation}\nonumber
MC_*({W},\mathfrak{H}) := \bigoplus\limits_{k \geq 0} \bigoplus\limits_{\substack{\sigma_k \in N(\mathcal{I})_k\\ \text{nondeg.} }} \mathbb{Z}_2 \langle \sigma_k \rangle \otimes MC_*(h_{s\sigma_k})
\end{equation}
with the differential given by
\begin{equation}\nonumber
\begin{split}
\partial : (f_1, \cdots, f_k; x) \mapsto \sum\limits_{i=1}^k \big(\partial_i&(f_1, \cdots, f_k); x\big) + \sum\limits_{i=1}^k\big(f_{i+1}, \cdots, f_k ; \varphi_{f_1, \cdots, f_i}(x)\big)\\ &+ (f_1, \cdots, f_k; \partial x),
\end{split}
\end{equation}
and the grading by
\begin{equation}\nonumber
|(\sigma_k;x)| := k+ |x|,
\end{equation}
for $\sigma_k = (f_1, \cdots, f_k) \in N(\mathcal{I})_k$ and $x \in Crit (h_{sf_1}).$
\end{definition}

\begin{theorem}
$\big( MC_*({W}, \mathfrak{H}), \partial\big)$ is a chain complex, the homology of which is isomorphic to $MH_*({W}, \mathfrak{H}^1).$
\end{theorem}

\begin{proof}
This is an application of Corollary \ref{lastcor}.
\end{proof}

\subsubsection*{A filtration} For $N \geq 1$ we denote
\begin{equation}\nonumber
\begin{split}
MC_*({W}, \mathfrak{H})^{N} := \bigl\{(\sigma_k; x) \in MC_*({W}, \mathfrak{H}) & \mid \sigma_k \in N(\mathcal{I})_k \\ & \text{ with } k \leq N \text{ and } x \in Crit(h_{t \sigma_k}) \bigr\}.
\end{split}
\end{equation}
$MC_*(W, \mathfrak{H})^N$ defines a chain complex with the induced differential because the differential does not increase the number of morphisms in the nerve part. In other words, they form a filtration on the chain complex $MC_*({W},\mathfrak{H})$:
\begin{equation}\nonumber 
\bigoplus\limits_{a \in \mathbb{Z}_{\geq 0}} MC_*(h_a) \subseteq MC_*({W}, \mathfrak{H})^{1} \subseteq MC_*({W}, \mathfrak{H})^{2} \subseteq \cdots \subseteq MC_*({W}, \mathfrak{H}).
\end{equation}
Moreover, we have the following maps that are induced from the inclusions:
\begin{equation}\nonumber 
H_*\big(MC_*({W}, \mathfrak{H})^{1}\big) \xrightarrow{} H_*\big(MC_*({W}, \mathfrak{H})^{2}\big)  \xrightarrow{} \cdots  \xrightarrow{} H_*\big(MC_*({W}, \mathfrak{H})\big).
\end{equation}

\subsection{Comparisons between two chain complexes}

Fixing an exhaustion $\mathcal{E}$ on $W,$ we compare two chain complexes $MC_*(W, \mathfrak{H})$ and $MC_*(W, \mathfrak{H}')$ for two higher data $\mathfrak{H}$ and $\mathfrak{H}'$ with $\mathcal{E} \subset \mathfrak{H}, \mathfrak{H}'.$ As we introduced in Section 2, there are two models for this purpose: the models with $\widehat{\mathcal{I}}$ and $\widehat{\mathcal{I}}.$

\subsubsection*{The model with $\widehat{\mathcal{I}}$}
\begin{condition}\label{extfuncchn}
For two given homotopy coherent diagrams 
\begin{equation}\nonumber
\mathscr{F}_{\mathfrak{H}}, \mathscr{F}_{\mathfrak{H}'} : N(\mathcal{I}) \rightarrow N_{dg}\big(Ch(\mathbb{Z}_2)\big),
\end{equation}
there is another
\begin{equation}\nonumber
\widehat{\mathscr{F}} : N(\widehat{\mathcal{I}}) \rightarrow N_{dg}\big(Ch(\mathbb{Z}_2)\big)
\end{equation} 
which extends $\mathscr{F}_{\mathfrak{H}}$ and $\mathscr{F}_{\mathfrak{H}'},$ i.e., $\widehat{\mathscr{F}}|_{N(\mathcal{I}) \times \{0\}}= \mathscr{F}_{\mathfrak{H}}$ and $\widehat{\mathscr{F}}|_{N(\mathcal{I}) \times \{1\}}= \mathscr{F}_{\mathfrak{H}'}$ 
\end{condition}

\begin{proposition} 
Suppose that Condition \ref{extfuncchn} holds for $\mathfrak{H}$ and $\mathfrak{H}'$. Then there exist quasi-isomorphisms of chain complexes
\begin{equation}\label{qrf}
\begin{tikzcd}
{} & colim {\widehat{\mathscr{F}}}(*)  & {}\\
MC_*(W, \mathfrak{H}) \arrow{ru}{\simeq}[swap]{\text{q-isom.}} & {} & MC_*(W, \mathfrak{H}') \arrow{lu}{\text{q-isom.}}[swap]{\simeq}
\end{tikzcd}
\end{equation}
\end{proposition}

Recall that in Example \ref{egstrdia}, for strict diagrams with the fixed chain complexes at the vertices, we checked that such an extension exists. It can apply to our situation, so we have the following corollaries.

\begin{corollary}
Let $\mathfrak{H}_s, \mathfrak{H}_s'$ be strict higher data with the same exhaustion $\mathcal{E}.$ Then the Morse chain complexes $MC_*(W, \mathfrak{H})$ and $MC_*(W, \mathfrak{H}')$ are related as in (\ref{qrf}) for some $\widehat{\mathscr{F}}.$
\end{corollary}

\subsubsection*{The model with $\overline{\mathcal{I}}$}
\begin{condition}\label{condlast}
There exist higher data $\overline{\mathfrak{H}}, \ {\mathfrak{H}}^{\overline{\mathcal{I}}}$ and homotopy coherent diagrams
\begin{equation}\nonumber
\begin{split}
\mathscr{F}_{\overline{\mathfrak{H}}} &: N(\mathcal{I}) \rightarrow N_{dg}\big(Ch(\mathbb{Z}_{2})\big),\\
\mathscr{F}_{{\mathfrak{H}}^{\overline{\mathcal{I}}}} &: N(\overline{\mathcal{I}}) \rightarrow N_{dg}\big(Ch(\mathbb{Z}_{2})\big)
\end{split}
\end{equation}
that extend $\mathscr{F}_{\mathfrak{H}}$ and $\mathscr{F}_{\mathfrak{H}'},$ i.e., we have
\begin{equation}\nonumber
\begin{split}
\mathscr{F}_{\overline{\mathfrak{H}}} \circ \Phi_{0*}' = \mathscr{F}_{\mathfrak{H}},& \ \mathscr{F}_{\overline{\mathfrak{H}}} \circ \Phi_{1*}' = \mathscr{F}_{\mathfrak{H}'},\\
\mathscr{F}_{{\mathfrak{H}}^{\overline{\mathcal{I}}}} \circ \Phi''_{0*} = \mathscr{F}_{\mathfrak{H}},& \ \mathscr{F}_{{\mathfrak{H}}^{\overline{\mathcal{I}}}} \circ \Phi''_{1*} = \mathscr{F}_{\mathfrak{H}'},\\
\end{split}
\end{equation}
that satisfy
\begin{equation}\nonumber
 \mathscr{F}_{{\mathfrak{H}}^{\overline{\mathcal{I}}}} \circ \Psi'_{*} = \mathscr{F}_{\overline{\mathfrak{H}}}.
\end{equation}
(See Section 2 for the notations.)
\end{condition}

According to our discussion in Section 2 and the fact that $\Phi_{0,*}''$ and $\Phi_{1,*}''$ are categorical equivalences, we can conclude:

\begin{proposition} 
Suppose that Condition \ref{condlast} holds for $\mathfrak{H}$ and $\mathfrak{H}'$. Then there exist quasi-isomorphisms of chain complexes
\begin{equation}\nonumber
\begin{tikzcd}
{} & colim \mathscr{F}_{{\mathfrak{H}}^{\overline{\mathcal{I}}}}(*)  & {}\\
MC_*(W, \mathfrak{H}) \arrow{ru}{\simeq}[swap]{\text{q-isom.}} & {} & MC_*(W, \mathfrak{H}') \arrow{lu}{\text{q-isom.}}[swap]{\simeq}
\end{tikzcd}
\end{equation}
\end{proposition}

\begin{remark}
(i) Contrary to Condition \ref{extfuncchn}, Condition \ref{condlast} is \textit{geometric} in that it requires a construction of geometric homotopy data not just a diagram that algebraically extends the given ones. (ii) We expect explicit constructions of such higher data $\overline{\mathfrak{H}}$ and ${\mathfrak{H}}^{\overline{\mathcal{I}}}$ to be possible in our future work, by a similar type of considerations as in Section 9.
\end{remark}

\newpage

\appendix

Most of the contents in the appendices are excerpts from [Kim] for the reader's convenience.

\section{The diagram $\mathscr{F}$}

In this section, we consider a diagram of the following form
$$\mathscr{F} : N(\mathcal{\mathcal{I}}) \rightarrow N_{dg}\big(Ch(\mathbb{Z}_2)\big)$$ and an extension $\overline{\mathscr{F}}$ to $N(\mathcal{I})^{\triangleright}$ such that its value at the cone point is a preferred chain complex $\widetilde{C}$ of (\ref{prfchncpx}). We show that $\overline{\mathscr{F}}$ is an initial object in the overcategory $N_{dg}\big(Ch(\mathbb{Z}_2)\big)_{\mathscr{F}/},$ hence a colimit of the diagram $\mathscr{F}.$

\subsection{$\infty$-categories}

In this subsection, we review the basics of $\infty$-category theory.

\subsubsection*{Simplicial sets and $\infty$-categories} Denote by $[n] :=  \{0,1, \cdots, n\}$ the ordered set of cardinality $n$ for $n \geq 0.$

Let $\Delta$ be a category which consists of the following data:

\begin{enumerate}[label= \textbullet]
\item Ob($\Delta$) $= \bigl\{ [n] \mid n \geq 0\bigr\},$
\item Mor($\Delta$) = $\{$order-preserving maps between objects$\}$, (i.e., a map $\alpha : [m] \rightarrow [n]$ is a morphism if it satisfies: $i \leq j \Rightarrow \alpha(i) \leq \alpha(j)$).
\end{enumerate}

\ \

Notice that the order-preserving injective (and surjective) maps are defined only if $m \leq n$ (and $m \geq n,$ respectively). We define the following family of injective maps ({\it{face}} maps),
\begin{equation}\nonumber
d_i : [n] \rightarrow [n+1], \ (i=0, \cdots, n+1),
\end{equation}
by
\begin{equation}\nonumber
d_i(j) = \begin{cases}
j &\text{ if } j \leq i\\
j + 1 &\text{ if } j \geq i+1,
\end{cases}
\end{equation}\nonumber
and surjective maps ({\it{degeneracy}} maps),
\begin{equation}
s_i : [n] \rightarrow [n-1], \ (i=0, \cdots, n),
\end{equation}
by
\begin{equation}\nonumber
s_i(j) = \begin{cases}
j &\text{ if } j \leq i\\
j - 1 &\text{ if } j \geq i+1.
\end{cases}
\end{equation}
These maps characterize the morphisms of the category $\Delta.$

\begin{definition}
A \textit{simplicial set} $X$ is defined to be a contravariant functor
\begin{equation}\nonumber
X : \Delta^{op} \rightarrow Set.
\end{equation}
\end{definition}

Denote $X_n := X\big([n]\big),$ and consider the images of the maps $d_i$ and $s_i$ denoted by $\partial_i := X(d_i), \text{ and } \sigma_i := X(s_i)$ respectively for each $n$ and $i = 0, \cdots n$:
\begin{equation}\nonumber
\begin{split}
\partial_i &: X_{n} \rightarrow X_{n-1},\\
\sigma_i &: X_{n} \rightarrow X_{n+1}.
\end{split}
\end{equation}
We call $\partial_i$ and $\sigma_i$ the $face$ and $degeneracy$ maps, respectively.
\begin{definition}
A \textit{morphism} of simplicial sets $X, Y \in Fun(\Delta^{op}, Set)$ is a natural transformation between the functors. We denote by $sSet$ the \textit{category of simplicial sets.}
\end{definition}

\begin{definition}
For $m$ and $k$ with $1 \leq k \leq m,$ we define a \textit{horn} $\Lambda^{m}_k$ to be the simplicial set which is obtained from $\Delta^m$ by removing its $k$-th face and its interior.
\end{definition}

Let $X$ be a simplicial set and $\Lambda^m_k \rightarrow X$ a morphism of simplicial sets. If this extends to a morphisms $\Delta^m \rightarrow X$ for any $0 \leq k \leq m$ and for all $m,$ we say $X$ satisfies the {\textit{Kan condition}}. If this condition holds for $0 < k < m,$ we say $X$ satisfies the {\textit{inner Kan condition.}}

\begin{definition}
An {\textit{$\infty$-category}} is defined to be a simplicial set that satisfies the inner Kan condition. A {\textit{morphism between $\infty$-categories}}, or an {\textit{$\infty$-functor}} is defined to be a morphism of the underlying simplicial sets, i.e., a natural transformation between functors.
\end{definition}

\begin{definition}
Let $K$ be a simplicial set and $\mathcal{D}$ an $\infty$-category. A \textit{homotopy coherent diagram} $K \rightarrow \mathcal{D}$ is defined to be the morphism of simplicial sets.
\end{definition}

\begin{definition}
The {\it{homotopy category}} $h\mathcal{C}$ of an $\infty$-category $\mathcal{C}$ is an ordinary category that is defined as follows. The objects are given by the vertices $\mathcal{C}_0$ and the morphisms are given by 1-simplices $\mathcal{C}_1$ up to the relations generated by: (i) $id_x \simeq \sigma_x(x)$ for all $x \in \mathcal{C}_0$ (ii) $\partial_1(y) \simeq \partial_0(y) \cdot \partial_1(y)$ for all $y \in \mathcal{C}_2$. 
\end{definition}

\subsection{Nerves and dg-nerves}

In this subsection, we study the nerve of a category as examples of $\infty$-categories.

\subsubsection*{The nerve of indexing category $\mathcal{I}$} Let $\mathcal{I}$ be the poset category. For each $k \geq 0,$ we consider the following sets:
\begin{equation}\nonumber
\begin{split}
& N(\mathcal{I})_{0} := \text{Ob}(\mathcal{I}).\\
& N(\mathcal{I})_{k} := \{ \text{sequences of $k$ composable morphisms of }\mathcal{I}\}, \ k \geq 1,\\
\end{split}
\end{equation}

For example, any $\sigma_k \in N(\mathcal{I})_k$ with $k \geq 1$ is of the form:
\begin{equation}\nonumber
\sigma_k = (f_{a_1,a_2}, f_{a_2,a_3} \cdots, f_{a_k, a_{k+1}}),
\end{equation}
where $f_{ab}$ denotes the unique morphism from $a$ to $b.$
(See subsection 2.1 for our convention for the morphisms of the poset category $\mathcal{I}.$)
In this paper, let us call $k$ the \textit{length} of $\sigma_k$ and denote $|\sigma_k| := k.$ 
We denote 
\begin{equation}\nonumber
N(\mathcal{I}) := \coprod\limits_{k \geq 0} N(\mathcal{I})_k
\end{equation}
and call it the \textit{nerve} of $\mathcal{I}.$

\begin{notation}\label{not12p}
For $\sigma_k \in N(\mathcal{I})_k$ with $k \geq 1$ and $1 \leq i \leq {k-1},$  we denote
\begin{equation}\label{not12}
\begin{split}
&(\sigma_k)^1_i := (f_{a_1,a_2}, \cdots f_{a_i,a_{i+1}}), \\
&(\sigma_k)^2_{k-i} := (f_{a_{i+1},a_{i+2}}, \cdots f_{a_{k-1},a_k}).
\end{split}
\end{equation}
\end{notation}

In fact, $N(\mathcal{I})$ is a simplicial set, and the face and degeneracy maps are defined as follows.
\begin{equation}\nonumber
\begin{split}
\partial_i &: N(\mathcal{I})_k \rightarrow N(\mathcal{I})_{k-1}, (i = 0, \cdots, k)\\
\sigma_i &: N(\mathcal{I})_k \rightarrow N(\mathcal{I})_{k+1}, (i = 0, \cdots, k-1)\\
\end{split}
\end{equation}

\begin{equation}\nonumber
\partial_i (f_{a_1,a_2}, \cdots, f_{a_k, a_{k+1}}) =
\begin{cases}
(f_{a_2,a_3}, \cdots, f_{a_k, a_{k+1}}) & \text{ if } i = 0,\\
(f_{a_1,a_2}, \cdots, f_{a_{i-1}, a_{i}}, f_{a_{i}, a_{i+2}},f_{a_{i+2},a_{i+3}},& \cdots , f_{a_k, a_{k+1}}) \\ & \text{ if } 1 \leq i \leq k-1,\\
(f_{a_1,a_2}, \cdots, f_{a_{k-1}, a_{k}}) & \text{ if } i =k,
\end{cases}
\end{equation}

\begin{equation}\nonumber
\sigma_i (f_{a_1,a_2}, \cdots, f_{a_k, a_{k+1}}) =
\begin{cases}
(id_{a_1}, f_{a_1,a_2}, \cdots, f_{a_k, a_{k+1}}) & \text{ if } i = 0,\\
(f_{a_1,a_2}, \cdots, f_{a_{i}, a_{i+1}}, id_{a_{i+1}}, f_{a_{i+1},a_{i+2}},& \cdots , f_{a_k, a_{k+1}}) \\ & \text{ if } 1 \leq i \leq k-1,\\
(f_{a_1,a_2}, \cdots, f_{a_{k}, a_{k+1}}, id_{a_k}) & \text{ if } i =k.
\end{cases}
\end{equation}

\begin{lemma}([Lur1] Example 1.1.2.6)
$N(\mathcal{\mathcal{I}})$ is an $\infty$-category.
\end{lemma}

\subsubsection*{The dg-Nerve of a dg category} We introduce the dg-nerve of a dg-category. Let $\mathcal{A}$ be a differential graded category or dg category over some base ring, say $R$. 
\begin{definition}
The \textit{dg-nerve} of $\mathcal{A}$ is defined by the following data:
for an integer $k \geq 1,$ consider a subset $I \subset [k],$ say $I =\{ j_1 < \cdots < j_l \}$ (with $|I|>1)$ and a pair
\begin{equation}\nonumber
N_{dg}(\mathcal{A})_k := \big(\{X_i\}_{i \in [k]}, \{\psi_{I}\}_{I \subset [k]}\big),
\end{equation}
where each $X_i$ is an object of $\mathcal{C}$ and $\psi_{I} : X_{j_1} \rightarrow X_{j_l}$ a morphism of degree $l-1,$ satisfying
\begin{equation}\label{diadgn}
\partial \circ \psi_I + \psi_I \circ \partial + \sum\limits_i \psi_{I \setminus \{i\}} + \sum\limits_i \psi_{\{j_{i+1} , \cdots,  j_l \}} \circ \psi_{\{j_1, \cdots, j_i \} } = 0.
\end{equation}
Denote by $N_{dg}(\mathcal{A})_k$ the set of all such data for a given $k,$ and the disjoint union:
\begin{equation}\nonumber
N_{dg}(\mathcal{A}) : = \coprod_{k \geq 0} N_{dg}(\mathcal{A})_k
\end{equation}
\end{definition}

Let $\alpha$ be a morphism of the {\textit{simplicial}} category $\Delta,$ that is, an order-preserving map $\alpha : [k] \rightarrow [m].$ Then there is an induced map:
\begin{equation}\nonumber
\begin{split}
\Phi^{\alpha} : N_{dg}(\mathcal{A})_m &\rightarrow N_{dg}(\mathcal{A})_k,\\
\big(\{X_i\}_{i \in [m]}, \{\psi_{I}\}_{I \subset [m]}\big) &\mapsto \big(\{X_{\alpha(j)}\}_{j \in [k]}, \{\psi^{\alpha}_{J}\}_{J \subset [k]}\big),
\end{split}
\end{equation}
where $\psi^{\alpha}_J : X_{\text{the smallest element of } \alpha(J) } \rightarrow X_{\text{the largest element of } \alpha(J) }$ for $J \subset [k]$ is a linear map given by
\begin{equation}\nonumber
\psi^{\alpha}_J :=
\begin{cases}
\psi_{\alpha(J)} & \text{ if } \alpha|_{J} \text{ is injective,}\\
id_{\alpha(j)} & \text{ if } J = \{j, j'\} \text{ with } \alpha(j) = \alpha(j'),\\
0 & \text{ otherwise.}
\end{cases}
\end{equation}

From a straightforward consideration, we have
\begin{lemma}
We have $\Phi^{\alpha'} \circ \Phi^{\alpha} = \Phi^{\alpha \circ \alpha'}.$
\end{lemma}

Consider $\Phi : \Delta^{op} \rightarrow Set$ a contravariant functor with $\Phi \big({[n]} \big) := N_{dg}(\mathcal{A})_n$ and $\Phi(\alpha) = \Phi^{\alpha}.$ Then $\Phi$ defines a simplicial set with the face and degeneracy maps given by $\Phi^{d_i}$ and $\Phi^{s_i}$ maps, respectively. In fact we have:

\begin{theorem}([Lur2] Proposition 1.3.1.10)
$N_{dg}(\mathcal{A})$ is an $\infty$-category.
\end{theorem}

\subsection{Colimit of a diagram}

In this subsection, we introduce the notion of a colimit of a diagram.

\subsubsection*{Join of $\infty$-categories} Let $K_1$ and $K_2$ be simplicial sets. We define the \textit{join} of $K_1$ and $K_2$ to be the simplicial set $K_1 \star K_2,$ where
\begin{equation}
(K_1 \star K_2)_m = (K_1)_m \cup (K_2)_m \cup \bigcup\limits_{i+j = m-1} (K_1)_i \times (K_2)_j.
\end{equation} 

The face and degeneracy maps are induced by those of $K_1$ and $K_2.$

\begin{theorem}([Lur1] Proposition 1.2.8.3)(Joyal)
If $K_1$ and $K_2$ are $\infty$-categories, then so is $K_1 \star K_2.$
\end{theorem}

\begin{definition}
In particular, {\it{cones}} are simplicial sets and their $k$-simplices are given as follows:
\begin{equation}
\begin{split}
K^{\triangleright} &: = K \star \Delta^0 \text{ (the {\it{right cone}}), }\\
K^{\triangleleft} &: = \Delta^0 \star K  \text{ (the {\it{left cone}}), }
\end{split}
\end{equation}
and
\begin{equation}
(K^{\triangleright})_k =
\begin{cases}K_1 \cup \Delta^0 & \text{ if } k = 0,\\
K_k \cup (K_{k-1} \times \Delta^0) & \text{ if } k \geq 1,
\end{cases}
\end{equation} 

\begin{equation}
(K^{\triangleleft})_k =
\begin{cases}\Delta^0 \cup K_0 & \text{ if } k = 0,\\
(\Delta^0 \times K_k) \cup K_{k-1} & \text{ if } k \geq 1.
\end{cases}
\end{equation} 
The face and degeneracy maps are naturally induced from those of $K.$ The 0-simplices from $\Delta^0$ in $K^{\triangleright}$ are called the {\it{cone points.}}

\end{definition}

\subsubsection*{Over/undercategories} Let $K$ be a simplicial set and $\mathcal{D}$ an $\infty$-category. 

\ \

\begin{definition}[Over/undercategories]
For a diagram  $p : K \rightarrow \mathcal{D},$ we consider the simplicial set $\mathcal{D}_{/p}$ with
\begin{equation}
(\mathcal{D}_{/p})_k := \text{Hom}_{p}(\Delta^k \star K, \mathcal{D}),
\end{equation}
where $\text{Hom}_p(\cdots)$ means that its elements are extensions $\overline{p} : \Delta^k \star K \rightarrow \mathcal{D}$ such that $\overline{p}|_{K} = p.$ When $K$ is an $\infty$-category, we call $\mathcal{D}_{/p}$ the \textit{overcategory} of $\mathcal{D}.$

Similarly, we define the \textit{undercategory} $\mathcal{D}_{p/}$ of $\mathcal{D}$ by 
\begin{equation}
(\mathcal{D}_{p/})_k := \text{Hom}_{p}(K \star \Delta^k, \mathcal{D}).
\end{equation}
\end{definition}

\ \

We do not use the following proposition later, but we provide it to mention the well-behavedness of over and under categories. 

\begin{proposition}\label{underoverinf}([Lur1] Proposition 1.2.9.3, Corollary 2.1.2.2) The simplicial sets $\mathcal{D}_{/p}$ and $\mathcal{D}_{p/}$ are $\infty$-categories. Moreover, if $\mathcal{D}_1 \rightarrow \mathcal{D}_2$ is a categorical equivalence (see Definition \ref{catequiv}) of $\infty$-categories, then so are the induced maps on over/undercategories.
\end{proposition}

When we consider $p : \Delta^0 \rightarrow \mathcal{D},$ that is, when $k=0,$ we sometimes denote $\mathcal{D}_{/X}$ (and  $\mathcal{D}_{X/}$ ) instead of $\mathcal{D}_{/p}$ (and $\mathcal{D}_{p/},$ respectively).

\begin{definition}
Let $\mathcal{D}$ be an $\infty$-category. For $X, Y \in \mathcal{D}_0,$  we define the \textit{right morphism space} $\text{Hom}_{\mathcal{D}}^R(X,Y)$ to be the simplicial set whose $k$-simplices are given by maps $\Delta^{k+1} \rightarrow \mathcal{D}$ such that
\begin{equation}
\begin{split}
\Delta^{k+1}&|_{[0, \cdots, k]} \text{ maps to the constant simplex }  X\\
\Delta^{k+1}&|_{k+1}  \text{ maps to } Y. 
\end{split}
\end{equation}
Here, the constant simplex $Y$ is the one that consists of one element $Y$ in each degree simplex. Similarly, the {\it{left}} morphism space $\text{Hom}_{\mathcal{D}}^L(X,Y)$ is defined to be the simplicial set whose $k$-simplices are given by maps $\Delta^{k+1} \rightarrow \mathcal{D}$ such that
\begin{equation}
\begin{split}
\Delta^{k+1}&|_{0} \text{ maps to }  X\\
\Delta^{k+1}&|_{[1, \cdots, k]} \text{ maps to the constant simplex } Y. 
\end{split}
\end{equation}

\end{definition}

Hom$_{\mathcal{D}}^R(X,Y)$ (and similarly Hom$_{\mathcal{D}}^L(X,Y)$) can be interpreted as a topological space, and it is a Kan complex. (See [Lur1] Proposition 1.2.2.3 for the details.) We say that it is \textit{contractible} if any map $\partial \Delta^k \rightarrow Hom^R_{\mathcal{D}}(p,q)$ can be filled, thus yielding an extension $\Delta^k \rightarrow Hom^R_{\mathcal{D}}(p,q).$ 

\ \

It is a fact that when $\mathcal{D}$ is an $\infty$-category, Hom$_{\mathcal{D}}^R(X,Y)$ and Hom$_{\mathcal{D}}^L(X,Y)$ are homotopy equivalent. (c.f. [Lur1] Remark 1.2.2.5.) Hence the contractibility of one implies the other, and vice versa.

\begin{definition}\label{finin}
Let $\mathcal{D}$ be an $\infty$-category. A vertex $X$ of $\mathcal{D}$ is said to be \textit{final (initial)} if the Kan complex ${\color{black}{\text{Hom}_{\mathcal{D}}^R(Y,X)}}$ (${\color{black}{\text{Hom}_{\mathcal{D}}^R(X,Y)}},$ respectively) is contractible for any object $Y$ of $\mathcal{D}.$  
\end{definition}

\begin{remark}
In [Lur1], this notion is called {\it{strong finality.}} (More precisely, it is presented as an equivalent condition in Proposition 1.2.12.4.) There, an object is said to be {\it{final}}, if it is final (in the usual sense) in the homotopy category. Strong finality implies finality for simplicial sets in general, and they coincide for $\infty$-categories. (See Corollary 1.2.12.5.)
\end{remark}

\subsubsection*{The colimit of a diagram} One can also define an $\infty$-categorical analogue of limits and colimits.

\begin{definition}
Let $p : K \rightarrow \mathcal{D}$ be a diagram from a simplicial set $K$ to an $\infty$-category $\mathcal{D}.$ We define a \textit{limit} of $p$ to be a \textit{final} object of $\mathcal{D}_{/p}$. Similarly, we define a \textit{colimit} of $p$ to be a \textit{initial} object of $\mathcal{D}_{p/}$. We denote them by $\lim\limits_{\longleftarrow}p$ and $\lim\limits_{\longrightarrow}p,$ respectively.
\end{definition}

By the definition of  $\mathcal{D}_{/p}$ ($\mathcal{D}_{p/}$), an initial (final) object is an element of $\text{Hom}_{p}(K^{\triangleleft}, \mathcal{D})$ ($\text{Hom}_{p}(K^{\triangleright}, \mathcal{D})$, respectively). In this case, by abuse of notation, we sometimes call the image of the cone point for the extension $\overline{p}(*)$ a limit (colimit).

\begin{remark}
For an ordinary category, the colimit (if it exists) is unique up to unique isomorphism. The corresponding notion of uniqueness in $\infty$-category theory is by the contractibility of right morphism spaces.
\end{remark}

\subsubsection*{The diagram $\mathscr{F}$ and the $\infty$-categorical setting} A $k$-simplex $\sigma_k = (f_{a_1a_2}, \cdots,$ $f_{a_ka_{k+1}})$ $\in N(\mathcal{I})_k$ is determined by a sequence of integers $a_1, a_2, \cdots, $ $ a_{k+1}.$ For the moment, we assume only nondegenerate simplices, that is, $a_1 < a_2 < \cdots < a_{k+1},$ and discuss later the degenerate case. Observe that each subset $I: = \{ b_1, \cdots, b_l \} \subset \{a_1 < a_2 < \cdots < a_{k+1}\}$ determines an $l$-simplex $\sigma_I \in N(\mathcal{I})_l.$ Then an $\infty$-functor
\begin{equation}\nonumber
\mathscr{F}:N(\mathcal{I}) \rightarrow N_{dg}\big(Ch(\mathbb{Z}_2)\big)
\end{equation}
is given as follows.

$\mathscr{F}$ maps each $k$-simplex in the following way.
\begin{equation}\nonumber
\mathscr{F}\big([k]\big) : \{a_1, a_2, \cdots, a_{k+1}\} \mapsto \big(\{X_i\}_{i \in [k]}, \{F_{\sigma_I}\}_{I \subset [k]}\big), \text{ for } [k] \in \text{Ob}(\Delta),
\end{equation}
where its ingredients are given by
\begin{equation}\nonumber
\begin{split}
X_i &: = \mathscr{F}\big([k]\big)(a_{i+1}),\\
\psi_I &:=  \mathscr{F}\big([k]\big)({\sigma_I}).
\end{split}
\end{equation}

For degenerate simplices, $\sigma_I \in N(\mathcal{I})_k$ with $k \geq 2,$ we let:
\begin{equation}\nonumber
\psi_I :=  \begin{cases}
id &\text{ if } k = 1,\\
0 &\text{ if } k \geq 2,
\end{cases}
\end{equation}
It is a simple exercise to check that the identity and the zero maps satisfy (\ref{diadgn}).

By construction, $\mathscr{F}$ respects simplicial structure on both sides. Hence it is an $\infty$-functor.

\subsection{A model for a colimit}

Our focus in this subsection is on the undercategory $N_{dg}\big(Ch(\mathbb{Z}_2)\big)_{\mathscr{F}/}$ defined for the $\infty$-functor
\begin{equation}\nonumber
\mathscr{F} : N(\mathcal{I}) \rightarrow N_{dg}\big(Ch(\mathbb{Z}_2)\big),
\end{equation}
and we compute its colimit.

Recall that the undercategory $N_{dg}\big(Ch(\mathbb{Z}_2)\big)_{\mathscr{F}/}$ is a simplicial set whose set of $k$-simplices is given by $Hom_{\mathscr{F}}\big(N(\mathcal{I} \star \Delta^k), N_{dg}\big(Ch(\mathbb{Z}_2)\big)\big),$ that is, the set of simplicial set morphisms $\bigl\{\phi: N(\mathcal{I} \star \Delta^k) \rightarrow N_{dg}\big(Ch(\mathbb{Z}_2)\big)\bigr\},$ satisfying $\phi|_{N(\mathcal{I})} \equiv \mathscr{F}.$ The simplicial set structure is naturally induced from that of $\Delta^k.$ Namely, for each $i,$ we have 
\begin{equation}\nonumber
\begin{cases}
\partial_i : \phi \mapsto \phi \circ d_i,\\
\sigma_i : \phi \mapsto \phi \circ s_i.
\end{cases}
\end{equation}

The face and degeneracy maps $d_i : \Delta^{k} \rightarrow \Delta^{k+1}$ and $s_i : \Delta^k \rightarrow \Delta^{k-1}$ induce maps among the nerves $N(\mathcal{I} \star \Delta^k).$ We still denote them by $d_i$ and $s_i$;

\begin{equation}\nonumber
\begin{cases}
\partial_i : N(\mathcal{I} \star \Delta^k) \rightarrow N(\mathcal{I} \star \Delta^{k+1}),\\
\sigma_i : N(\mathcal{I} \star \Delta^k) \rightarrow N(\mathcal{I} \star \Delta^{k-1}).
\end{cases}
\end{equation} 

We write $\mathcal{D}$ for this undercategory, and in fact $\mathcal{D}$ is an $\infty$-category. (See [Lur2].) colim$\mathscr{F}$ is defined to be an initial object of $\mathcal{D}$. In other words, it is a vertex $p \in \big(N_{dg}\big(Ch(\mathbb{Z}_2)\big)_{\mathscr{F}/}\big)_0,$ i.e., $p \in Hom_{\mathscr{F}}\big(N(\mathcal{I}) \star \Delta^0, N_{dg}\big(Ch(\mathbb{Z}_2)\big)\big),$ such that: 
\begin{equation}\nonumber
Hom^R_{\mathcal{D}}(p,q)\text{ is a contractible Kan complex for any } q \in \mathcal{D}_0. 
\end{equation}

Consider the following chain complex $\widetilde{C}$
\begin{equation}\label{prfchncpx}
\widetilde{C} := \bigoplus_{k \geq 0} \bigoplus_{\substack{\sigma_k \in N(\mathcal{I}),\\ \text{nondeg.}}} \mathbb{Z}_2 \langle \sigma_k \rangle \otimes C_{s\sigma_k}, 
\end{equation}
with the differential $\partial$ given by
\begin{equation}\nonumber
\partial : (\sigma_k ; x) \mapsto \sum\limits_{i=0}^{k-1} (\partial_i \sigma_k ; x)+ \sum\limits_{i=0}^{k-1} \big((\sigma_k)_{i}^2 ; \mathscr{F}((\sigma_k)_{k-i}^1)(x)\big) + (\sigma_k ; \partial x)
\end{equation}

We put a grading on $\widetilde{C}$ by
\begin{equation}\nonumber
|(\sigma_k;x)| := k+ |x|,
\end{equation}
where $|x|$ on the right hand side is that of $C_{s\sigma_k}.$ One can check that the degree of $\partial$ is equal to $-1.$ 

We claim that $\widetilde{C}$ can be written as a colimit of $\mathscr{F}.$ More precisely, we will show that there is an extension $\overline{\mathscr{F}}$ such that $\overline{\mathscr{F}}(*) = \widetilde{C},$ satisfying some relation we will discuss later.

\subsubsection*{An extension of $\mathscr{F}$} We define the extension $\overline{\mathscr{F}}: N(\mathcal{I})^{\triangleright} \rightarrow N_{dg}\big(Ch(\mathbb{Z}_2)\big)$ of $\mathscr{F}$ by

\begin{enumerate}[label = \textbullet]
\item $\overline{\mathscr{F}}(*) = \widetilde{C}.$ \\

Let $\sigma_k \in N(\mathcal{I})^{\triangleright}_k$ be a $k$-simplex.
\item If $\sigma_k \in N(\mathcal{I})_k, \ (k \geq 0)$, then we let 
\begin{equation}\nonumber
\overline{\mathscr{F}}(\sigma_k) := \mathscr{F}(\sigma_k) : C_{s \sigma_k} \rightarrow C_{t \sigma_k},
\end{equation}
(In particular, if $\sigma_k$ is degenerate we have $\overline{\mathscr{F}}(\sigma_k) = 0.$)
\item If $\sigma_k \notin N(\mathcal{I})_k, \ (k > 0),$ i.e., $t \sigma_k =*$, then $\sigma_k$ is necessarily of the form $(f_1, \cdots, f_{k-1}, *),$ where $f_i \in Mor(\mathcal{I})$ such that $t f_i = s f_{i+1}.$ (Recall that $s$ and $t$ mean `source' and `target', respectively.) For nondegenerate $\sigma_k,$ we put
\begin{equation}\nonumber
\begin{split}
\overline{\mathscr{F}}(\sigma_k) : C_{s \sigma_k} &\rightarrow C_{t \sigma_k} = \widetilde{C},\\
x &\mapsto (f_1, \cdots, f_{k-1} ; x).
\end{split}
\end{equation}
For degenerate $\sigma_k,$ we simply put $\overline{\mathscr{F}}(\sigma_k) =0.$
\end{enumerate} 

\ \

Observe that for any $\sigma_k \in N(\mathcal{I})^{\triangleright}_k, k \geq 2,$we have
\begin{enumerate}[label = \textbullet]
\item For nondegenerate $\sigma_k,$ $\overline{\mathscr{F}}(\sigma_k)$ is a degree $k-1$ linear map,
\item For degenerate $\sigma_k,$ $\overline{\mathscr{F}}(\sigma_k)$ is the zero map,
\item $\overline{\mathscr{F}}(\sigma_k)$ satisfies
\begin{equation}\nonumber
\partial \circ \overline{\mathscr{F}}(\sigma_k) + \overline{\mathscr{F}}(\sigma_k) \circ \partial + \sum^k_{i=1} \overline{\mathscr{F}}(\partial_i \sigma_k) + \sum^k_{i=1} \overline{\mathscr{F}}\big((\sigma_k)^2_{k-i}) \circ \overline{\mathscr{F}}((\sigma_k)^1_i \big) = 0.
\end{equation}
\end{enumerate}

The first and second items are clear by construction. The third one can be checked to be equivalent to the expression for $\partial (f_1, \cdots, f_k ;x)$ in $\widetilde{C}$. 
\ \

We will check that the extension $\overline{\mathscr{F}}$ is an initial object of $N_{dg}\big(Ch(\mathbb{Z}_2)\big)_{\mathscr{F}/}$ with $\overline{\mathscr{F}}(*) = \widetilde{C}.$ Before that, we need to investigate the undercategory $\mathcal{D}$ more by explicitly describing its simplices of higher degree with $k \geq 1.$

\begin{figure}[h]
\centering
\includegraphics[width=0.4\textwidth]{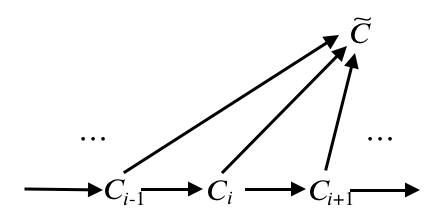}
\caption{The (homotopy coherent) diagram for the extension $\overline{\mathscr{F}}$}
\end{figure}

\subsubsection*{Edges of $\mathcal{D}$} We first explain the edges of $\mathcal{D},$ or the elements in $\mathcal{D}_1 = Hom_{\mathscr{F}}\big(N(\mathcal{I}) \star \Delta^1, N_{dg}\big(Ch(\mathbb{Z}_2)\big)\big).$ We denote $\Delta^1 = (*_0, *_1).$\begin{enumerate}[label = \textbullet]
\item To a vertex $a \in (N(\mathcal{I}) \star \Delta^1)_0,$ 
we assign a chain complex $C_a,$ thus implying $a \in N(\mathcal{I})_0,$ so that $C_a = \mathscr{F}(a).$
\item To a $k$-simplex $\sigma_k \in  (N(\mathcal{I}) \star \Delta^1)_k,$ we assign a degree $(k-1)$ map, ${\varphi}_{\sigma_k} : C_{s \sigma_k} \rightarrow C_{t \sigma_k},$ satisfying   
\begin{equation}\nonumber
\partial \circ \mathscr{F}({\sigma_k}) + \mathscr{F}({\sigma_k}) \circ \partial + \sum^k_{i=1} \mathscr{F}({\partial_i \sigma_k}) + \sum^k_{i=1} \mathscr{F}\big(({\sigma_k})^2_{k-i}) \circ \mathscr{F}(({\sigma_k})^1_i\big) = 0.
\end{equation}
\end{enumerate} 

\subsubsection*{Identity edges of $\mathcal{D}$} In particular, the degenerate edge $id_q$ for $q \in \mathcal{D}_0,$ that is, the image of the degeneracy map $\sigma : \mathcal{D}_0 \rightarrow \mathcal{D}_1$ of $q$ (induced from $s_0 : \Delta^1 \rightarrow \Delta^0$) is given by
\begin{enumerate}[label = \textbullet]
\item To $a \in (N(\mathcal{I}) \star \Delta^1)_0,$ we assign the chain complex $C_a := q(a).$ In particular, we have $q(*_0) = q(*_1).$
\item To $\sigma_k \in (N(\mathcal{I}) \star \Delta^1)_k$ with $k \geq 1,$ we assign a degree $(k-1)$-map, $\varphi_{\sigma_k} : C_{s \sigma_k} \rightarrow C_{t \sigma_k},$ satisfying the relation (\ref{diadgn}) and
\begin{enumerate}[label = (\roman*)]
\item $\varphi_{(*_0, *_1)} = id_{q(*_0)}$ (when $k=1$),
\item $\varphi_{(\cdots, \widehat{*_0}, *_1)} = \varphi_{(\cdots, *_0)} = q(\cdots, *_0) = q(\cdots, \widehat{*_0}, *_1),$
\item $\varphi_{(\cdots, *_0, *_1)} = 0,$
\item $\varphi_{(\cdots, \widehat{*_0}, \widehat{*_1})} = F_{(\cdots)}.$
\end{enumerate}
These are easily checked if we just look at how the degeneracy map $\sigma : \mathcal{D}_0 \rightarrow \mathcal{D}_1$ acts. (It is induced from the (standard) degeneracy map on $s : \Delta_0 \rightarrow \Delta_1$.)

\end{enumerate}

\begin{remark}
From the assumption that $q \in \mathcal{D}_0,$ it follows that $\sigma_k$ satisfies the relation (\ref{diadgn}) for $\varphi_{\sigma_k}.$ Hence it suffices to check the relation (\ref{diadgn}) for the maps of the form $\varphi_{(\cdots, *_0, *_1)}$: 
\begin{equation}\nonumber
\partial \circ \varphi_{(\cdots, *_0, *_1)} + \varphi_{(\cdots, *_0, *_1)} \circ \partial + \sum^k_{i=1} \varphi_{\partial_i (\cdots, *_0, *_1)}+ \sum^k_{i=1} \varphi_{(...)^2_{k-i}} \circ \varphi_{(...)^1_i} = 0,
\end{equation}
which obviously holds.
\end{remark}

\subsubsection*{$k$-simplices of $\mathcal{D}$ with $k \geq 2$} Similarly, $k$-simplices of $\mathcal{D},$ with $k \geq 2$ that is,  the elements in $\mathcal{D}_k = Hom_{\mathscr{F}}\big(N(\mathcal{I})\star \Delta^k, N_{dg}\big(Ch(\mathbb{Z}_2)\big)\big),$ can be described as follows. We denote $\Delta^k = (*_0, \cdots, *_k).$
\begin{enumerate}[label = \textbullet]
\item To an $m$-simplex $\sigma_m \in (N(\mathcal{I}) \star \Delta^k)_m,$ with $m \geq 1,$ we assign a degree $m-1$ linear map $\varphi_{\sigma_m} : C_{s \sigma_m} \rightarrow C_{t \sigma_m},$ satisfying (\ref{diadgn}) for $ f_{\sigma_m}.$
In particular, if $\sigma_m$ contains no vertices from $\Delta^k,$ we assign ${F}_{\sigma_m}.$
\end{enumerate}

\subsubsection*{Contractibility of the left morphism space} Now we discuss the contractibility for the colimit. Fix $q \in \mathcal{D}_0$ and suppose that we are given  $X_0, \cdots, X_k \in Hom^L_{\mathcal{D}}(\overline{\mathscr{F}}, q)_{k-1}$ that coincide with each other at the boundaries. By this, we mean that $X_0, \cdots, X_k$ are elements in $\mathcal{D}_{k} := Hom_{\mathscr{F}}\big(N(\mathcal{I})\star \Delta^{k}, N_{dg}\big(Ch(\mathbb{Z}_2)\big)\big)$ given by
\begin{equation}\nonumber
X_i : N(\mathcal{I}) \star (*_{d_i(0)}, \cdots, *_{d_{i}(k)}) \rightarrow N_{dg}\big(Ch(\mathbb{Z}_2)\big), \ (i = 0, \cdots, k),
\end{equation}
where $d_i : \Delta^{k}  = (*_0, \cdots *_{k}) \hookrightarrow \Delta^{k+1} \ (i= 0, \cdots, k+1)$ denotes the face map between the standard simplices, so that they satisfy
\begin{equation}\label{intsec}
X_i|_{N(\mathcal{I}) \star (d_i(\Delta^{k}) \cap d_j(\Delta^{k}))} = X_j|_{N(\mathcal{I}) \star (d_i(\Delta^{k}) \cap d_j(\Delta^{k}))}, \text{ for all } i, j.
\end{equation}

\ \

We observe that $X_0, \cdots, X_k \in \mathcal{D}_{k}$ satisfy:
\begin{enumerate}[label=\textbullet]
\item $X_{i} |_{N(\mathcal{I})} \equiv \mathscr{F}$ (from the definition of $\mathcal{D}_k$),
\item $X_{i} |_{N(\mathcal{I}) \star \{*_0\}} \equiv \overline{\mathscr{F}}, \ X_{i} |_{N(\mathcal{I}) \star \{*_j\}} = q \in \mathcal{D}_0 \ (j \neq 0)$ (from the definition of $Hom^L_{\mathcal{D}}(\overline{\mathscr{F}}, q)$),
\item $X_{i} |_{*_0} = \widetilde{C},$ $X_i|_{*_j} = C' \ (j =1, \cdots, k)$ for a fixed chain complex $C'$ (from definition of $\mathcal{D}_0$),
\item $X_i|_{(*_j, *_k)} = id_{C'} $ for all $1 \leq j \leq k$ (from the definition of the identity edges),
\item $X_i|_{(*_{j_1}, \cdots, *_{j_l})} = 0$ for all $l \geq 3$ (from the definition of degenerate simplices).
\end{enumerate}

\ \

Contractibility of the left morphism space $Hom^L_{\mathcal{D}}(\overline{\mathscr{F}}, q)$ is now rephrased as follows.

\begin{claim}
Given the above $X_0, \cdots, X_k \in \mathcal{D}_{k},$ there exists $\widetilde{X} \in \mathcal{D}_{k+1}$ such that
\begin{equation}\nonumber
\widetilde{X}|_{N(\mathcal{I}) \star d_i(\Delta^{k})} = X_i, \ i=0, \cdots, k+1.
\end{equation}
\end{claim}

To determine a map
\begin{equation}\nonumber
\widetilde{X} : N(\mathcal{I}) \star \Delta^{k+1} \rightarrow N_{dg}\big(Ch(\mathbb{Z}_2)\big),
\end{equation}
we need to know which linear map of degree $m-1$ to assign to each $m$-simplex ($m \geq 1$) of $N(\mathcal{I}) \star \Delta^{k+1}.$ There are three cases:
\begin{enumerate}[label = (\roman*)] 
\item If $m < k+1,$ we assign the maps that are already given by the $X_i$ data. By construction these maps are of degree $m-1$ and satisfy (\ref{diadgn}).
For example, if $(*_{i_1}, \cdots *_{i_m}) \subset d_l(\Delta^k),$ for some $l,$ then we let 
\begin{equation}\label{degenassgn0}
(*_{i_1}, \cdots *_{i_m}) \mapsto X_{l}(*_{i_1}, \cdots *_{i_m}) =: \varphi_{(*_{i_1}, \cdots *_{i_m})}=0,
\end{equation} 
and this is independent of choice of $l$ by (\ref{intsec}).
\item If $m=k+1,$ we assign:
\begin{equation}\nonumber
\varphi_{\Delta^{k+1}} = \varphi_{(*_0, \cdots, *_{k+1})} : \widetilde{C} \rightarrow C',
\end{equation} 
given by
\begin{equation}\nonumber
\begin{split}
\varphi_{\Delta^{k+1}}(\sigma_l; x) &:= \sum^{k+1}_{i=0} \varphi_{(\sigma_l, d_i(\Delta^{k}))}(x),\\
&= \varphi_{(\sigma_l, \widehat{*_0}, {*_1}, \cdots, *_{k+1})}(x) + \cdots + \varphi_{(\sigma_l, *_0, \cdots, \widehat{*_{k+1}})}(x).
\end{split}
\end{equation}
where $\sigma_l \in N(\mathcal{I})_l$ with $l \geq 0, x\in C_{s \sigma_l},$ and $\Delta^k = (*_0, \cdots, *_k).$ We see that the degree of $\varphi_{\Delta^{k+1}}$ is equal to $|\varphi_{(\sigma_l, d_i(\Delta^{k}))}| - l= (l+ k+1)-1 -l =k =m-1.$ Also, $\varphi_{\Delta^k}$ can be checked to satisfy (\ref{diadgn}).

\item If $m > k+1$, we assign a map
\begin{equation}\nonumber
\varphi_{(a_1, \cdots, a_{m-k}; *_0, \cdots *_k)} : C_{a_1} \rightarrow C_{*_k},
\end{equation}
and all the other assignments can be set to be those provided by the data $\{X_0, \cdots, X_k\}.$ In this case, we let 
\begin{equation}\nonumber
\varphi_{(a_1, \cdots, a_{m-k}; *_0, \cdots *_k)}(x) = 0, \ \text{ for each } (a_1, \cdots, a_{m-k}),
\end{equation}
or by denoting $\vec{a} = (a_1, \cdots, a_{m-k}),$
\begin{equation}\nonumber
\varphi_{(\vec{a};\Delta^k)} =0,  \ \text{ for each } \vec{a}.
\end{equation} 
\end{enumerate}

\begin{lemma}
In the cases $(ii)$ and $(iii),$ we have :
\begin{equation}\nonumber
\begin{cases}
\partial \circ \varphi_{{\Delta}^{k+1}} + \varphi_{{\Delta}^{k+1}} \circ \partial + \sum\limits_{j=1}^k \varphi_{({\Delta}^{k+1})^2_{k+1-j}} \circ \varphi_{({\Delta}^{k+1})^1_j} + \sum\limits_{j=1}^k \varphi_{\partial_j {\Delta}^{k+1}} =0,\\
\partial \circ \varphi_{(\vec{a};\Delta^k)} + \varphi_{(\vec{a};\Delta^k)} \circ \partial + \sum\limits_{j=1}^m \varphi_{(\vec{a};\Delta^k)^2_{m+1-j}} \circ \varphi_{(\vec{a};\Delta^k)^1_j}  + \sum\limits_{j=1}^m \varphi_{\partial_j (\vec{a};\Delta^k)} =0.
\end{cases}
\end{equation}

\end{lemma}

\begin{proof}
\begin{enumerate}[label = (\roman*)]
\item  For $(\sigma_l; x) \in \widetilde{C},$ we denote
\begin{equation}\nonumber
\begin{split}
\textcircled{1} &= \partial \circ \varphi_{{\Delta}^{k+1}}(\sigma_l; x), \ \textcircled{2} = \varphi_{{\Delta}^{k+1}} \circ \partial (\sigma_l; x), \ \textcircled{3} = \sum\limits_{j=1}^k \varphi_{\partial_j {\Delta}^{k+1}}(\sigma_l; x),\\
\textcircled{4} &= \sum\limits_{j=1}^k \varphi_{({\Delta}^{k+1})^2_{k+1-j}} \circ \varphi_{({\Delta}^{k+1})^1_j}(\sigma_l; x). 
\end{split}
\end{equation}
First of all, we have
\begin{equation}\nonumber
\textcircled{1} = \partial \circ \sum\limits_{i=0}^{k+1} \varphi_{(\sigma_l; *_{0}, \cdots, \widehat{*_{i}}, \cdots,*_{k+1})}(x) = \sum\limits_{i = ^{k+1}} \textcircled{1}_i,
\end{equation}
where
\begin{equation}\nonumber
\begin{split}
& \textcircled{1}_i = \partial \circ \varphi_{ (\sigma_l; *_{0}, \cdots, \widehat{*_i}, \cdots,*_{k+1}) }(x)\\
&= \varphi_{(\sigma_l; *_{0}, \cdots, \widehat{*_i}, \cdots,*_{k+1})}( \partial x) 
+ \sum\limits_{j=1}^l \varphi_{(\partial_j \sigma_l; *_{0}, \cdots, \widehat{*_i}, \cdots,*_{k+1})}(x)\\
&+ \sum\limits_{j = i+1}^{i-1} \varphi_{(\sigma_l; *_{0}, \cdots, \widehat{*_j}, \cdots, \widehat{*_i}, \cdots,*_{k+1})}(x) + \sum\limits_{j = i +1}^{k} \varphi_{(\sigma_l; *_{0}, \cdots, \widehat{*_i}, \cdots, \widehat{*_j} \cdots,*_{k+1})}(x)\\
&+ \sum\limits_{j=1}^l \varphi_{((\sigma_l)^2_{l-j}; *_{0}, \cdots, \widehat{*_i}, \cdots,*_{k+1})} \circ \varphi_{(\sigma_l)^1_j}(x)\\
&+(1- \delta_{0i} ) \ \varphi_{(*_{0}, \cdots, \widehat{*_i}, \cdots,*_{k+1})} \circ \varphi_{(\sigma_l, *_0)}(x)\\
&+ \sum\limits_{j=1}^{i-1}\varphi_{(*_j, \cdots, \widehat{*_i}, \cdots, *_{k+1})} \circ \varphi_{(\sigma_l, *_0, \cdots, *_j)}(x) + \sum\limits_{j = i+1}^{k-1} \varphi_{(*_j, \cdots, *_{k+1})} \circ \varphi_{(\sigma_l, *_0, \cdots, \widehat{*_j}, \cdots, *_j)}(x).
\end{split}
\end{equation}

\begin{equation}\nonumber
\begin{split}
\textcircled{1} & = \sum\limits_{i=0}^{k+1} \varphi_{(\sigma_l ; *_0, \cdots, \widehat{*_i}, \cdots, *_{k+1})}(\partial x) + \sum\limits_{i=0}^{k+1} \sum\limits_{j=1}^{l} \varphi_{(\partial_j \sigma_l ; *_0, \cdots, \widehat{*_i}, \cdots, *_{k+1})}(x) \\
&+ \sum\limits_{i=0}^{k+1} \sum\limits_{j=1}^l \varphi_{((\sigma_l)^2_{l-j} ; *_0, \cdots, \widehat{*_i}, \cdots, *_{k+1})} \circ \varphi_{(\sigma_l)^1_j}(x) + \sum\limits_{i=1}^{k+1}\varphi_{(*_0, \cdots, \widehat{*_{i}}, \cdots, *_{k+1})} \circ \varphi_{(\sigma_l; *_0)}(x)
\end{split}
\end{equation}
(Here the last two terms vanish by \ref{degenassgn0}).

$\textcircled{2}, \textcircled{3},$ and $\textcircled{4}$ are
\begin{equation}\nonumber
\begin{split}
\textcircled{2} &= \varphi_{\Delta^{k+1}} \Big( \sum\limits_{j=1}^l (\partial_j \sigma_l; x) + \sum\limits_{j=1}^l \big((\sigma_l)^2_{l-j}; \varphi_{(\sigma_l)^1_j}(x)\big) + (\sigma_l; \partial x) \Big) \\
&= \sum\limits_{j=1}^l \sum\limits_{i=0}^{k+1} \varphi_{(\partial_j \sigma_l ;*_{0}, \cdots, \widehat{*_{i}}, \cdots,*_{k+1})}(x)+ \sum\limits_{j=1}^l \sum\limits_{i=0}^{k+1} \varphi_{((\sigma_l)^2_{l-j} ;*_{0}, \cdots, \widehat{*_{i}}, \cdots,*_{k+1})} \circ \varphi_{(\sigma_l)^1_j}(x)\\
&+ \sum\limits_{i=0}^{k+1} \varphi_{(\sigma_l ;*_{0}, \cdots, \widehat{*_{i}}, \cdots,*_{k+1})} (\partial x),
\end{split}
\end{equation}
\begin{equation}\nonumber
\textcircled{3} = \sum\limits_{j=1}^k \varphi_{(*_0, \cdots, \widehat{*_j}, \cdots, *_{k+1})}(\sigma_l; x) =  \sum\limits_{j=1}^k \varphi_{(*_0, \cdots, \widehat{*_j}, \cdots, *_{k+1})} \circ \varphi_{(\sigma_l; *_0)}(x), 
\end{equation}
and
\begin{equation}\nonumber
\begin{split}
\textcircled{4} &= \sum\limits_{j=1}^k \varphi_{(*_0, \cdots, \widehat{*_j}, \cdots, *_{k+1})}(\sigma_l; x) =  \sum\limits_{j=1}^k \varphi_{(*_j, \cdots, *_{k+1})} \circ \varphi_{(*_0, \cdots, *_{j})} (\sigma_l; x) \\
&= \varphi_{(*_k, *_{k+1})} \circ \varphi_{(*_0, \cdots, *_{k})} (\sigma_l; x) = \varphi_{(*_0, \cdots, *_{k})} \circ \varphi_{(\sigma_l; *_0)}(x).
\end{split}
\end{equation}
One can check that $\textcircled{1} + \textcircled{2} +\textcircled{3} +\textcircled{4} = 0$ immediately follows.

\item For $x \in C_{a_1},$ we check 
\begin{equation}\nonumber
\begin{split}
0 &= \partial \circ \varphi_{(a_1, \cdots, a_{m-k}, *_0, \cdots, *_k)}(x) + \varphi_{(a_1, \cdots, a_{m-k}, *_0, \cdots, *_k)}(\partial x)\\
&+ \sum\limits_{i=1}^{m-k} \varphi_{(a_1, \cdots, \widehat{a_i}, \cdots, a_{m-k}, *_0, \cdots, *_k)}(x) + \sum\limits_{i=2}^{m-k} \varphi_{(a_i, \cdots, a_{m-k}, *_0, \cdots, *_k)} \circ \varphi_{(a_1, \cdots, a_{i})} (x)  \\
&+ \sum\limits_{i=0}^{k-1} \varphi_{(a_1, \cdots, a_{m-k}, *_0, \cdots, \widehat{*_i}, \cdots, *_k)}(x)\\
&+ \varphi_{( *_0, \cdots, *_k)} \circ \varphi_{(a_1, \cdots, a_{m-k}, *_0)}(x) + \sum\limits_{i=1}^{k-1} \varphi_{(*_i, \cdots, *_k)} \circ \varphi_{(a_1, \cdots, a_{m-k}, *_0, \cdots, *_i)}(x).
\end{split}
\end{equation}
The terms in the first row vanish by the way we defined $\varphi_{(a_1, \cdots, a_{m-k}, *_0, \cdots, *_k)}$ and those in the second row vanish for the same reason. The terms in the third and forth rows cancel by the way we defined the $\varphi_{(*_i, \cdots, *_k)}$'s $(i \geq 0).$ 

\end{enumerate}

\end{proof}

Hence by the homotopy equivalence of the left and right morphism spaces, we have shown the following:
\begin{proposition}
The right morphism space $Hom^R_{\mathcal{D}}(\overline{\mathscr{F}}, q)$ is contractible for any $q : N(\mathcal{I})^{\triangleright} \rightarrow N_{dg}\big(Ch(\mathbb{Z}_2)\big),$ so that the extension $\overline{\mathscr{F}}$ is an initial object of $N_{dg}\big(Ch(\mathbb{Z}_2)\big)_{\mathscr{F}/}.$\end{proposition}

\section{Mapping telescope construction}

In this section, starting from the algebraic mapping telescopes, we study a family of some simplicial sets. From this, we show that the homology of the chain complex $\widetilde{C}$ is isomorphic to the direct limit of the family of homologies.

\subsection{The algebraic mapping telescope}
Let $\{C_i\}_i$ be a family of chain complexes (over $\mathbb{Z}_2$) together with a family of chain maps $\{\psi_i : C_i \rightarrow C_{i+1}\}_i.$ We consider the chain complex
\begin{equation}\nonumber
\widetilde{C}^T := \bigoplus^{\infty}_{i=0} (C_i \oplus \overline{C}_i),
\end{equation}
where $\overline{C}_i$ is another copy of the chain complex $C_i.$ To distinguish its elements from those of $C_i,$ let us use the notation : $C_i \ni x \leftrightarrow \overline{x} \in \overline{C}_i.$ The differential $\partial^T$ of $\widetilde{C}^T$ is given by:
\begin{equation}\nonumber
C_i \oplus \overline{C}_i \ni (y, \overline{x}) \mapsto \big(\partial y + x + \psi_i (x), \overline{\partial x}\big).
\end{equation}
It immediately follows that $\partial ^2 =0.$

\begin{figure}[h!]
\centering
\includegraphics[width=0.5\textwidth]{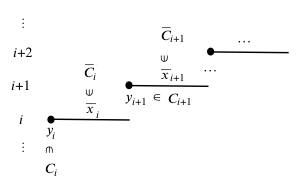}
\caption{Graphical description of the chain complex $\widetilde{C}^T$}
\end{figure}

Observe that $\widetilde{C}^T$ has a subspace $C^o := \bigoplus\limits_{i = 0}^{\infty} C_i$ which can be checked to form a chain complex with the restricted differential $\partial^o : = \partial|_{C^o}.$

\begin{lemma}\label{2lmt} We have
\begin{enumerate}[label = (\roman*)]
\item $\ker \partial^T = \ker \partial^{o} + \textup{im}\partial^T,$
\item $\frac{\ker \partial^o \cap \textup{im}\partial^T}{\textup{im} \partial^o} \simeq  \langle  x + {\psi_{i}}_*x \rangle _{x \in H_*(C_i), i \geq 0} .$
\end{enumerate}
\end{lemma}

\begin{proof}
\begin{enumerate}[label = (\roman*)]
\item $\ker \partial^T \supset \ker \partial^{o} + \text{im} \partial^T$ is clear. We show the opposite direction: $\ker \partial^T \subset \ker \partial^{o} + \text{im} \partial^T.$
Let $\sum\limits_{i \in I} \overline{x_i}+ \sum\limits_{j \in J} y_j$ be an element in $\ker \partial^T,$ where $I$ and $J$ are some finite subsets of $\mathbb{Z}_{\geq 0}$ with $\overline{x_i} \in \overline{C}_{i}$ and $y_j \in C_j$ for all $i$ and $j.$ We remark that if $\sum\limits_{i \in I} z_i = 0,$ then $z_i =0,$ for all $i.$ This is because for different $i$'s, the $z_i'$s lie in different $C_i'$s. Observe that $\partial^o x_i =0$ for all $i,$ which follows from the fact that the term $\overline{\partial^o x_{i}}$ that appears in $\partial^T \overline{x_{i}} = x_{i} + \psi_{i*} x_{i} +  \overline{\partial^o x_{i}}$ must vanish. 
\ \
Now suppose that $i'$ is the smallest element in $I.$ Then we observe that $\partial^o y_j = 0$ if $j < i'.$ Also, (up to addition of an element in $\ker \partial^o$) we have $\partial^o y_{i'} = x_{i'},$ so that the minimality of $i'$ forces the sum $\partial^o x_{i'} + y_{i'}$ in $\partial^T (\sum\limits_{i \in I} \overline{x_i}+ \sum\limits_{j \in J} y_j)$ to vanish. (Note that $\psi_{i*}$ sends elements in $C_i$ to those in $C_{i+1}.)$
\ \
Then $\sum\limits_{i \in I} \overline{x_i}+ \sum\limits_{j \in J} y_j$ can be rewritten as $\sum\limits_{i \in I \setminus {i'}} \overline{x_i}+ \sum\limits_{j \in J \setminus {i'}} y_j + x_{i'} + y_{i'} = \sum\limits_{i \in I \setminus {i'}} \overline{x_i}+ \sum\limits_{j \in J \setminus {i'}} y_j + \overline{\partial y_{i'}} + (\overline{y_{i'}} + \psi_{i'*}y_{i'} +  \partial^T \overline{y_{i'}}),$ and finally as 
\begin{equation}\nonumber
\sum\limits_{i \in I \setminus {i'}} \overline{x_i}+ \sum\limits_{j \in J \setminus {i'}} y'_j + \partial^T \overline{y_{i'}},
\end{equation}
where
\begin{equation}\nonumber
y'_j=
\begin{cases}
y_j + \psi_{i'*}y_{i'} & \text{ if } j = i' +1\\
y_j & \text{ if } j \neq i' +1.
\end{cases}
\end{equation}
Repeating this process we can remove all the terms of $\sum\limits_{i \in I} \overline{x_i}$ and end up with the sum of the form
$ \sum\limits_{j \in J} z_j + \partial^T( \cdots ), $ where $z_j \in \ker \partial^o$ for all $j,$ which finishes the proof of (i). 
\item Elements in $\ker \partial^o \cap \text{im} \ \partial^T$ are generated by those of the form $\partial^T{(\overline{x_i} + y_i)}$ with $\partial^o x_i =0.$ By sending each class $[\partial^T{(\overline{x_i} + y_i)}]=[x_i + \psi_{i*} x_i + \partial^o y_i]$ in $\frac{\ker \partial^o \cap \text{im} \partial^T}{\text{im} \partial^o}$ to $[x_i + \psi_{i*} x_i],$ we determine a well-defined map, which is clearly an isomorphism.

\end{enumerate}
\end{proof}

From this lemma, we have

\ \

\begin{proposition}
The family $\big\{H_*(C_i), \{\psi_{i*}\} \big\}_i$ forms a direct system and we have
\begin{equation}\nonumber
H_*(\widetilde{C}^T) \simeq \lim\limits_{\longrightarrow} H_*(C_i).
\end{equation}
\end{proposition}
\begin{proof}

Consider the subchain complex $(C^{o} := \bigoplus\limits_{i=0}^{\infty}C_i, \partial^o) \subset (\widetilde{C}^T, \partial^T).$

\begin{equation}\nonumber
\begin{split}
\frac{\ker \partial^T}{\text{im} \partial^T} &\simeq \frac{\ker \partial^o  + \text{im} \partial^T}{\text{im} \partial^T} \simeq \frac{\ker \partial^o  }{\text{ker} \partial^o   \cap \text{im} \partial^T},
\\
&\simeq \frac{\frac{\ker \partial^o  }{\text{im} \partial^o  }}{\frac{\text{ker} \partial^o   \cap \text{im} \partial^T }{\text{im} \partial^o  }}\simeq \frac{\ker \partial^o  }{\text{im} \partial^o}/ \langle \{ x + {\psi_{i}}_*x \}_{x \in H_*(C_i), i} \rangle,
\end{split}
\end{equation}
We use the second isomorphism theorem to obtain the second isomorphism. The third isomorphism follows from the third isomorphism theorem. Let $\sim$ denote the equivalence relation $x \sim y \Longleftrightarrow \text{ there exists } k > i, j \text{ such that } \psi_{k-1_*} \cdots \psi_{i*}(x) = \psi_{k-1_*} \cdots \psi_{j*} (y).$ It is not difficult to check that taking the quotient by elements in $ \{ x + {\psi_{i}}_*x \}_{x \in H_*(C_i), i}$ generates this equivalence relation, so that the resulting quotient module is isomorphic to $ \bigoplus\limits_{i=0}^{\infty} H_*(C_i) / \sim,$ which is isomorphic to the direct limit, $\lim\limits_{\longrightarrow}H_{*}(C_i).$

\end{proof}

\subsubsection*{The colimit diagram $\overline{\mathscr{F}}^T$} We consider the simplicial set $K^T$ that consists of:
\begin{enumerate}[label = \textbullet]
\item $K^T_0 := N(\mathcal{I})_0.$
\item $K^T_1 := \bigl\{f \in N(\mathcal{I})_1 \mid f \neq id \Rightarrow tf =sf +1 \bigr\}.$
\item $K^T_{k(\geq 2)}  := \big\{ (f_1, \cdots, f_k) \in N(\mathcal{I})_k \mid \#\{i \mid f_i \neq id\} \leq 1 \big\}.$ (I.e., its elements are of the form $(id, \cdots, id, f, id, \cdots, id)$ with $f \in K^T_1.$)
\item The face and degeneracy maps are induced from those of $N(\mathcal{I}).$
\end{enumerate}

Notice that the inclusion $\iota^T : K^T \hookrightarrow N(\mathcal{I})$ is a morphism of simplicial sets. Define 
\begin{equation}\nonumber
\mathscr{F}^T: K^T \rightarrow N_{dg}\big(Ch(\mathbb{Z}_2)\big)
\end{equation}
to be the diagram given by $\mathscr{F}^T := \mathscr{F} \circ \iota^T.$

Observe that we have:
\begin{enumerate}[label=(\roman*)]
\item $\mathscr{F}^T(a) = C_a,$
\item $\mathscr{F}^T(f_a) = \mathscr{F}(f_a),$ for $f_a \in K^T_1,$ (in particular, $\mathscr{F}^T(id_a) =id_{C_a}$),
\item $\mathscr{F}^T(\sigma_k) = 0$ for $\sigma_k \in K^T_k$ with $k \geq 2.$
\end{enumerate}

It is possible to extend $\mathscr{F}^T$ to 

\begin{equation}\nonumber
\overline{\mathscr{F}}^T : (K^T)^{\triangleright} \rightarrow N_{dg}\big(Ch(\mathbb{Z}_2)\big),
\end{equation}

so that $\overline{\mathscr{F}}^T(*) = \widetilde{C}^T$ as follows.
\begin{enumerate}[label= \textbullet]
\item If $\sigma_k \in K^T_k,$ $\overline{\mathscr{F}}^T(\sigma_k) := \mathscr{F}^T(\sigma_k).$
\item If $k=0,$ and $\sigma_0 \notin K^T_0,$ $\overline{\mathscr{F}}^T(*) := \widetilde{C}^T,$
\item If $k=1,$ and $\sigma_1 \notin K^T_1,$ $\overline{\mathscr{F}}^T(\sigma_1) : C_{s \sigma_1} \rightarrow C_{t \sigma_1} = C_* = \widetilde{C}^T,$ takes $x$ to $x,$
\item If $k=2,$ and $\sigma_2 \notin K^T_2,$ \\
\quad \quad \quad $ \overline{\mathscr{F}}^T(\sigma_2) : C_{s \sigma_2} \rightarrow C_{t \sigma_2} = C_* = \widetilde{C}^T,$ takes $x$ to $(f_{s \sigma_2}, x),$
\item If $\sigma_k \notin K^T_k$ and $k>2,$ $\overline{\mathscr{F}}^T(\sigma_k) := 0.$
\end{enumerate}

\begin{proposition}
$\overline{\mathscr{F}}^T$ is a colimit diagram.
\end{proposition}

\begin{proof}
See Corollary \ref{colimtbar}.
\end{proof}

\subsubsection*{A sequence of simplicial subsets of $N(\mathcal{I})$} For $a \geq 0,$ we denote by $K^a = \{K^a_k\}_{k \geq 0}$ the simplicial set that consists of:
\begin{enumerate}[label = \textbullet]
\item If $k \leq a, \ K^a_k := N(\mathcal{I})_k.$ 
\item If $k > a, \ K^a_k := \bigl\{ (f_1, \cdots, f_k) \in N(\mathcal{I})_k | \# \{j \mid f_j \neq id\} \leq a \bigr\}.$
\item The face and degeneracy maps are naturally induced from those of $N(\mathcal{I}).$
\end{enumerate}

\begin{figure}[h!]
\centering
\includegraphics[width=0.4\textwidth]{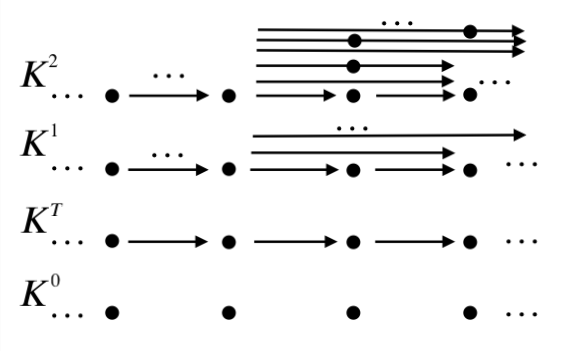}
\caption{Graphical description of the simplicial sets $K^{\bullet}$}
\end{figure}

We also denote the inclusions of simplicial sets by $\iota^a : K^a \hookrightarrow N(\mathcal{I}) \ (a = 0,1, \cdots).$
Notice that we have a sequence of inclusions of simplicial subsets of $N(\mathcal{I})$:
\begin{equation}\nonumber
K^0 \subset K^T \subset K^1 \subset K^2 \subset \cdots.
\end{equation}

They can be depicted as Figure 7, where dots, arrows, and arrows with dots mean vertices, 1-simplices, and higher simplices, respectively.

\subsection{Cofinal maps}

There is a functor (c.f. [Lur1] Definition 1.1.5.1),
\begin{equation}\nonumber
\mathfrak{C} : sSet \rightarrow Cat_{\Delta}
\end{equation}
from the category of simplicial sets to the category of simplicially enriched categories. We explain how $\mathfrak{C}$ is defined. First of all, the simplicial set $\Delta^I$ associated to a finite (totally) ordered set $I$ is mapped to $\mathfrak{C}(\Delta^I)$ which is a simplicially enriched category given by the following data:
\begin{enumerate}[label = \textbullet]
\item $Ob\big(\mathfrak{C}(\Delta^I)\big) = I$
\item $Mor_{\mathfrak{C}(\Delta^I)}(i,j)$ is a simplicial set given by
\begin{equation}\nonumber
Mor_{\mathfrak{C}(\Delta^I)}(i,j) = \begin{cases}
\emptyset & \text{ if } i >j,\\
N P_{ij} & \text{ if } i \leq j,
\end{cases}
\end{equation}
where $P_{ij}$ is the collection of subsets $J$ of $I$ that satisfy: (i) $i,j \in J,$ (ii) $k \in J \Longrightarrow i \leq k \leq j.$ Note that $P_{ij}$ is a partially ordered set whose order is given by inclusions of subsets. $NP_{ij}$ is the simplicial nerve of the this ordered set. For vertices and edges, one can refer to the following.
\begin{equation}\nonumber
\begin{cases}
\begin{split}
&(NP_{ij})_0 \ni J' = \{i_0 = i, i_1, \cdots, i_{l-1}, i_l = j\} \xLeftrightarrow{notation} P_{i,i_{1}}  \circ \cdots \circ P_{i_{l-1}, j}\\
&(NP_{ij})_1 \ni (J'' \subset J'), \text{ where } J'' \text{ is a subset of } J' \text{ with } i,j \in J''.\\
& \ \ \ \ \ \ \ \ \ \vdots
\end{split}
\end{cases}
\end{equation}

The face and degeneracy maps applied to a $k$-simplex $J^{(k+1)} \subset J^{(k)} \subset \cdots \subset J^{''} \subset J^{'}$ are given in the following patterns:
\begin{equation}\nonumber
\begin{split}
\cdots \subset J^{(i+1)} \subset J^{(i)} \subset J^{(i-1)} \subset \cdots &\mapsto \cdots \subset J^{(i+1)} \subset J^{(i-1)} \subset \cdots,\\
\cdots \subset J^{(i+1)} \subset J^{(i)} \subset J^{(i-1)} \subset \cdots &\mapsto \cdots \subset J^{(i+1)} \subset J^{(i)} \subset J^{(i)} \subset J^{(i-1)} \subset \cdots,
\end{split}
\end{equation}
respectively.

The composition of $Mor_{\mathfrak{C}(\Delta^I)}$ is given as follows:
\begin{equation}\nonumber
\begin{split}
Mor_{\mathfrak{C}(\Delta^I)}(i,j) \times Mor_{\mathfrak{C}(\Delta^I)}(j,k) &\rightarrow Mor_{\mathfrak{C}(\Delta^I)}(i,k)\\
\big((J_1^{i,j} \subset \cdots \subset J^{i,j}_{l}),(J^{j,k}_1 \subset \cdots \subset J^{j,k}_{l})\big) &\mapsto ( J_1^{i,j} \cup J^{j,k}_1 \subset \cdots \subset J^{i,j}_{l} \cup J^{j,k}_{l}) \text{ on } (NP_*)_l, \\
&\text{ for all }l \text{ and } (J_1^{*,*} \subset \cdots \subset J^{*,*}_{l}) \in (NP_{*,*})_l.
\end{split}
\end{equation}

\end{enumerate}

\begin{lemma}
When $i \leq j,$ we have $Mor_{\mathfrak{C}(\Delta^I)}(i,j) \simeq [0,1]^{\#\{k \in I \mid i < k < j\}}.$
\end{lemma}
\begin{proof}
This follows from a simple consideration for the structure of $NP_*.$
\end{proof}

One can extend this definition of $\mathfrak{C}(\cdot)$ to arbitrary simplicial sets $K.$ For our later use, however, it is enough to consider when $K = N(\mathcal{I})$ and $K^T.$ Observe that $\mathfrak{C}\big(N(\mathcal{I})\big)$ should be the same as the above $\mathfrak{C}(\Delta^I)$ except that the finite ordered set $I$ is replaced with $N(\mathcal{I})_0.$

We briefly explain how the functor $\mathfrak{C}$ acts on the simplicial set $K^T.$ Since $\mathfrak{C}$ is left adjoint to the nerve functor $N,$ (see Remark \ref{functorc}), it preserves colimits. We consider the following diagram in the category $sSet$:

\begin{equation}\nonumber
\mathcal{S}:= \coprod\limits_{i \in \mathbb{Z}_{\geq 0}} \Delta^0 \rightrightarrows \coprod\limits_{i \in \mathbb{Z}_{\geq 0}} \Delta^1,
\end{equation}
where the two maps are the inclusion of the vertices in the edges as ending points. Then the corresponding coequalizer is the same as the simplicial set $K^T,$ i.e., $colim \mathcal{S} = K^T.$
On the other hand, we have 
\begin{equation}\nonumber
\mathfrak{C}(S) = \coprod\limits_{i \in \mathbb{Z}_{\geq 0}} \mathfrak{C}({\Delta^0}) \rightrightarrows \coprod\limits_{i \in \mathbb{Z}_{\geq 0}} \mathfrak{C}(\Delta^1).  
\end{equation}
Since we have from the description of  $\mathfrak{C}(\Delta^I)$

\begin{figure}[h]
\centering
\includegraphics[width=0.4\textwidth]{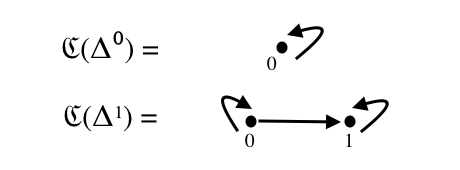}
\end{figure}

We can conclude that colim$\mathfrak{C}(\mathcal{S})$ is the category that looks like 

\ \

\ \

\begin{figure}[h]
\centering
\includegraphics[width=0.4\textwidth]{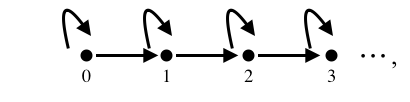}
\end{figure} as we indicate below in (ii).

\begin{remark}\label{functorc}
The functor $\mathfrak{C}$ is in fact left adjoint to the \textit{simplicial} nerve functor $N.$ In fact, in [Lur1] (c.f. Definition 1.1.5.5.), the functor $N$ is {\it{defined}} as such a right adjoint.
\end{remark}

\subsubsection*{} According to the above definition, we have the following description of $\mathfrak{C}(K^T)$ and $\mathfrak{C}(K^a)$ with $a \geq 1.$

\begin{enumerate}[label = (\roman*)]

\item $\mathfrak{C}\big(N(\mathcal{I})\big)$ is a simplicially enriched category which consists of the following data:

\begin{enumerate}[label = \textbullet]
\item $Ob\big(\mathfrak{C}\big(N(\mathcal{I}))\big) = N(\mathcal{I})_0,$
\item $Mor_{\mathfrak{C}(N(\mathcal{I}))}(i,j)$ is a simplicial set such that 
\begin{equation}\nonumber
\begin{split}
-&\text{ If } i \leq j, Mor_{\mathfrak{C}(N(\mathcal{I}))}(i,j) \simeq [0,1]^{\#\{k \mid i < k < j\}}\\
-&\text{ If } i > j, Mor_{\mathfrak{C}(N(\mathcal{I}))}(i,j) = \emptyset.
\end{split}
\end{equation}
\end{enumerate}

\item $\mathfrak{C}(K^T)$ is a simplicially enriched category which consists of the following data:
\begin{enumerate}[label = \textbullet]
\item $Ob\big(\mathfrak{C}(K^T)\big) =K^T_0 = N(\mathcal{I})_0,$
\item If $i \leq j, \ Mor_{\mathfrak{C}(K^T)}(i,j)$ is a simplicial set such that
\begin{equation}\nonumber
\begin{split}
- \big(Mor_{\mathfrak{C}(K^T)}(i,j)\big)_0 := \text{ the one-p}& \text{oint set, } \{f_{i, i+1} \circ \cdots \circ f_{j-1, j}\},\\
- \big(Mor_{\mathfrak{C}(K^T)}(i,j)\big)_k := \text{ degenerat}&\text{e } k \text{ simplices of the form of the point } \\ 
& \{f_{i, i+1} \circ \cdots \circ f_{j-1, j}\} \ k \geq 1.\\
\end{split}
\end{equation}
\item If $i >j, \ Mor_{\mathfrak{C}(K^T)}(i,j) = \emptyset.$ 
\end{enumerate}

\item $\mathfrak{C}(K^a)$ is a simplicially enriched category which consists of the following data:

\begin{enumerate}[label = \textbullet]
\item $Ob\big(\mathfrak{C}(K^a)\big) =K^a_0 = N(\mathcal{I})_0,$
\item If $a \geq 1, \ Mor_{\mathfrak{C}(K^a)}(i,j)$ is a simplicial set such that 
\begin{equation}\nonumber
\begin{split}
-&\text{ If } i \leq j, Mor_{\mathfrak{C}(K^a)}(i,j) \subset [0,1]^{\#\{k \mid i < k < j\}}\\
-&\text{ If } i > j, Mor_{\mathfrak{C}(K^a)}(i,j) = \emptyset.
\end{split}
\end{equation}
\end{enumerate}
\end{enumerate}

\ \

\subsubsection*{Homotopy categories.} Since the category of simplicial sets $sSet$ and the category of compactly generated weak Hausdorff spaces $\mathfrak{C}\mathfrak{G}$ are Quillen equivalent, their homotopy categories are equivalent. We write $\mathcal{H}$ for either of the homotopy categories and call it the \textit{homotopy category of spaces}. 

We consider a functor from the category of simplicially enriched categories to that of $\mathcal{H}$-enriched categories:
\begin{equation}\nonumber
\begin{split}
h : Cat_{\Delta} &\rightarrow Cat_{\mathcal{H}},\\
\mathcal{C} &\mapsto h \mathcal{C},
\end{split}
\end{equation}
where $h \mathcal{C}$ is the category enriched over $\mathcal{H}$ that is given by:
\begin{enumerate}[label = \textbullet]
\item $Ob(h \mathcal{C}) = Ob(\mathcal{C})$
\item $Mor_{h\mathcal{C}}(a,b) = [Mor_{\mathcal{C}}(a,b)].$
\end{enumerate}
Here $[\cdot]$ denotes the set of weak homotopy equivalence classes. Recall that two topological spaces are said to be \textit{weakly homotopy equivalent} if there is a map between them that induces isomorphisms on $\pi_i$ for all $i =0,1, \cdots.$ (If $i=0,$ it is a set isomorphism, and if $i \geq 1,$ it is a group isomorphism.)

We describe $h \mathfrak{C}(K^T)$ and $h \mathfrak{C}(K^i)$ for $i \geq 1:$

\begin{enumerate}[label = (\roman*)]

\item $h \mathfrak{C}(N(\mathcal{I}))$ is the $\mathcal{H}$-enriched category which consists of the following data:

\begin{enumerate}[label = \textbullet]
\item $Ob\big(h\mathfrak{C}\big(N(\mathcal{I})\big)\big) = Ob\big(\mathfrak{C}(K^a)\big),$
\item $Mor_{h \mathfrak{C}(N(\mathcal{I}))}(i,j)$ is given by
\begin{equation}\nonumber
\begin{split}
-\text{If } i \leq j, \ Mor_{h \mathfrak{C}(N(\mathcal{I}))}(i,j) &= [Mor_{\mathfrak{C}(N(\mathcal{I}))}(i,j)]\\ & \simeq \big[[0,1]^{\#\{k \mid i < k < j\}} \big] \simeq [pt],\\
-\text{If }i > j, \ Mor_{h \mathfrak{C}(N(\mathcal{I}))}(i,j) &= \emptyset.
\end{split}
\end{equation}

\end{enumerate}

\item $h \mathfrak{C}(K^T)$ is an $\mathcal{H}$-enriched category which consists of the following data:
\begin{enumerate}[label = \textbullet]
\item $Ob\big(h\mathfrak{C}(K^T)\big) =K^T_0,$
\item $Mor_{h\mathfrak{C}(K^T)}(i,j)$ is given by
\begin{equation}\nonumber
\begin{split}
-&\text{If } i \leq j, \ Mor_{h \mathfrak{C}(K^T)}(i,j) = [Mor_{\mathfrak{C}(K^T)}(i,j)] \simeq [pt], \\
-&\text{If }i > j, \ Mor_{h \mathfrak{C}(K^T)}(i,j) = \emptyset.
\end{split}
\end{equation}
\end{enumerate}

\end{enumerate}

\begin{definition}\label{catequiv}
We say that an edge in an $\infty$-category $\mathcal{C}$ is an {\it{equivalence}} if it induces an isomorphism in the homotopy category.
\end{definition}

\subsubsection*{Cofinality of morphisms} To obtain equivalences among colimits, we need the notion of cofinal morphisms of simplicial sets. We briefly introduce it.

\begin{definition}\label{coflur}
Let $f: X \rightarrow Y$ be a morphism of simplicial sets. 
\begin{enumerate}[label =(\roman*)]
\item $f$ is a \textit{right fibration} if it has the right lifting property with respect to all the inclusions $\Lambda^n_i \hookrightarrow \Delta^n, \ 0 < i \leq n.$
\item $f$ is \textit{cofinal} if for any right fibration $X \rightarrow S,$ the induced map 
\begin{equation}\nonumber
Map(S, X) \rightarrow Map(S, Y)
\end{equation}
is a homotopy equivalence.
\end{enumerate}
\end{definition}

\begin{proposition}([Lur1] Proposition 4.1.1.8)\label{cofprop}
Let $K$ and $K'$ be small simplicial sets. Then for an $\infty$-category $\mathcal{C},$ the following are all equivalent.
\begin{enumerate}[label = (\roman*)]
\item $f : K' \rightarrow K$ is cofinal
\item For a diagram $p : K \rightarrow \mathcal{C},$ the induced map $\mathcal{C}_{p /} \rightarrow \mathcal{C}_{p \circ f /}$ is an equivalence of $\infty$-categories.
\item Let $\overline{p}$ be a colimit of $p : K' \rightarrow \mathcal{C},$ then $\overline{p} \circ \overline{f}$ is a colimit of $p \circ f : K' \rightarrow \mathcal{C},$ where $\overline{f} : K'^{\triangleright} \rightarrow (K)^{\triangleright}$ is the naturally induced map on the corresponding cones.
\end{enumerate}
\end{proposition}

\begin{corollary}([Lur1] Corollary 4.1.1.9)
A categorical equivalence is cofinal.
\end{corollary}

\begin{corollary}
For a cofinal map $f,$ colimits $\overline{p},$ $\overline{p} \circ \overline{f}$ in Proposition \ref{cofprop} are all isomorphic in the homotopy category.
\end{corollary}
\begin{proof}
This follows immediately from the above proposition.
\end{proof}

A consequence of Proposition \ref{cofprop} is (c.f. Section 4.1.3 of [Lur1]) that a cofinal map $f$ induces an equivalence of colimits
\begin{equation}\nonumber
colim(p) \simeq colim(p \circ f).
\end{equation}
That is to say, there is a map between them in the $\infty$-category $\mathcal{C}$ which induces an isomorphism in the homotopy category $h\mathcal{C}.$ 

\subsection{A colimit of $\mathscr{F}$} Returning to our example, we consider the inclusions of simplicial sets $\iota^T:K^T \rightarrow N(\mathcal{\mathcal{I}})$ and $\iota^a : K^a \hookrightarrow N(\mathcal{I}) \ (a \geq 1).$

\begin{proposition}\label{cofit}
The inclusion of simplicial sets $\iota^T:K^T \rightarrow N(\mathcal{\mathcal{I}})$ is a cofinal map.
\end{proposition}
\begin{proof}
The functor $\iota^T : K^T \hookrightarrow N(\mathcal{I})$ induces a functor $\iota^T : \mathfrak{C}(K') \rightarrow \mathfrak{C}\big(N(\mathcal{I})\big)$ which is the identity map on objects, and the inclusion $Mor_{\mathfrak{C}(K^T)}(i,j) \simeq \{pt\} \hookrightarrow [0,1]^{\#\{k \mid i < k < j\}} \simeq Mor_{\mathfrak{C}\big(N(\mathcal{I})\big)}(i,j)$ with $i \leq j$ whose image coincides with the last vertex of the cube, that is $(1,1, \cdots, 1).$ Since (the geometric realization of) the cube is contractible, we obtain the identity functor on the homotopy categories.
\end{proof}

Consider the compositions of morphisms of simplicial sets, $\mathscr{F}^T:= \mathscr{F} \circ \iota^T$ and $\mathscr{F}^a := \mathscr{F} \circ \iota^a$ with $i \geq 1$ which are again diagrams. Since $\iota^T$ and $\iota^a \ (a \geq 1),$ are cofinal, we have
\begin{corollary}\label{colimtbar}
Denote by $\overline{\iota}^{\bullet}$ the natural extension of the inclusion $\iota^{\bullet}.$ Then $\overline{\mathscr{F}}^T := \overline{\mathscr{F}} \circ \overline{\iota}^T$ and $\overline{\mathscr{F}}^a: = \overline{\mathscr{F}} \circ \overline{\iota}^a.$ are colimit diagrams of $\overline{\mathscr{F}}^T$ and $\overline{\mathscr{F}}^a,$ respectively.
\end{corollary}
\begin{proof}
It follows from (iii) of Proposition \ref{cofprop}.
\end{proof}

We consider 
\begin{equation}\nonumber
\widetilde{C}^{N} := \overline{\mathscr{F}}^N(*),
\end{equation}
where we denote $\overline{\mathscr{F}}^N := \overline{\mathscr{F}} \circ \overline{\iota}^N$ with $N \geq 1.$ Then we have:
\begin{proposition}
\begin{enumerate}[label = (\roman*)]
\item $\widetilde{C}^N$ is written as 
\begin{equation}\nonumber
\widetilde{C}^N := \bigoplus\limits_{k \leq N} \bigoplus\limits_{\substack{\sigma_k \in N(\mathcal{I})_k,\\ \textup{nondeg.}}} \mathbb{Z}_2 \langle \sigma_k \rangle \otimes C_{s \sigma_k}.
\end{equation}
\item $\widetilde{C}$ is filtered by the length of simplices in $N(\mathcal{I})$ as follows.
\begin{equation}\nonumber
\widetilde{C}^0 \subset \widetilde{C}^1 \subset \cdots \subset \widetilde{C}.
\end{equation}
\end{enumerate}
\end{proposition}
\begin{proof}
(i) immediately follows from the way we defined $\iota^N.$ For (ii), notice that for each $N,$ $\widetilde{C}^N$ is a sub-chain complex of $\widetilde{C},$ since the differential on $\widetilde{C}$ does not increase the length.
\end{proof}

Observe from Proposition \ref{cofit} and Corollary \ref{colimtbar} that we have weak equivalences
\begin{equation}\nonumber
colim \mathscr{F} \simeq colim \mathscr{F}^T,
\end{equation}
which implies the existence of a quasi-isomorphism of the chain complexes
\begin{equation}\nonumber
\overline {\mathscr{F}}(*) \simeq \overline{\mathscr{F}}^T(*).
\end{equation}
In other words, these two are equivalent in the {\it{homotopy category}} $hN_{dg}\big({Ch(\mathbb{Z}_2)}\big).$ Since $Ch(\mathbb{Z}_2)$ is an abelian category, we can consult the following

\begin{proposition}([Lur2]) 
$hN_{dg}\big({Ch(\mathbb{Z}_2)}\big)$ is equivalent to $hCh(\mathbb{Z}_2).$
\end{proposition}

Hence by Proposition \ref{cofit}, we obtain

\begin{corollary}\label{lastcor}
We have
\begin{equation}\nonumber
H_*(\widetilde{C}) \simeq H_*(\widetilde{C}^T) \simeq \lim\limits_{\longrightarrow} H_*(C_i).
\end{equation}
\end{corollary}

\bibliographystyle{plain}
\bibliography{bibfile}

\medskip
 

\begin{thebibliography}{9}

\bibitem[Abo]{abouzaid}
Mohammed Abouzaid, \emph{Symplectic cohomology and Viterbo's theorem,} 2013.

\bibitem[AD]{MTFH}
Michèle Audin and Mihai Damian, \emph{Morse theory and Floer homology,} Universitext, Springer-Verlag, 2013.

\bibitem[AS]{abouzaidseidel}
Mohammed Abouzaid and Paul Seidel \emph{An open string analogue of Viterbo functoriality,} Geom. Topol. 14 (2010) 627-718.

\bibitem[BV]{BoardmanVogt:HIASTS}
John M. Boardman and Rainer M. Vogt, \emph{Homotopy invariant algebraic structures on
  topological spaces}, Lecture Notes in Math., vol. 347, Springer, 1973.

\bibitem[CF]{CF}
Cieliebak and Frauenfelder \emph{Morse homology on noncompact manifolds,} Journal of the Korean Mathematical Society 48(4), 2009 .

\bibitem[CP]{CordierPorter:VTOCOHCD}
Jean-Marc Cordier and Timothy Porter,
\textit{Vogt's theorem on categories of homotopy coherent diagrams}, Mathematical Proceedings of the Cambridge Philosophical Society, 100(1), 65-90.

\bibitem[HLS]{HendricksLipshitzSarkar:FCEFHA}
Kristen Hendricks, Robert Lipshitz, and Sucharit Sarkar, \emph{A flexible construction of equivariant Floer homology and applications},
Journal of Topology, Volume 9, Issue 4, 1 December 2016, Pages 1153–1236.
  
\bibitem[Kan]{Kang:IPMHNM}
Jungsoo Kang, \emph{Invariance property of Morse homology on noncompact manifolds,} preprint arxiv:1012.5571, 2010.

\bibitem[Kim]{KIM:OHDSOCC}
Taesu Kim, \emph{On homotopy direct systems of chain complexes,} PhD thesis, Seoul National University, 2018.

\bibitem[Lur1]{Lurie:HTT}
Jacob Lurie, \emph{Higher Topos Theory}, Annals of Mathematics Studies 170, Princeton University Press, 2009.

\bibitem[Lur2]{Lurie:HA}
Jacob Lurie, \emph{Higher Algebras}, http://www.math.harvard.edu/lurie/papers/HA.pdF, 2017.

\bibitem[Rie]{Riehl:CHT}
Emily Riehl, \emph{Categorical Homotopy Theory,} New Mathematical Monographs, 24, Cambridge University Press, 2014.  

\bibitem[Sch]{Schwarz:BVROMMC}
Matthias Schwarz, \emph{Morse homologies,} Progress in Mathematics, 111, Birkhäuser Basel, 1993.

\bibitem[Vog]{Vogt:HLC} 
Rainer M. Vogt \emph{Homotopy limits and colimits}, Math. Z. 134, 11-52, 1973.


\end{thebibliography}
\end{document}